\documentclass{article}
\pdfoutput=1
\usepackage[totalwidth=17cm,totalheight=24cm]{geometry}

\usepackage[affil-it]{authblk}
\usepackage{setspace}

\usepackage{subfig}

\usepackage[T1]{fontenc}
\usepackage[dvips]{graphicx}
\usepackage{epsfig}
\usepackage{amsmath,amssymb,amscd,amsthm}
\usepackage[latin1]{inputenc}
\usepackage{latexsym}
\usepackage{graphics}
\usepackage{colortbl}
\usepackage{color}
\usepackage{ae}

\usepackage{enumitem}

\def\H{{\mathbb H}}

\usepackage{textcomp}
\begin{document}

\setcounter{secnumdepth}{3}
\setcounter{tocdepth}{2}

\newtheorem{definition}{Definition}[section]
\newtheorem{lemma}[definition]{Lemma}
\newtheorem{sublemma}[definition]{Sublemma}
\newtheorem{corollary}[definition]{Corollary}
\newtheorem{proposition}[definition]{Proposition}
\newtheorem{theorem}[definition]{Theorem}

\newtheorem{remark}[definition]{Remark}
\newtheorem{example}[definition]{Example}

\newcommand{\mf}{\mathfrak}
\newcommand{\mb}{\mathbb}
\newcommand{\ol}{\overline}
\newcommand{\la}{\langle}
\newcommand{\ra}{\rangle}
\newcommand{\hess}{\mathrm{Hess}}
\newcommand{\grad}{\mathrm{grad}}
\newcommand{\II}{\textsc{I\hspace{-0.05 cm}I}}

\renewcommand{\O}{\mathcal{O}}

\newtheorem{thmprime}{Theorem}
\renewcommand{\thethmprime}{1.\arabic{thmprime}\textquotesingle}


\newtheorem{Alphatheorem}{Theorem}
\renewcommand{\theAlphatheorem}{\Alph{Alphatheorem}}

\newtheorem{Alphatheoremprime}{Theorem}
\renewcommand{\theAlphatheoremprime}{\Alph{Alphatheoremprime}\textquotesingle}

\theoremstyle{remark}

\newtheorem{exB}{Balls}
\newtheorem{excone}{Cones}
\newtheorem{exF}{Example E}
\newtheorem{exFuch}{Fuchsian}
\newtheorem{expast}{quasi-Fuchsian domain}
\newtheorem{exGamma}{quasi-Fuchsian convex set}
\newtheorem{familly}{A familly of examples}


\theoremstyle{plain}
\newtheorem*{flp}{Flip algorithm}
\newtheorem{clm}{Claim}

\newcommand{\EM}{\ensuremath}
\newcommand{\norm}[1]{\EM{\left\| #1 \right\|}}

\newcommand{\modul}[1]{\left| #1\right|}

\def\co{\colon\thinspace}
\def\I{{\mathcal I}}
\def\N{{\mathbb N}}
\def\R{{\mathbb R}}
\def\Z{{\mathbb Z}}
\def\Sph{{\mathbb S}}
\def\Tor{{\mathbb T}}
\def\Disk{{\mathbb D}}
\def\Fl{{\mathbb F}}

\def\H{{\mathbb H}}
\def\RP{{\mathbb R}{\mathrm{P}}}
\def\dS{{\mathrm d}{\mathbb{S}}}

\def\phi{\varphi}
\def\epsilon{\varepsilon}
\def\V{{\mathcal V}}
\def\E{{\mathbb E}}
\def\F{{\mathcal F}}
\def\C{{\mathcal C}}
\def\K{{\mathcal K}}

\renewcommand{\span}{\operatorname{span}}
\newcommand{\ang}{\operatorname{ang}}
\newcommand{\area}{\operatorname{area}}
\newcommand{\cone}{\operatorname{cone}}
\newcommand{\pr}{\operatorname{pr}}
\newcommand{\vol}{\operatorname{vol}}
\newcommand{\covol}{\operatorname{covol}}
\newcommand{\st}{\operatorname{st}}
\newcommand{\ost}{\operatorname{st}^\circ}
\newcommand{\inn}{\operatorname{int}}
\renewcommand{\d}{\operatorname{d}}

\newcommand{\bu}{{\bar u}}
\newcommand{\bv}{{\bar v}}
\newcommand{\bm}{{\bar{M}}}
\newcommand{\tr}{\operatorname{tr}}
\newcommand{\sff}{|\mathbf{II}|}
\newcommand{\sfa}[1]{|\mathbf{II}|^{#1}}
\newcommand{\kp}[1]{k_p\left(#1\right)}
\newcommand{\Kp}[1]{\mathcal{K}_p\left(#1\right)}\newcommand{\nor}[1]{\left\| #1\right\|}
\newcommand{\cmed}[1]{\frac{1}{m}\sum_{ #1=1}^{m}k_{#1}}
\newcommand{\nphiq}{\left|\phi\right|^2}
\newcommand{\nphin}[1]{\left|\phi\right|^{#1}}
\newcommand{\nphi}{\left|\phi\right|}
\newcommand{\trphi}{\operatorname{tr}\phi}
\newcommand{\ricc}{\operatorname{Ric}}
\newcommand{\mricc}{{}^M \operatorname{Ric}}
\newcommand{\nricc}{{}^N \operatorname{Ric}}
\newcommand{\tnn}{\tilde N_{\times/{}}}
\newcommand{\riem}{\operatorname{Riem}}
\newcommand{\mriem}{{}^M \operatorname{Riem}}
\newcommand{\nriem}{{}^N \operatorname{Riem}}
\newcommand{\rr}{\mathbb{R}}
\newcommand{\hh}{\mathbb{H}}
\newcommand{\sss}{\mathbb{S}}
\newcommand{\ttt}{\mathbb{T}}
\newcommand{\Hess}{\operatorname{Hess}}
\newcommand{\sect}{\operatorname{Sect}}
\newcommand{\msect}{{}^M \operatorname{Sect}}
\newcommand{\nsect}{{}^N \operatorname{Sect}}
\newcommand{\supp}{\operatorname{supp}}
\newcommand{\snk}{\operatorname{sn}_k}
\newcommand{\cnk}{\operatorname{cn}_k}
\newcommand{\ink}{\operatorname{in}_k}
\newcommand{\tnk}{\operatorname{tn}_k}
\newcommand{\dr}{\nabla r}
\newcommand{\dxi}{\nabla \xi}
\newcommand{\du}{\nabla u}
\newcommand{\aaa}{\frac{\left(\sigma-1\right)\mu}{\mu-1}}
\newcommand{\bbb}{\frac{\mu-1}{\left(\sigma-1\right)\mu}}
\newcommand{\inte}[2]{\int_{#1}#2 d\textit{v}_M}
\newcommand{\intebar}[2]{\int_{#1}#2 d\textit{v}_{\bar{M}}}
\newcommand{\intex}[2]{\int_{#1}#2 d\tilde{\textit{v}}_M}
\newcommand{\intem}[1]{\int_{M}#1 d\textit{v}_M}
\newcommand{\intsx}[1]{\int_{\supp\left(\xi\right)}#1 d\textit{v}_M}
\newcommand{\bux}{\mathcal{B}\left(u,\xi\right)}
\newcommand{\tu}{\tilde{u}}
\newcommand{\tir}{\tilde{r}}
\newcommand{\tq}{\tilde{q}}
\newcommand{\tv}{\tilde{v}}
\newcommand{\tb}[1]{\tilde{B}_{Q_0}\left(#1\right)}
\newcommand{\td}{\tilde{\nabla}}
\newcommand{\tn}{\tilde{N}}
\newcommand{\tm}{\tilde{M}}
\newcommand{\inteb}[1]{\int_{\tb{1}}#1 d\tilde{\textit{v}}_M}
\newcommand{\ve}{\varphi_{\varepsilon}}
\newcommand{\vel}{\ve^{\left(l\right)}}
\newcommand{\dive}{\operatorname{div}}
\newcommand{\conv}{\operatorname{conv}}
\newcommand{\hc}{\hat{c}}
\newcommand{\mmetr}[2]{\left\langle #1,#2\right\rangle_M}
\newcommand{\bmmetr}[2]{\left\langle #1,#2\right\rangle_\bm}
\newcommand{\nmetr}[2]{\left\langle #1,#2\right\rangle_N}
\newcommand{\hsmetr}[2]{\left\langle #1,#2\right\rangle_{HS}}
\newcommand{\Capp}[1]{\operatorname{Cap}_{#1}}
\newcommand{\abs}[1]{\left|#1\right|}
\newcommand{\ps}[2]{\left\langle#1,#2\right\rangle}
\newcommand{\ton}[1]{\left(#1\right)}
\newcommand{\nr}[2]{{}^N r(#1,#2)}
\newcommand{\dist}{\operatorname{dist}}
\newcommand{\tphi}{{\tilde\phi}}
\newcommand{\tx}{{\tilde x}}
\newcommand{\tX}{{\tilde X}}
\newcommand{\tih}{{\tilde h}}
\newcommand{\hG}{{\hh^d/\Gamma}}

\def\dist{\mathrm{dist\,}}
\def\diam{\mathrm{diam\,}}

\def\M{{\mathcal M}}
\def\Mt{{\mathcal M}_{\mathrm{tr}}}
\def\T{{\mathcal T}}
\def\Vt{V_{\mathrm{tr}}}

\setlength{\abovedisplayshortskip}{1pt}
\setlength{\belowdisplayshortskip}{3pt}
\setlength{\abovedisplayskip}{3pt}
\setlength{\belowdisplayskip}{3pt}


\title{Lorentzian area measures and the Christoffel problem}

 \author{Fran\c{c}ois Fillastre and Giona Veronelli}

\date{\today}

\maketitle

Universit\'e de Cergy-Pontoise, UMR CNRS 8088, F-95000 Cergy-Pontoise, France.

Universite Paris 13, Sorbonne Paris Cit\'e, LAGA, CNRS (UMR 7539), F-93430 Villetaneuse, France. 

 francois.fillastre@u-cergy.fr, veronelli@math.univ-paris13.fr

\begin{abstract}
 We introduce a particular class of unbounded closed convex sets
of $\R^{d+1}$, called F-convex sets (F stands for future). To define them, we use   the Minkowski bilinear form of 
signature $(+,\ldots,+,-)$ instead of the usual scalar product, and we ask the Gauss map to be a surjection 
onto the hyperbolic space
$\H^d$. 
Important examples are embeddings of the universal cover 
of some globally hyperbolic maximal flat Lorentzian manifolds.

Basic tools are first derived, similarly to the classical study of convex bodies.
For example, F-convex sets are determined by their  support function, which is defined on $\H^d$.
Then the area measures of order $i$, $0\leq i\leq d$ are defined. As in 
the convex bodies case, they are the coefficients of the polynomial in $\epsilon$ which is the volume of
an $\epsilon$ approximation of the convex set. Here the area measures are defined with respect to the
 Lorentzian structure.

Then we focus on the area measure of order one. Finding necessary and sufficient conditions for a measure 
(here on $\H^d$) to be the first area measure of an F-convex set is the Christoffel Problem.
We derive many results about this problem. If we restrict to F-convex set
setwise invariant under linear isometries acting cocompactly on $\H^d$, then
the problem is totally solved, analogously to the case of convex bodies. In this case the measure can be
given on a compact hyperbolic manifold.

Particular attention is given on the smooth and polyhedral cases. In those cases, the Christoffel problem is equivalent to
prescribing the mean radius of curvature
and the edge lengths respectively.\end{abstract}



\setcounter{tocdepth}{3}
\tableofcontents

\section{Introduction}

\subsection{Area measures and the Christoffel problem for convex bodies}
Let $K$ be a convex body in $\mathbb{R}^{d+1}$ and $\omega$ be a Borel set of
 the sphere $\mathbb{S}^d$, seen as the set of unit vectors of  $\mathbb{R}^{d+1}$.
 Let $B_{\epsilon}(K,\omega)$ be the set of points $p$ which are at distance
as most $\epsilon$ from their metric projection $\overline{p}$ onto $K$ and such that $p-\overline{p}$ is collinear
to a vector belonging to $\omega$. It was proved in \cite{FJ38} that
the volume of $B_{\epsilon}(K,\omega)$  is a polynomial with respect to $\epsilon$:
\begin{equation}\label{eq:area I vol}\tag{I}
V(B_{\epsilon}(K,\omega))= \frac{1}{d+1}\sum_{i=0}^{d}\epsilon^{d+1-i}  \binom{d+1}{i} S_i(K,\omega).
\end{equation}
Each 
$S_i(K,\cdot)$ is a finite positive measure on the Borel sets of the sphere, called
the \emph{area measure of order $i$}. $S_0(K,\cdot)$ is only the Lebesgue measure of the sphere $\mathbb{S}^d$, and
$S_d(K,\omega)$ is the $d$-dimensional Hausdorff measure of the pre-image of $\omega$ for the Gauss map.
The problem of prescribing the $d$th area measure is the (generalized) Minkowski problem,
and the one of prescribing the first area measure is the (generalized) Christoffel problem
(each problem having a smooth and polyhedral specialized version). 

There are other ways of introducing the area measures \cite{Sch93}. If
$K_{\epsilon}:=K+\epsilon B$, with $B$ the unit closed ball, we have
\begin{equation}\label{eq:area II St}\tag{II}S_{d}(K_{\epsilon},\omega)=\sum_{i=0}^{d}\epsilon^{d-i}\binom{d}{i} S_i(K,\omega). 
\end{equation}

We can also use the mixed-volume $V(\cdot,\ldots,\cdot)$.
Let $h_K$ be the support function of $K$:
$$h_K(x)=\sup_{k\in K}\langle x,k \rangle $$
where $\langle \cdot,\cdot \rangle$ is the usual scalar product and $x\in\mathbb{S}^{d}$.
The set of support functions of convex bodies in $\mathbb{R}^d$
is a convex cone, that spans a linear subspace of the space of continuous functions on 
$\mathbb{S}^d$. Identifying a convex body with its support function, 
the mixed-volume is the unique 
symmetric $(d+1)$-linear form on the space of convex bodies of $\mathbb{R}^{d+1}$ with 
$V(K,\ldots,K)=V(K)$, if $V$ is the volume.  It is continuous and hence, 
fixing convex bodies $K_1,\ldots,K_{d}$, $V(K,K_1,\ldots,K_{d})$, seen as a function of $h_K$, is an additive functional on a subset
of the space of continuous functions on the sphere $\mathbb{S}^{d}$. It can be extended to
the whole space, and by the Riesz representation theorem, there exists a unique measure
$S(K_1,\ldots,K_d;\cdot)$ on the Borel sets of the sphere with
$$V(K,K_1,\ldots,K_{d})=\frac{1}{d+1}\int_{\mathbb{S}^{d}} h_K(x) \mathrm{d}S(K_1,\ldots,K_d,x). $$
The area measure of order $i$ can then be defined as
$$
S_i(K,\cdot)=S(\underbrace{K,\ldots,K}_{i},B,\ldots,B, \cdot),
$$
so the first area measure of $K$ is the unique positive measure on the sphere such that for any convex body $K'$,
\begin{equation}\label{eq:area III mix a}\tag{III}V(K',K,B,\ldots,B)=\frac{1}{d+1}\int_{\mathbb{S}^d}h_{K'}(x)\mathrm{d}S_1(K,x). \end{equation}

A last way of defining the first area measure is due to C.~Berg \cite{Ber69}. 
In the case of a strictly convex body with $C^2$ boundary $K$, the first area measure
is $\phi_K \mbox{d}\mathbb{S}^d$, with $\mbox{d}\mathbb{S}^d$ the usual volume form on the sphere
and $\phi_K$ the mean radius of curvature of $K$ (the sum of the principal radii of curvature of $\partial K$
 divided by $d$). One can compute $\phi_K$ as 
\begin{equation}\label{eq:area IV distr}\tag{IV}
\frac{1}{d}{}^{\mathbb{S}^d}\!\!\!\Delta h_K + h_K
\end{equation}
where  ${}^{\mathbb{S}^d}\!\!\!\Delta$ is the Laplacian on $\mathbb{S}^d$. The fact is that, for any convex 
body $K$, $S_1(K,\cdot)$ is equal in the sense of distributions to the formula above, defined in the sense of distributions.
All those definitions of area measures use approximation results of a convex body by a 
sequence of polyhedral or smooth convex bodies.

The Christoffel problem was completely solved independently  by  W.~Firey (for 
sufficiently smooth case in \cite{Fir67}, then generally by approximation
in \cite{Fir68}) and C.~Berg \cite{Ber69}. 
See \cite{Fir81} for an history of the 
problem to the date, and Section~4.3 in \cite{Sch93}. See \cite{GZ99}, \cite{GYY11}  for  developments around \cite{Ber69}.

\subsection{Content of the paper}
There is an active research about problems \`a la Minkowski and Christoffel for space-like hypersurfaces
of the Minkowski space (at least too many to be cited exhaustively; some references will be given further). 
However they mainly concern smooth hypersurfaces, and often in the $d=2$ case. 
One of the aim of the present paper is to introduce a class of convex set which are intended to be
the analog of convex bodies when the Euclidean structure is considered. In particular, 
they are the objects arising naturally for this kind of problems.

 In the first section of the paper we define \emph{F-convex sets}. They are
 intersection of the future sides of space-like hyperplanes,  such that
any future time-like vector is a support vector of the convex set.
This section is almost self-contained, as we have to prove 
all the basics results similar to the convex bodies theory,  for which the main source was  \cite{Sch93}. Actually we will use
some results contained in \cite{Bon05}. For example, the support functions of F-convex sets are defined 
on $\H^d$. Also, single points, which are convex bodies, are not F-convex bodies. Their analogues are future cones of single points.
However the matter is complicated because conditions on the boundary enter the picture (F-convex sets may have light-like support planes).

The motivation behind the definition of F-convex set is to be able to get 
the analog of \eqref{eq:area I vol} for the Lorentzian structure. The volume is independent of the signature of the metric, but not the orthogonal projection.
The idea is to first prove it for particular F-convex sets, called \emph{Fuchsian convex sets}  which
are F-convex sets invariant under a group of linear isometries $\Gamma$ of the Minkowski space acting cocompactly on $\H^d\subset\R^{d+1}$.
In many aspects they behave very analogously 
to convex bodies, roughly speaking compactness is replaced by ``cocompactness'' (this was noted in \cite{FF}). For them, we find formulas analogous to \eqref{eq:area I vol} and \eqref{eq:area III mix a}. As the definition of area measure is local, we use
a result of ``Fuchsian extension'' (Subsection~\ref{subsub: fuch ext}) of any part of an F-convex set to treat the general case.

We then focus on the first area measure. 
In the regular case, it is absolutely continuous with respect to the volume form of $\H^d$ with 
density the mean radius of curvature $\phi$, obtained as
\begin{equation}\label{eq_hn}
	\frac{1}{d}\Delta h -  h =\phi
\end{equation}
where $\Delta$ is the Beltrami--Laplace operator on $\mathbb{H}^d$.
In the general case, the area measure of order one is given by the formula above in the sense of distribution.

To find conditions on a given measure $\mu$ on $\H^d$ such that there exists
an F-convex set with $\mu$ as first area measure is the Christoffel problem. 
Section~\ref{sec:sol} contains computations related to the Christoffel problem. 
In the smooth case, related results were proved in \cite{So81,So83,OS83,LS06}.
Our computations go back to \cite{He59,Hel62}, and generalizes the preceding ones. See Remark~\ref{rem: helga firey et co} for more details.
In the polyhedral case, we adapt a classical construction, which appears to be related to
more recent works on Lorentzian geometry
 \cite{Mes07,Bon05,BB09}, see Remark~\ref{rem:mess pol}. 
 
 The content of Section~\ref{sec:qf sol} will be described later.

\subsection{The Fuchsian case}
Fuchsian convex sets are very special  F-convex sets, because they are at the same time
invariant under the action of a (cocompact) group and contained in the future cone of a point, which is a relevant property
as it will appear.  Seemly, they are the only F-convex sets for which a definitive result can be given,
very analogous to the one of convex bodies. By invariance, the support function of Fuchsian convex bodies can be defined 
on the compact hyperbolic manifold $\H^d/\Gamma$ instead of $\H^d$. The following statement 
stands to give an idea about the kind of results we obtained, we cannot define precisely all the terms in the introduction.

\begin{theorem}\label{thm: base general}
Let $\Gamma$ be  so that $\mathbb H^d/\Gamma$ is a compact hyperbolic manifold with universal covering map
$P_{\Gamma}:\mathbb H^d\to\mathbb H^d/\Gamma$. Let $\bar\mu$ be a positive Radon measure on $\mathbb H^d/\Gamma$. 
Define a positive Radon measure $\mu:=P_\Gamma^{\ast}\bar\mu$ on $\mathbb H^d$ as the pull-back distribution of $\bar\mu$ (see Subsection~\ref{sub:fuchssol}) 
and define the distribution 
\begin{equation*}
h_\mu:=\int_{\mathbb H^d}G(x,y)\mathrm{d}\mu(y)
\end{equation*}
where $G(x,y)$ is the kernel function defined by
\begin{equation*}
	G(x,y)= \frac{\cosh d_{\mathbb H^d}(x,y)}{v_{d-1}}
\int_{+\infty}^{d_{\mathbb H^d}(x,y)}\frac{\mathrm{d}t}{\sinh^{d-1}(t)\cosh^2(t)}
\end{equation*}
($v_{d-1}$ is the area of $\mathbb{S}^{d-1}\subset \R^d$) and the precise action of $h_\mu$ is explained in \eqref{action}. Then 
\begin{enumerate}[nolistsep]
	\item $h_\mu$ is a solution to equation 
	\begin{equation*}
\frac{1}{d}\Delta h - h = \mu 
\end{equation*}
in the sense of distributions on $\mathbb H^d$.
\item There exists a unique $\Gamma$-invariant F-convex set $K$ with first area measure $\bar\mu$ if and only if
\begin{enumerate}[nolistsep]
	\item \begin{equation*}
	      \left| \int_{\mathbb H^d}G(x,y)\mathrm{d}\mu(y)\right|<+\infty,\quad\forall x\in\H^d,
	      \end{equation*}
	\item the convexity condition 
	\begin{equation*}
\int_{\mathbb H^d} \Lambda(\eta,\nu,y)\mathrm{d}\mu(y) \geq 0,
\end{equation*}
is satisfied for all future time-like vectors $\eta,\nu$, where $\Lambda(\eta,\nu,y)$ is 
 \[
  \Lambda(\eta,\nu,y)=\Gamma(\eta,y)+\Gamma(\nu,y)-\Gamma(\eta+\nu,y)
  \]
  and $\Gamma(\eta,y)=\|\eta\|_-G\left(\frac{\eta}{\|\eta\|_-},y\right)$.
\end{enumerate}
\item  If $\mu=\bar\phi \mathrm{d}\H^d$ for some $0<\bar\phi\in C^{k,\alpha}(\hG)$, $k\geq0$ and $0\leq\alpha<1$, then $h_{\mu}\in C^{k+2,\alpha}(\hh^d)$ if $\alpha>0$ and 
$h_{\mu}\in C^{1,\beta}(\hh^d)$ for all $\beta<1$ if $\alpha=k=0$.
\end{enumerate}
\end{theorem}

If the $\bar \phi$  above is $C^2$ another characterization of convexity is given in Proposition~\ref{convexity_Lopes}. In 
this case $\bar\phi$ is the mean radius of curvature of the Fuchsian convex set with support function $h_{\mu}$.

Those conditions are very cumbersome, so necessary conditions could be wished. 
In the compact Euclidean case, necessary conditions were first given in \cite{Pog53,Pog73} 
(a proof is  in \cite{GM03}), but it does not seem to have an analogue in our case, see Remark~\ref{rem:sufficient conditions}, and the next  subsection.

\subsection{Quasi-Fuchsian convex sets and flat spacetimes}\label{sub:intro quasifuc}

A \emph{quasi-Fuchsian convex set} is the data of an F-convex set $K$ and a 
group of isometries $\Gamma$ of the Minkowski space such that
\begin{itemize}[nolistsep]
 \item $K$ is setwise invariant under the action of $\Gamma$,
 \item $\Gamma$ is isomorphic to its linear part $\Gamma_0$, which is such that $\H^d/\Gamma_0$ is a 
 compact hyperbolic manifold.
\end{itemize}

Their interest comes in part from general relativity. Actually for any  group $\Gamma$ as
above, there is a unique convex open set $\Omega$, maximal for the inclusion,
such that $\Gamma$ acts freely properly discontinuously on it. The closure of $\Omega$ is an F-convex set. 
The quotient $\Omega /\Gamma$ is a future complete flat Lorentzian spacetime, globally hyperbolic, maximal, 
spatially compact and homeomorphic to $\H^d/\Gamma_0\times \R$.  We refer to  \cite{Bar05} for a classification of such manifolds. For more details on $\Omega$, see 
\cite{Mes07,Mes07+,Bon05,BB09}.

Section~\ref{sec:qf sol} contains in particular a kind of slicing of those spacetimes by 
constant mean radius of curvature hypersurface (the ``dual'' problem of slicing by constant mean curvature hypersurfaces is classical, see \cite{ABBZ12}), 
with the particularity that the slicing goes ``outside'' of the future complete space-time and then slices a past complete spacetime.

\subsection{The Christoffel--Minkowski problem}
From now on let us consider only smooth objects. The classical Christoffel--Minkowski problem consists of
characterizing functions on the sphere which are elementary symmetric functions of
the radii of curvature of convex bodies.  Aside from the 
cases corresponding to the Minkowski and Christoffel problems, the Christoffel--Minkowski problem is not yet solved. 
Active research is still going on, see \cite{STW04,GM03,GLM06,GMZ06}  and the references inside
(see \cite{GLL12} for the ``dual'' problem of prescribing curvature measures). 
Another aim of the present paper is to bring attention to the fact 
that similar analysis can be done on the hyperbolic space or on compact hyperbolic manifolds, that still have a geometric interpretation.
Convex bodies are then replaced by F-convex sets.

For example, a Minkowski theorem (smooth version) was proved for quasi-Fuchsian convex sets in \cite{BBZ11}, in the case $d=2$.
In the Fuchsian case, it is proved in any dimension  \cite{OS83}. The Minkowski problem for quasi-Fuchsian convex sets is the subject of \cite{BF}.

%
%
%

\subsection{Acknowledgement}
The authors want to thank Francesco Bonsante, Thierry Daud\'e, Gerasim Kokarev, Yves Martinez-Maure
and Jean-Marc Schlenker. The first author enjoyed useful conversations with  Yves Martinez-Maure
about hedgehogs.  Francesco Bonsante pointed out to the first author 
the relation between the first area measure and measured geodesic laminations.

Several remarks and suggestions from the anonymous referee contributed to improve both the content and the presentation of this paper. The authors are grateful to him for such a careful reading.

The first author was partially supported by the \textit{ANR GR-Analysis-Geometry}. The second author was partially supported by \textit{INdAM COFUND-2012 Outgoing Fellowships in Mathematics and/or Applications Cofunded by Marie Curie Actions.}

\section{Background on convex sets}\label{sec:back}

\subsection{Notations}

\paragraph{Subsets of $\R^{d+1}$.}
For a set $A\subset \R^{d+1}$ we will denote by
$\overline{A}, \stackrel{\circ}{A}, \partial A$ respectively the closure, the interior and the boundary
of $A$.
A hyperplane $\mathcal{H}$ of $\R^{d+1}$ is a \emph{support plane} of a closed convex set $K$ if it has a non empty intersection with $K$ and $K$ is totally contained in one side of $\mathcal{H}$.
In this paper, a vector orthogonal to a support plane and inward pointing is a \emph{support vector} of $K$.
A
\emph{support plane at infinity} of $K$ is a hyperplane $\mathcal{H}$ such that
 $K$ is contained in one side of $\mathcal{H}$, and any
parallel displacement of $\mathcal{H}$ in the direction of $K$ meets the interior of $K$ 
($\mathcal{H}$ and $K$ may have empty intersection).  A support plane is  a support plane at infinity.

We denote by $V$ the volume form of $\R^{d+1}$ (the Lebesgue measure).

\paragraph{Minkowski space.}
The Minkowski space-time of dimension $(d+1)$, $d\geq 1$,  is  
$\mathbb{R}^{d+1}$ endowed with the symmetric bilinear form
$$\langle x,y\rangle_-=x_1y_1+\cdots+x_ny_n-x_{d+1}y_{d+1}.$$
The interior of the future cone of a point $p$ is denoted by $I^+(p)$.
We will denote  $I^+(0)$ by $\mathcal{F}$,  it is the set of future time-like 
 vectors:
$$\F=\{x\in \R^{d+1}\vert \langle x,x\rangle_-<0, x_{d+1}>0\}.$$
 $\partial \F^{\star }$ and $\overline{\F}^{\star }$ are respectively $\partial \F$ and $\overline{\F}$ without the origin
(respectively the set of future light-like vectors and the set of future vectors).
Let us also denote
$$\C(p):=\overline{I^+(p)}$$
and for $t>0$,
$$K(\H_t):=\{x\in \R^{d+1}\vert \langle x,x\rangle_-\leq -t^2, x_{d+1}>0\} $$
with $K(\H):=K(\H_1)$.

For a differentiable real function $f$ on an open set of $\mathbb{R}^{d+1}$, 
$\operatorname{grad}_{x}f$ will be the Lorentzian gradient of $f$ at $x$:
$$D_xf(X)=\langle X,\operatorname{grad}_xf\rangle_-, $$
namely the Lorentzian gradient is the vector with entries $(\frac{\partial f}{\partial x_1},\ldots,\frac{\partial f}{\partial x_d},-\frac{\partial f}{\partial x_{d+1}})$.

For two points $x,y$ on a causal (i.e.~no-where space-like) line, the Lorentzian distance 
is $$d_L(x,y)=\sqrt{-\langle x-y,x-y\rangle_-},$$ and $\| x\|_-:=d_L(x,0).$
 We have the reversed triangle inequality:
\begin{equation}\label{eq:revers triang ineq} \|x\|_-+\|y\|_-\leq \|x+y\|_-,\qquad\forall x,y\in\F.\end{equation}

An isometry $f$ of the Minkowski space has the form $f(x)=l(x)+v$, with $v\in\R^{d+1}$ and 
$l\in O(d,1)$, the group of linear maps such that $lJl=J$, with $$J=\mathrm{diag}(1,\ldots,1,-1).$$
We refer to \cite{One83} for more details. For a $C^2$ function $f:\R^{d+1}\rightarrow \R$, the wave operator is
$$\square f=\frac{\partial^2 f}{\partial x_{1}^2}+\ldots +\frac{\partial^2 f}{\partial x_{d}^2}-\frac{\partial^2 f}{\partial x_{d+1}^2} .$$

\paragraph{Hyperbolic Geometry.}
In all the paper, the hyperbolic space is identified with the pseudo-sphere
$$\H^d=\{x\in \R^{d+1}\vert \langle x,x\rangle_-=-1, x_{d+1}>0\}, $$
i.e.~$\H^d=\partial K(\H)$. We denote by $g, \nabla, \nabla^2, \Delta=\operatorname{div}\nabla$ respectively the Riemannian metric, the gradient,
the Hessian and the  Laplacian  of $\H^d$. Using
 hyperbolic coordinates on $\F$ (any
orthonormal frame $X_1,\ldots,X_d$  on
$\H^d$ extended to  an orthonormal frame of $\F$ with the decomposition
$r^2g_{\H^d}-\mathrm{d}r\otimes \mathrm{d}r$ of the metric on $\F$),
the Hessian of a function $f$ on $\F$ and the hyperbolic Hessian of its restriction to $\H^d$ are related by
\begin{equation}\label{restr hessien}
\operatorname{Hess}f=\nabla^2 f-\frac{\partial f}{\partial r}g.
\end{equation}

A function $H$ on $\F$ is \emph{positively homogeneous of degree one}, or in short \emph{$1$-homogeneous},
if $$H(\lambda \eta)=\lambda H(\eta)\, \forall \lambda > 0.$$
It is determined by its restriction $h$ to $\H^d$ via
$H(\eta)=h(\eta/\|\eta\|_-)/\|\eta\|_-$. A function $H$ obtained in 
this way will be called the \emph{$1$-extension} of $h$.

\begin{lemma}
 Let $h$ be a $C^1$ function on $\H^d$ and $H$ be its $1$-extension to $\F$.
Then
\begin{equation}\label{eq:nablanabla}\grad_{\eta}H=\nabla_{\eta}h-h(\eta)\eta. \end{equation}
Moreover, if $h$ is $C^2$, then
$\forall X,Y \in T_{\eta}\H^d$,
\begin{equation}\label{eq:II-1}\operatorname{Hess}_{\eta}H(X,Y)= \nabla^2 h (X,Y)- h g(X,Y), \end{equation}
and, for $\eta\in\H^d$,
$$\square_{\eta}H=\Delta h - dh.$$
\end{lemma}
See Figure~\ref{fig:nabla} for a geometric interpretation of \eqref{eq:nablanabla}.
\begin{proof}
Using hyperbolic coordinates on $\F$,
 $\grad_{\eta}H$ has $d+1$ entries, and, at $\eta\in\H^d$,
the $d$ first ones are the coordinates of $\nabla_{\eta} h$. 
We identify $\nabla_{\eta} h\in T_{\eta}\H^d\subset \mathbb{R}^{d+1}$
with a vector of $\mathbb{R}^{d+1}$. The last component of $\grad_{\eta}H$ is $-\partial H /\partial r(\eta)$, and,
 using
the homogeneity of $H$, it is equal to $-h(\eta)$ when $\eta\in \mathbb{H}^d$.
Note that at such a point, $T_{\eta}\F$ is the orthogonal sum of $T_{\eta}\mathbb{H}^d$ and
$\eta$, and \eqref{eq:nablanabla} follows.

On the other hand, $\nabla^2 h (X,Y)=g(D_X\nabla h,Y)$, with $X,Y\in T_{\eta}\H^d$, where 
$D$ is the Levi-Civita connection of $\H^d$. By the Gauss Formula, it is equal to the connection of
$\R^{d+1}$ plus a normal term. Differentiating $\nabla_{\eta} h=\grad_{\eta} H+h(\eta)\eta$ and using
that $\eta$ is orthogonal to $Y$ leads to \eqref{eq:II-1}. This also follows from \eqref{restr hessien}.
The last equation is well-known, see e.g. Lemma~25 in \cite{He59}.
\end{proof}

For $x_0\in \H^d$, $\rho_{x_0}(x)$ is the hyperbolic distance between 
$x_0$ and $x\in\H^d$.  This gives local spherical coordinates 
$(\rho_x,\Theta=(\theta_2,\ldots,\theta_{d}))$ centered at $x_0$ on $\H^d$.  A particular $x_0$ is
$e_{d+1}$, the vector with entries $(0,\cdots,0,1)$
and we will denote $\rho_{e_{d+1}}(x)$ by 
$\rho(x)$. We have
$\langle x,-e_{d+1}\rangle_-=x_{d+1}=\cosh \rho(x).$

As we identify the hyperbolic space with a pseudo-sphere in Minkowski space,
we will identify hyperbolic isometries with isometries of Minkowski space.
More precisely, the group of hyperbolic isometries is identified with the
group of linear isometries of the Minkowski space preserving $\F$, see  \cite{Rat06}.
In all the paper, $\Gamma$ is a given group of hyperbolic isometries (hence of 
linear Minkowski isometries) such that  $\H^d/\Gamma$ is a compact manifold.

\paragraph{Cocycles.} Let $C^1(\Gamma,\R^{d+1})$ be the space of \emph{$1$-cochains}, i.e.~the space 
of maps 
$\tau:\Gamma\rightarrow \R^{d+1}$.
For $\gamma_0\in\Gamma$, we will denote $\tau(\gamma_0)$ by
$\tau_{\gamma_0}$.
The space of \emph{$1$-cocycles} $Z^1(\Gamma,\R^{d+1})$ is the subspace of 
$C^1(\Gamma,\R^{d+1})$ of maps satisfying 
\begin{equation}\label{eq:cocycle}
 \tau_{\gamma_0 \mu_0}=\tau_{\gamma_0}+\gamma_0 \tau_{\mu_0}.
\end{equation}

For any $\tau\in Z^1(\Gamma,\R^{d+1})$ we get a group $\Gamma_{\tau}$ of isometries of 
Minkowski space, with linear part $\Gamma$ and with translation part given by $\tau$:
for  $x\in\R^{d+1}$, $\gamma\in\Gamma_{\tau}$ is defined by
$$\gamma x=\gamma_0 x+ \tau_{\gamma_0}. $$
The cocycle condition \eqref{eq:cocycle} expresses 
the fact that $\Gamma_{\tau}$ is a group. In other words, $\Gamma_{\tau}$ is a group of isometries
which is isomorphic to its linear part $\Gamma$. Of course, $\Gamma_0=\Gamma$.

The space of  \emph{$1$-coboundaries} $B^1(\Gamma,\R^{d+1})$
is the subspace of 
$C^1(\Gamma,\R^{d+1})$ of maps of the form 
$\tau_{\gamma_0}=\gamma_0v-v$ for a given $v\in\R^{d+1}$. This has the following meaning.
Let $v\in\R^{d+1}$ and let $f$ be an isometry of the Minkowski space 
with linear part $f_0$ and translation part $v$, so $f(x)=f_0(x)+v$
and $f^{-1}(x)=f_0^{-1}(x-v)$. Suppose that, for $\tau,\tau'\in Z^1$,
$\Gamma_{\tau}$ and $\Gamma_{\tau'}$ are conjugated by $f$:
$\forall\gamma\in \Gamma_{\tau}$ and $\forall\gamma'\in \Gamma_{\tau'}$
with the same linear part $\gamma_0$,
$\gamma=f\circ\gamma'\circ f^{-1}$. 
Developing $\gamma x=f\gamma' f^{-1} x$, we get 
$$\gamma_0x+\tau_{\gamma_0}=f_0\gamma_0 f_0^{-1} x-f_0\gamma_0f_0^{-1} v +f_0\tau_{\gamma_0}'+v $$ 
so for any $\gamma_0\in\Gamma$, $f_0\gamma_0 f_0^{-1}=\gamma_0$, hence
$f_0$ is trivial \cite[12.2.6]{Rat06}, $f$ is a translation by $v$, and $\tau$ and $\tau'$ differ by 
a 
$1$-coboundary. Conversely, it is easy to check that if $\tau$ and $\tau'$ differ by a $1$-coboundary,
then  $\gamma x=f\gamma' f^{-1} x$, with $f$ a translation.

Note that $B^1(\Gamma,\R^{d+1})\subset Z^1(\Gamma,\R^{d+1})$,  that they are both linear spaces, and that the dimension
of $B^1(\Gamma,\R^{d+1})$ is $d+1$. The names come from the usual cohomology of groups, and 
$H^1(\Gamma,\R^{d+1})=Z^1(\Gamma,\R^{d+1})/B^1(\Gamma,\R^{d+1})$ is the $1$-cohomology
group. 
The following lemma, certainly well-known, says that those notions are relevant only
for $d>1$. Note that for $d>2$, $H^1(\Gamma,\R^{d+1})$ may be trivial.

\begin{lemma}\label{lem:quasi dim}
$Z^1(\Gamma,\R^2)=B^1(\Gamma,\R^2)$.
\end{lemma}
\begin{proof}
 $\Gamma$  is the free group generated by 
 a Lorentz boost of the form
\begin{equation} \label{eq:boost}\gamma_0=\left(
    \begin{array}{cc}
       \cosh t  & \sinh t \\
   \sinh t & \cosh t
    \end{array}
\right)\end{equation} 
for a  $t\not= 0$. As $\gamma_0$ is a Lorentz boost on the plane, $(\mathrm{Id}-\gamma_0)$ is invertible.
Let $\tau$ be a cocycle, and define $v=:(\mathrm{Id}-\gamma_0)^{-1}\tau_{\gamma_0}$.
Then one checks easily that
for any integer $n$, $\gamma^n x=\gamma_0^n x + v -\gamma_0^n v$, that means that
$\tau$ is a coboundary. 
\end{proof}

As we will deal only with $1$-cocycles and $1$-coboundaries, we will call them
cocycles and coboundaries respectively.

\paragraph{$\tau$-equivariant functions.}
Let $\tau$ be a cocycle.  A function $H:\mathcal{F}\rightarrow \mathbb{R}$
is called \emph{$\tau$-equivariant} if it is $1$-homogeneous and satisfies 
\begin{equation}\label{eq:sup func gamma}H(\gamma_0 \eta)=H(\eta)+\langle \gamma_0^{-1}\tau_{\gamma_0} ,\eta\rangle_-.\end{equation}
See Remark~\ref{ex: gamma inv supp fct} for the  existence of such functions. 
A function $h:\H^d\rightarrow \mathbb{R}$ is called $\tau$-equivariant
if its $1$-extension is $\tau$-equivariant. Note that 
a $0$-equivariant map on $\H^d$ satisfies 
$$h(\gamma_0 \eta)=h(\eta)$$
$\forall\eta\in\H^d$, and hence has a well-defined 
quotient on the compact hyperbolic manifold $\H^d/\Gamma$.
Conversely, the lifting of any function defined on $\H^d/\Gamma$ gives
a $0$-equivariant map on $\H^d$, that is, a $\Gamma$ invariant map.

Examples of  $\tau$-equivariant functions are given in the lemma below.
Non-trivial examples will follow from Remark~\ref{ex: gamma inv supp fct}.

\begin{lemma}\label{lem:equivariant}
Let $\tau,\tau'$ be two cocycles. 
\begin{enumerate}[nolistsep,label={\bf(\roman{*})},ref={\bf(\roman{*})}]
\item\label{item diff}  The difference of two $\tau$-equivariant maps is $0$-equivariant.
\item\label{item sum} The sum of a $\tau$-equivariant and a $\tau'$-equivariant map is
$(\tau+\tau')$-equivariant. The product of  a $\tau$-equivariant map with a real $\alpha$ is
$(\alpha\tau)$-equivariant. 
\item\label{item quasi inv} If there exists $H:\mathcal{F}\rightarrow \mathbb{R}$  
at the same time $\tau$-equivariant
and $\tau'$-equivariant, then $\tau=\tau'$.
\item \label{item coboundary}  If $\tau$ is a coboundary ($\tau_{\gamma_0}=v-\gamma_0v$),
then the map $\eta\mapsto\langle\eta,v\rangle_-$ is $\tau$-equivariant.
\item \label{item cocycle diff} If $\tau$ is a coboundary and
$H$ is $\tau$-equivariant, then there exists a $0$-equivariant map $H_0$ with $H=H_0+\langle\cdot,v\rangle_-$.
\end{enumerate}
\end{lemma}
\begin{proof} \ref{item diff} and \ref{item sum} are straightforward from \eqref{eq:sup func gamma}.
 
\ref{item quasi inv} From \eqref{eq:sup func gamma}, for any $\eta\in\F$, $\gamma_0\in\Gamma$,
$$H(\eta)+\langle \gamma_0^{-1}\tau_{\gamma_0}' ,\eta\rangle_-
=H(\gamma_0\eta)=H(\eta)+\langle \gamma_0^{-1}\tau_{\gamma_0} ,\eta\rangle_-,$$
so for any $\eta\in\F$,
$\langle \gamma_0^{-1}(\tau_{\gamma_0}-\tau_{\gamma_0}') ,\eta\rangle_-=0 $
that leads to $\tau_{\gamma_0}=\tau_{\gamma_0}'$. 

 \ref{item coboundary} It is immediate that $\langle \gamma_0^{-1}\tau_{\gamma_0},\eta \rangle_-=
\langle v,\eta\rangle_--\langle v,\gamma_0\eta\rangle_-$.

\ref{item cocycle diff} $H-\langle\cdot,v\rangle_-$ is $0$-equivariant by \ref{item diff}
and \ref{item coboundary}
\end{proof}

The general structure of the set of equivariant maps can be summarized as follows.
\begin{itemize}[nolistsep]
\item  $\mathcal{E}(\Gamma)$ is the vector space of $0$-equivariant functions.
\item  $\mathcal{E}(\Gamma_{\tau})$ is the affine space over $\mathcal{E}(\Gamma)$  of 
$\tau$-equivariant functions.
 \item  $\cup_{\tau\in Z^1}\mathcal{E}(\Gamma_{\tau})$ is the vector space
of equivariant functions for $\Gamma$. The union is disjoint.
\end{itemize}

 Let $H$ be a $C^1$ $\tau$-equivariant function. 
For any  $\gamma\in\Gamma_{\tau}$ it is easy to check that 
\begin{equation}\label{eq:eq grad}\grad_{\gamma_0\eta}H=\gamma\grad_{\eta}H\end{equation}
and, if $H$ is $C^2$, for $X,Y\in\R^{d+1}$,
$$\hess_{ \gamma_0 \eta}H(\gamma_0 X,\gamma_0 Y)= \hess_{\eta}H(X,Y).$$
From \eqref{eq:II-1},   if $X,Y\in T_{\eta}\H^d$ and $h$ is the restriction of
$H$ to $h$,
$$ \nabla^2_{\gamma_0\eta} h (d_{\eta}\gamma_0(X),d_{\eta}\gamma_0(Y))-
 h(\gamma_0\eta) g(d_{\eta}\gamma_0( X),d_{\eta}\gamma_0(Y))= \nabla^2_{\eta} h (X,Y)- h(\eta) g(X,Y). $$
Let us state it as
\begin{lemma}\label{lem: hess gamma inv} Let $h$ be  a $\tau$-equivariant map on $\H^d$. Then
 $\nabla^2 h - h g$ is $0$-equivariant.
\end{lemma}

\subsection{F-convex sets}

Let $K$ be a proper closed convex set of $\mathbb{R}^{d+1}$ defined as the intersection of the future side of space-like
hyperplanes.

\begin{lemma}\label{lem: base convexe} Let $K$ be a convex set as above.
 \begin{enumerate}[nolistsep,label={\bf(\roman{*})},ref={\bf(\roman{*})}]
\item\label{base:cone} $\forall k\in K, I^+(k)\subset \stackrel{\circ}{K}$,
\item\label{base:interieur}  $K$ has non empty interior,
\item\label{base:time}  $K$ has no time-like support plane,
\item\label{base:ray}  if $k\in\partial K$ is contained in a light-like support plane of $K$, then 
 $k$ belongs to a light-like half-line contained in $\partial K$.
\end{enumerate}
\end{lemma}
\begin{proof}
\ref{base:cone} The definition says that there exists a family $\eta_i$, $i\in I$ of future time-like vectors 
and a family $\alpha_i$ of real numbers such that any $k\in K$ satisfies 
$\langle k,\eta_i\rangle_-\leq \alpha_i $
for all $i\in I$. For any future time-like or light-like vector $\ell$ we have $\langle \eta_i,\ell\rangle_-<0$,
hence $\langle k+\ell,\eta_i\rangle_-\leq \alpha_i $. \ref{base:interieur} follows from \ref{base:cone}.

\ref{base:time} If $k\in K$ is contained in a time-like support plane, then 
$ I^+(k)$ is not in the interior of $K$, that contradicts \ref{base:cone}.

\ref{base:ray} The intersection of the light-like support  hyperplane
with the boundary of $I^+(k)$ must be contained in the boundary of $K$.
\end{proof}
The half-line in \ref{base:ray} can not be extended in the past, because any light-like line meets any space-like hyperplane. 
But the end-point of the half-line is not necessarily contained in a space-like support plane, see Example~\ref{ex:cotenu light}.

 An \emph{F-convex set}
is a convex set as above such that  any future time-like vector is a support vector:
\begin{equation}\label{eq: def Lcvx}
\forall \eta\in\F,  \exists \alpha\in\R, \forall k\in K,  \langle \eta,k\rangle_-\leq \alpha~. 
\end{equation}
For example the intersection of the future side of two 
space-like hyperplanes is not an F-convex set.
The  $K(\H_t)$'s are  F-convex sets. They will play a role analogue to the balls centered at the origin in the classical case. 
 The cone $\C(p)$ of a point $p$, in particular $\overline{\mathcal{F}}$, is an F-convex set. This example shows that an F-convex set can have
light-like support planes.

The following observation can be helpful.

\begin{lemma}\label{lem: conv dans Fconv}
A proper closed convex set  defined as the intersection of the future side of space-like
hyperplanes contained in an F-convex set is an F-convex set.
\end{lemma}

The following lemma says that for an F-convex set $K$, any space-like hyperplane is parallel to a support plane of $K$.

\begin{lemma}[{\cite[Lemma~3.13]{Bon05}}]\label{lem:existence point boundary}
 Let $K$ be an F-convex set. Then $\forall \eta\in\F,  \exists \alpha\in\R, \exists k\in K,  \langle \eta,k\rangle_-= \alpha$.
\end{lemma}

%
%
%
%

\begin{lemma}\label{lem:pas half space}
If an F-convex set $K$ contains a half-line in its boundary, then this half-line is light-like.
\end{lemma}
\begin{proof}
 It follows from Lemma~\ref{lem: base convexe} that the half-line cannot be time-like. Let us suppose that
the boundary contains a space-like half-line starting from $x$ and directed by the space-like vector $v$. 
Hence for any $\lambda>0$, $x+\lambda v\in K$. Let $\eta\in\H^d$ be such that
$\langle \eta,v\rangle_- >0$. By definition of F-convex set, there exists $\alpha\in\mathbb{R}$
such that $\forall k\in K$, $\langle k,\eta\rangle_-\leq \alpha$.
Then for any $\lambda$,
$ \langle \eta,x+\lambda v\rangle_- \leq \alpha $, that is impossible.
\end{proof}

We denote by  $\partial_{s}K$ the set of points of $\partial K$ which are contained
in a space-like support plane. 

\begin{lemma}\label{lem: acausal}
Let $k_1,k_2\in  \partial_{s}K$. Then $k_1-k_2$ is space-like. 
\end{lemma}
\begin{proof}
 Let us suppose that $k_1-k_2$ is not space-like. Up to exchange $k_1$ and $k_2$, let us suppose
 that $k_1-k_2$ is future (light-like or time-like).
 Let $\eta$ be a support future time-like vector of $k_1$. Then $\langle \eta,k_2\rangle_-\leq \langle \eta,k_1\rangle_-$,
i.e.~$\langle \eta,k_1-k_2\rangle_-\geq 0$, that is impossible for two future vectors (they are not both light-like). 
\end{proof}

 \begin{remark}[\textbf{P-convex sets}]\label{rem:Pconvex}{\rm 
 Similarly to the definition of F-convex set, a \emph{P-convex set} 
 $K$ is a proper closed convex set of $\mathbb{R}^{d+1}$ defined as the intersection of the past side of space-like
 hyperplanes and such that any past time-like vector is a support vector:
 \begin{equation*}\label{eq: def Pcvx}
 \forall \eta\in\F,  \exists \alpha\in\R, \forall k\in K,  \langle -\eta,k\rangle_-\leq \alpha~. 
 \end{equation*}
 The study of P-convex sets reduces to the study of F-convex sets because clearly the symmetry
 with respect to the origin is a bijection between F-convex and P-convex sets.
 Note that the symmetric of a $\tau$-F-convex set is a $(-\tau)$-P-convex set.
 In particular, the symmetric of a $\tau$-F-convex set is a $\tau$-P-convex set if and only if $\tau=0$.   
 }
 \end{remark}

 \begin{example}[\textbf{$\tau$-F-convex sets}]\label{ex: domain}{\rm 
 Let $\tau$ be a cocycle and $\Gamma_{\tau}$ be the corresponding group. A \emph{$\tau$-F-convex set} is an F-convex set setwise invariant 
under the action of $\Gamma_{\tau}$. They are the quasi-Fuchsian convex sets mentioned in the introduction.
If $\tau=0$, the F-convex sets are $\Gamma$ invariant. They are  ``Fuchsian'' according to the terminology of the introduction.
The $K(\H_t)$'s and $\overline{\F}$ are Fuchsian.

$\tau$-F-convex sets are F-convex sets \cite[Lemma~3.12]{Bon05}.

Moreover, there exists a unique maximal domain $\Omega_{\tau}$ on which $\Gamma_{\tau}$ acts freely and 
properly discontinuously. Its closure $\overline{\Omega_{\tau}}$ is a $\tau$-F-convex set.
Actually, $\Omega_{\tau}$ is maximal in the sense that any $\tau$-F-convex set is contained
in $\overline{\Omega_{\tau}}$ \cite{Bar05,Bon05}. 

The elementary example is $\Omega_0=\F$. There also exists a past domain with the same property.
See Subsection~\ref{sub:intro quasifuc} and also  \cite{BMS13} for an up-to-date overview.
}
\end{example}

\begin{remark}[\textbf{Regular domains}]{\rm 
A (future) regular (convex) domain is a convex set which is the intersection of
the future sides of light-like hyperplanes,
 and such
that at least two light-like support planes exist. Regular domains were introduced in \cite{Bon05}.
 See also \cite{BB09} for the $d=2$ case. 
The intersection of the future side of two light-like hyperplanes is a regular domain but not an F-convex set.
The F-convex set $K(\H)$ bounded by $\H^d$ is an F-convex set which is not a regular domain. We will call
\emph{F-regular domains} the regular domains which are F-convex sets. Future cones of points are
F-regular domains.  The $\overline{\Omega_{\tau}}$ are F-regular domains
}
\end{remark}

\subsection{Gauss map}

Let $K$ be an F-convex set.
The inward unit normal of a space-like support plane is identified with an element of $\H^d$. 
The \emph{Gauss map} $G_K$ of $K$ is a set-valued map from $\partial K$ to $\H^d$. It
 associates to 
each point on $\partial K$ the inward unit normals of all the space-like support planes at this point.
The Gauss map is defined only on  $\partial_s K$. 
By the definition and Lemma~\ref{lem:existence point boundary}, F-convex sets are exactly the future convex sets with $G_K(\partial_sK)=\H^d$.

\begin{example} {\rm
 The Gauss map of $K(\H_t)$ is $x\mapsto x/t$.
 The Gauss map of  $\C(p)$ 
is defined only at the apex $p$ of the cone. It maps $p$ onto the whole $\H^d$.   
}
\end{example}

\subsection{Minkowski sum}

The (Minkowski) sum of two sets $A,B$ of $\R^{d+1}$ is 
$$A+B:=\{a+b | a\in A, b\in B\}. $$
It is immediate from \eqref{eq: def Lcvx}
that the sum of two F-convex sets is an F-convex set.
It is also immediate that if $\lambda>0$ and $K$ is an F-convex set, then
$\lambda K=\{\lambda k | k\in K\}$ is also an F-convex set. If $\lambda<0$, $\lambda K$ is a 
P-convex set.

Note that $\C(p)=\{p\}+\overline{\F}$. Moreover if $K$
is an F-convex set and $k\in K$, then  $\C(k)\subset K$,  
so $K+\overline{\F}=K$ and  then, 
for any $p\in\R^{d+1}$, $K+\C(p)=K+\{p\}$. $K+\{p\}$  is the set obtained by a translation of $K$ along the vector $p$.

\begin{example}{\rm
Let $K$ be a $\tau$-F-convex set and $p\in\R^{d+1}$.
Then $K+\{p\}$ is a $\tau'$-convex set, with $\tau'$ differing 
from $\tau$ by a coboundary: $\tau'_{\gamma_0}=\tau_{\gamma}+p-\gamma_0 p$.
Lemma~\ref{lem:quasi dim} says that in $d=1$, 
any $\tau$-F-convex set is the translation of a Fuchsian convex set.
}\end{example}

\subsection{Extended support function}\label{sub: extended}

Let $K$ be an F-convex set. The \emph{extended support function} 
$H_K$ of $K$ is the map from $\F$ to $\mathbb{R}$ defined by
\begin{equation}\label{def:sup func}
\forall \eta\in\F, H_K(\eta)=\mbox{sup}\{\langle k,\eta\rangle_- \vert k\in K\}
\end{equation}
Note that the sup is a max by Lemma~\ref{lem:existence point boundary}. By definition
$$
K=\{k\in\R^{d+1} \vert \langle k,\eta\rangle_-\leq H_K(\eta), \forall \eta \in \F\}.
$$

An extended support function
is sublinear, that is $1$-homogeneous
and subadditive: $$H(\eta+\mu)\leq H(\eta)+H(\mu).$$ 
For a $1$-homogeneous function, subadditivity and convexity are equivalent. In particular
$H$ is continuous. 
 Note that, for $\lambda >0$,
\begin{equation}\label{eq sum supp}H_{K+K'}=H_K+H_{K'}, \, H_{\lambda K}=\lambda H_K.\end{equation}
Hence
$$K+K'=K+K''\Rightarrow K'=K''.$$

\begin{example}{\rm
The extended support function of $K(\H_t)$ is $-t\| \eta\|_-$. The sublinearity is
equivalent to the reversed triangle inequality \eqref{eq:revers triang ineq}.
The extended support function of
$\C(p)$ is the restriction to $\F$ of the linear form $\langle \cdot, p\rangle_-$. In particular the support function of $\C(0)=\overline{\F}$ is the 
null function.
}\end{example}

 As from the definition
$$K\subset K'\Leftrightarrow  H_K\leq H_{K'}$$
it follows from the example above that

$$ K\subset \overline{\F} \Leftrightarrow H_K\leq 0.$$

Actually, for $ K\subset \overline{\F}$, if $0\in K$, then $\overline{F} \subset K$ and then  $ K= \overline{\F}$.
That says that
\begin{equation}\label{eq: cone 0}
K\subset \overline{\F}^{\star } \Leftrightarrow H_K<0.
\end{equation}

\begin{remark}\label{ex: gamma inv supp fct}{\rm
Let $K$ be a $\tau$-F-convex with extended support function $H$. By definition of the support function, for $\eta\in F$ and $\gamma\in\Gamma_{\tau}$ with linear part
$\gamma_0$,
$$H(\gamma_0\eta)=\mbox{sup}\{\langle k,\gamma_0\eta\rangle_- \vert k\in K\}
=\mbox{sup}\{\langle \gamma k,\gamma_0\eta\rangle_- \vert \gamma k\in K\} $$
$$=\mbox{sup}\{\langle \gamma_0 k,\gamma_0\eta\rangle_- + \langle \tau_{\gamma_0},\gamma_0\eta\rangle_- \vert  k\in K\} 
=H(\eta)+\langle \tau_{\gamma_0},\gamma_0\eta\rangle_-, $$
so $H$ is $\tau$-equivariant.
In particular the existence of $\tau$-F-convex sets implies the existence of 
$\tau$-equivariant functions, and Lemma~\ref{lem:equivariant}
gives properties on $\tau$-F-convex sets.
For example, from \eqref{eq sum supp} we get that if $K$ (resp. $K'$) is a
$\tau$-F-convex set (resp. $\tau'$-convex set) then 
 $\alpha K+K'$ is a $(\alpha \tau+\tau')$-convex set.
Also, a  $\tau$-F-convex set can not be a $\tau'$-convex set
if $\tau\not= \tau'$.
}
\end{remark}

\subsection{Total support function}

The extended support function $H$ of an F-convex set 
is defined only on $\F$ and we will see that this suffices to determine the 
F-convex set. 
The \emph{total support function} of $K$ is, $\forall \eta\in\R^{d+1}$,
$$  \tilde{H}_K(\eta)=\mbox{sup}\{\langle k,\eta\rangle_- \vert k\in K\}.$$
 We have $\tilde{H}_K(0)=0$ and $\tilde{H}_K=H_K$ on $\F$. We also have 
 $\tilde{H}_K=+\infty$ outside of $\overline{\F}$. This expresses the fact that
 $K$ has no time-like support
plane and that $K$ is not in the past of a non time-like hyperplane.
 The question
is what happens on $\partial\F$. 
As a supremum of a family of continuous functions, 
 $\tilde{H}_K$ is lower semi-continuous, hence a classical result gives the following lemma,  see
propositions~IV.1.2.5 and 1.2.6 in \cite{HUL93} or theorems 7.4 and 7.5 in \cite{Roc97}.

\begin{lemma}\label{lem cl}
For any $\ell\in\partial\F$ and any $\eta\in\F$, we have
$$\tilde{H}_K(\ell)=\underset{t\downarrow 0}{\mathrm{lim}} H_K(\ell+t(\eta-\ell)). $$
\end{lemma}

 Let $K$ be an F-convex set and $\tilde{H}$ be its total support function.
If $\tilde{H}(\ell)$ is finite for a future light-like vector $\ell$, then the 
light-like hyperplane $$\ell^{\bot}:=\{x\in\R^{d+1}|\langle x,\ell\rangle_-= \tilde{H}(\ell)\}$$ 
is a support plane at infinity of $K$: $K$ is contained in the future side of $\ell^{\bot}$, and any
parallel displacement of $\ell^{\bot}$ in the future direction meets the interior of $K$. Of course
$\ell^{\bot}$ and $K$ may have empty intersection, for example any light-like vector hyperplane
is a support plane at infinity for $K(\H)$, but they never meet it.

The following fundamental result allows to recover the F-convex set from a sublinear function.

\begin{lemma}\label{lem: deter Hcs}
 Let $H:\F\rightarrow \R$ be a sublinear function. Then $H$
is the extended support function of the F-convex set
\begin{equation}\label{eq:def K supp}
K=\{x\in\R^{d+1} \vert \langle x,\eta\rangle_-\leq H(\eta), \forall \eta \in \F\}.
\end{equation}
\end{lemma}
The set $K$ as defined above is clearly a convex set  as an intersection of 
half-spaces. If it is an F-convex set,  it has an extended support function $H'$, and a priori $H'\leq H$.
\begin{proof}
 We define 
 $\tilde{H}$ as the closure of the convex function which is $H$ on 
$\F$ and $+\infty$ outside of $\F$: $\tilde{H}(x)$ is defined as $\operatorname{Liminf}_{x\rightarrow y}H(y)$.
$\tilde{H}$ is then lower semi-continuous and sublinear \cite[p.~205]{HUL93}.
We know  
(see e.g.~Theorem 2.2.8 in \cite{Hor07} or V.3.1.1. in \cite{HUL93}) that the set $$  
F=\{x\in\R^{d+1}\vert \langle x,\eta\rangle_- \leq \tilde{H}(\eta) \,\forall \eta\in\R^{d+1} \}
$$
is a closed convex set with total support function $\tilde{H}$.
As $\tilde{H}$ takes infinite values on $\R^n\setminus \overline{\F}$, we have
$$  
F=\{x\in\R^{d+1}\vert \langle x,\eta\rangle_- \leq \tilde{H}(\eta) \,\forall \eta\in\overline{\F} \}
$$
Finally as $\tilde{H}$ and $H$ coincide on $\F$
\cite[IV, Proposition~1.2.6]{HUL93}, and by definition of $\tilde{H}$, we get $F=K$. 
It follows that $K$ is a closed convex set with $H$ as  extended 
support function. The definition of $K$ says exactly that it is the intersection of 
the future of space-like hyperplanes, and as its extended support function is defined for 
any $\eta\in\F$, it is an F-convex set.
\end{proof}

\begin{remark}\label{rem: bord infini domaine}\rm{
For any $\eta\in\H^d$, consider a sequence 
 $(\gamma_0(n))_n$  of $\Gamma$ such that $\gamma_0(n)\eta/(\gamma_0(n)\eta)_{d+1}$ converges to 
a light-like vector $\ell$. Then, for any $\tau$-equivariant function $H$ we have
$$H\left(\frac{\gamma_0(n)\eta}{(\gamma_0(n)\eta)_{d+1}}\right)=
 \frac{H(\gamma_0(n)\eta)}{(\gamma_0(n)\eta)_{d+1}}
=\frac{H(\eta)}{(\gamma_0(n)\eta)_{d+1}}+\left\langle \frac{\gamma_0(n)\eta}{(\gamma_0(n)\eta)_{d+1}},
\tau_{\gamma_0(n)}\right\rangle_-. $$
This limit does not depend on the choice of the $\tau$-invariant function. Note that if
 $\tau=0$  the limit is $0$.
Take care that, even in the case where the limit above is finite, we cannot deduce 
 that 
the extended support function of a $\tau$-F-convex set has finite value 
at $\ell$.
When $\gamma_0(n)=\gamma_0^n$, all the orbits are on the geodesic fixed by the isometry
$\gamma_0$, and 
 Lemma~\ref{lem cl} says that the limit of the expression above is $\tilde{H}(\ell)$, and  
Proposition~3.14 in \cite{Bon05} says that the value is finite. It says even more, that 
$\ell$ is normal to a support plane (and not only to a support plane at infinity).
Actually $H$ has a continuous extension on the sphere \cite{BF}, but 
the set of directions of light-like support planes has zero measure \cite[Proposition~4.15]{BMS13}.
}
\end{remark}

\subsection{Restricted support function}

As an extended support function is homogeneous of degree one, it is determined by its restriction
to 
$\H^d$, which we call the  \emph{(restricted) support function}.

\begin{example}\label{ex:restricted lin}{\rm
 The support function of $K(\H_t)$ is the constant function $-t$.

The expression of support function $h_p$ of $\C(p)$ depends on $p$, and is given by the standard formulas 
relating the distance in the hyperbolic space and the Minkowski bilinear form, see  \cite{Thurcour1}.
\begin{itemize}[nolistsep]
 \item If $p$ is the origin, $h_p=0$.
\item If $p$ is time-like, then $h_p(\eta)=\pm \|p\|_- \cosh \rho_{\overline{p}}(\eta)$
where the sign depends on if $p$ is past or future, and $\overline{p}$ is the central projection
of $p$ (or $-p$) on $\H^d$.
\item If $p$ is space-like, then $h_p(\eta)= \langle p,p\rangle_-^{1/2} \sinh d^*(\eta,p^{\bot})$
where $d^*$ is the signed distance from $\eta$ to
the totally geodesic hyperplane defined by the orthogonal $p^{\bot}$ of the vector $p$.
\item If $p$ is light-like then $h_p(\eta)=\pm e^{d^*(\eta,H_p)}$
where $d^*$ is the distance between $\eta$ and the horosphere $$\{x\in\H^d | \langle x,\pm p\rangle_-=-1\}$$
the sign depending on whether $p$ is past or future.
\end{itemize}}
\end{example}

Let us consider spherical coordinates $(\rho,\Theta)$ on $\H^d$
centered at $e_{d+1}$. 
Along radial directions, the subadditivity of the extended support function can be read
on the restricted support function.

\begin{lemma}\label{lem:convexite radiale}
 Let $h$ be the support function of an F-convex set. If $\Theta$ is fixed, 
 then for any real $\alpha$,
\begin{equation}\label{eq:hyp convex'}
 h(\rho+\alpha,\Theta)+h(\rho-\alpha,\Theta)\geq 2 \cosh (\alpha)  h(\rho,\Theta).
\end{equation}
\end{lemma}
\begin{proof}
 As $\Theta$ is fixed, let us denote $h(\rho):=h(\rho,\Theta)$.
 The proof is based on the following elementary formula: for $\rho,\rho'\in\R$ we have
\begin{equation}\label{eq:addition cosh}
 \binom{\sinh \rho}{\cosh \rho}+\binom{\sinh \rho'}{\cosh \rho'}=2\cosh\left(\frac{\rho-\rho'}{2}\right)\binom{\sinh\frac{\rho+\rho'}{2}}{\cosh \frac{\rho+\rho'}{2}}.
\end{equation}
 This is easily checked by direct computation but it is more fun to use the hyperbolic exponential
(see e.g. supplement C in \cite{Yag79} or \cite{CB11})
$$e^{\mathbf{h} \rho}=\cosh \rho + \mathbf{h} \sinh \rho $$
where $\mathbf{h} \notin\R$ is such that $\mathbf{h}^2=1$. As in the complex case we get
$$e^{\mathbf{h} \rho}+e^{\mathbf{h} \rho'}=e^{\mathbf{h} \rho}e^{\mathbf{h} \frac{\rho'-\rho}{2}}\left(e^{\mathbf{h} 
\frac{\rho'-\rho}{2}}+e^{-\mathbf{h} \frac{\rho'-\rho}{2}} \right)=2\cosh \left(\frac{\rho-\rho'}{2}\right)e^{\mathbf{h} \frac{\rho'+\rho}{2}}.$$
Then
\begin{equation}\label{eq:conv d1}
 h(\rho)+h(\rho')=  H\left(\binom{\sinh \rho}{\cosh \rho}\right)+H\left(\binom{\sinh \rho'}{\cosh \rho'}\right)\geq   
 H\left(\binom{\sinh \rho}{\cosh \rho}+\binom{\sinh \rho'}{\cosh \rho'}\right)
 \stackrel{\eqref{eq:addition cosh}}{=}2\cosh\left(\frac{\rho-\rho'}{2}\right)h\left(\frac{\rho+\rho'}{2}\right)
\end{equation}
 which is \eqref{eq:hyp convex'} up to change of variable. 
\end{proof}

Fixing a $\Theta$ we get a radial direction  along a half-geodesic of $\H^d$.
It corresponds to a half time-like plane in $\R^{d+1}$, whose intersection with 
$\partial \F$ gives a light-like half-line. We denote by $\ell_{\Theta}$ the light-like vector on 
this line which has last coordinate equal to one.

\begin{lemma}\label{lem:h infini}
For an F-convex set $K$ we have 
$$
\underset{\rho\rightarrow +\infty}{\mathrm{lim}}\frac{h_K(\rho,\Theta)}{\cosh (\rho)}=\tilde{H}_K(\ell_{\Theta}).
$$
In particular, $K$ has a support plane at infinity directed by $\ell_{\Theta}$ if and only if
\begin{equation}
 \underset{\rho\rightarrow +\infty}{\mathrm{lim}}\frac{h_K(\rho,\Theta)}{\cosh (\rho)}<+\infty. 
\end{equation}
\end{lemma}
\begin{proof}
 We have
$$h_K(\rho,\Theta)=(\rho,\Theta)_{d+1}H_K\left(\frac{(\rho,\Theta)}{(\rho,\Theta)_{d+1}}\right)= 
\cosh (\rho) H_K\left(\frac{(\rho,\Theta)}{(\rho,\Theta)_{d+1}}\right).$$
We can write (see Figure~\ref{fig:radial})
$$\frac{(\rho,\Theta)}{(\rho,\Theta)_{d+1}}=(1-\tanh (\rho)) e_{d+1} + \tanh (\rho)\ell_{\Theta}. $$
Setting $t:=1-\tanh(\rho)$ the result follows because by Lemma~\ref{lem cl}
$$\tilde{H}_K(\ell_{\Theta})=\underset{t\rightarrow 0}{\mathrm{lim}}H_K(te_{d+1}+(1-t)\ell_{\Theta}).$$
\end{proof}

\begin{figure}[ht]
\begin{center}
\includegraphics[scale=0.5]{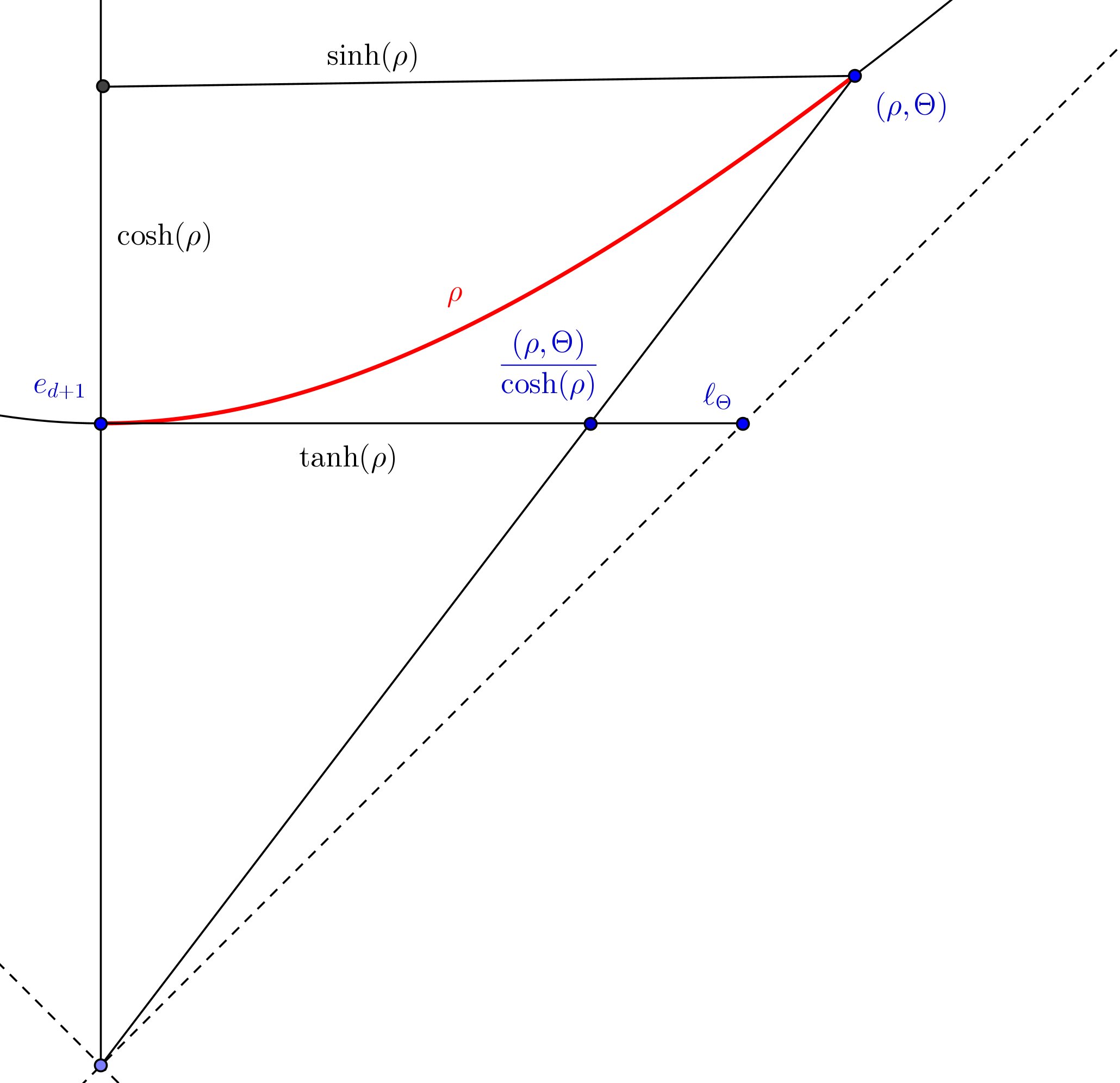}  
\end{center}
\caption{To Lemma~\ref{lem:h infini}.}\label{fig:radial}
\end{figure}

\begin{lemma}\label{lem: petit o}
 Let $K$ be an F-convex set with support function $h$ and total support function 
$\tilde{H}$. 
\begin{itemize}[nolistsep]
 \item If for any $\Theta$, 
$h(\eta,\Theta)=o(\cosh(\rho(\eta))),\eta\to\infty$  (in particular if $h$ is bounded) then 
$\tilde{H}$ equals $0$ on $\partial \F$. 
\item If $\tilde{H}$ equals $0$ on $\partial \F$, $h$ is either negative and $K\subset \overline{\F}^{\star }$
 or $h$ is equal to $0$
and $K=\overline{\F}$.
\end{itemize}
\end{lemma}

\begin{proof}
 If $h$ satisfies the hypothesis, it is immediate from the previous lemma that $\tilde{H}$ 
 equals $0$ on $\partial \F$. 
As $\tilde{H}$ is convex and equal to $0$ on $\partial \F$, it is non-positive on $\F$.
Suppose that there exists $x\in\F$ with $\tilde{H}(x)=0$, and let $y\in\F\setminus\{x\}$. 
By homogeneity,  $\tilde{H}( \lambda x)=0$ for all  $\lambda>0$. Up to choose an appropriate
$\lambda$, we can suppose that the line joining $\lambda x$ and
$y$ meets $\partial \F$ in two points. Let $\ell$ be the one such that
there exists $t\in]0,1[$ with $\lambda x=t \ell+(1-t) y$. By convexity and because
$\tilde{H}(\lambda x)=\tilde{H}(\ell)=0$, we get $0\leq \tilde{H}(y)$, hence $\tilde{H}(y)=0$ and
$H\equiv 0$. The conclusion follows from \eqref{eq: cone 0}.
\end{proof}

\begin{remark}[\textbf{Intrinsic properties of restricted support functions}]{\rm
Let $h$ be the  restricted support function of a F-convex set $K$, with homogeneous extension $H$ on $\F$.
\begin{itemize}[nolistsep]
 \item $h$ is locally Lipschitz. Actually as a convex function, $H$
 is locally Lipschitz for the usual Euclidean metric on $\R^{d+1}$. Clearly $h$ is Lipschitz 
 on $\H^d$ for the Riemannian metric induced by the Euclidean ambient metric. The result follows 
 because local Lipschitz condition does not depend on the Riemannian metric.
 \item $h$ is $(-1)$-convex. This means that on any unit speed geodesic $\gamma$, there exists a function $f$ such that $f$ is a
 solution of $f''-f=0$, $h\circ \gamma=f$ at a point $t$, and $h\geq f$ near $t$. This is usually denoted by
 $h''-h\geq 0$. 
 In our case, $f$ is given by the restriction of the support function of the future cone of any point
 of $\partial_sK$ contained in a support plane orthogonal to $\gamma(t)\in\H^d\subset\R^{d+1}$.
 \item  Actually any  $(-1)$-convex function $h$ on $\H^d$ is a support function. It suffices to prove
 that the extension $H$ of $h$ to $\F$ is subadditive. Let $x,y\in \F$. They span a plane $P$, which
 defines a geodesic $\gamma$ on $\H^d$. Let $f$ and $F$ as above, such that $h=f$ at $(x+y)/\|x+y\|_-$. Then
 $H(x+y)=F(x+y)=F(x)+F(y)\leq H(x)+H(y)$.
 \item A classical example of $(-1)$-convex function on $\H^d$ is the composition of $\cosh(\cdot)-1$ with the distance to
 a point $p$. The extension of this function on $\F$ is $-\langle p,\cdot\rangle_- - \|\cdot\|_-$, which is the
 support function of the convex side of the unit hyperboloid in the future cone at $-p$.
\end{itemize}
}\end{remark}

\begin{remark}[\textbf{Euclidean support function of F-convex sets}]\label{rem:eucl}{\rm 
Let $\eta$ be a support vector of an F-convex set $K$,
orthogonal to a support plane $\mathcal{H}$. 
For a vector $v\in\mathcal{H}$, $\langle v,\eta\rangle_-=0$,
i.e.~in matrix notation, $^tv.J.\eta=0$, $J=\mathrm{diag}(1,\ldots,1,-1).$
So $v$ is orthogonal to $J\eta$ for the 
standard Euclidean metric: $J\eta$ is an Euclidean outward support vector to $K$.
Hence the Euclidean support function of an F-convex set is defined 
on the intersection of the Euclidean unit sphere and the interior of the past cone of the origin. Let us denote by $S$
the map from $\H^d$ to this part of $\mathbb{S}^d$:
$$S(\eta)=\frac{J\eta}{\|J\eta\|}=\frac{J\eta}{\|\eta\|}, $$
with $\langle\cdot,\cdot\rangle$ the usual scalar product and $\|\cdot\|$ the associated norm.
Let $x\in K$ with $h(\eta)=\langle x,\eta \rangle_-$. So
$$h(\eta)=\langle x,J\eta\rangle=\langle x ,S(\eta)\rangle \|\eta\|, $$
and for suitable radial coordinates on $\H^d$, $\eta=(0,\ldots,0,\sinh(\rho),\cosh(\rho))$, so  if $h^E$ is the Euclidean support function of $K$ (the supremum is reached at the same point
$x$ for the two bilinear forms):
$$h(\eta)=\sqrt{\cosh(2\rho)}h^E(S(\eta)).$$

}

\end{remark}

\begin{remark}[\textbf{Restricted support function on the ball}]{\rm 
 An extended support function $H$ is also defined by its restriction onto the intersection of any space-like hyperplane with the interior of $\F$. This set can be identified
 with the open ball $B$ of $\mathbb{R}^d$. We do not need this here, but this is relevant for example to study the Minkowski problem \cite{BF}. This reference also consider F-convex sets as
 graphs on a space-like hyperplane, that we also do not consider here. But let us note for the next remark that restriction of extended support functions on $B$ are exactly convex functions on
 $B$. 
}\end{remark}

\begin{remark}[\textbf{A function not bounded on the boundary}]\rm{ It is tempting to say that if,
 for any $\Theta$,
$ \underset{\rho\rightarrow +\infty}{\mathrm{lim}}\frac{h_K(\rho,\Theta)}{\cosh (\rho)}$ 
is finite, then $K$ is contained in the future cone of
a point, taking the supremum for $\Theta$ of the limits. But this is false: there exist convex functions on the closed ball $\overline{B}$ (see the preceding remark), 
lower-semi continuous and unbounded. They are constructed in Lemma p.~870 of \cite{GKR}. In this reference there are also constructed  convex functions on the closed ball $\overline{B}$, 
lower-semi continuous, bounded, which attain no maximum.
} \end{remark}

\subsection{Polyhedral sets}\label{subsub back poly}

Let $p_i, i\in I$, be a discrete set of points of $\R^{d+1}$. Let us suppose that
for all $\eta\in\H^d$, $\operatorname{sup}_{i} \langle \eta,p_i\rangle_-$
is finite, and moreover that the supremum is attained. That is obviously not
always the case, as $-i e_{d+1}$ and $\frac{1}{i}e_{d+1}$ show for $i\in\mathbb{N}$. 
The function 
$$H(\eta)=\operatorname{max}_i \langle \eta,p_i\rangle_- $$
from $\F$ to $\mathbb{R}$ is clearly sublinear. From 
Lemma~\ref{lem: deter Hcs} there exists an F-convex set $K$ with support function $H$.
We call an F-convex set obtained in this way an \emph{F-convex polyhedron}. In particular 
$$K=\{x\in\R^{d+1}|  \langle x,\eta\rangle_- \leq \operatorname{max}_i \langle \eta,p_i\rangle_-  \}. $$
Without loss of generality, we suppose that the set $p_i,i\in I$ is minimal, in the sense
that if a $p_j$ is removed from the list, a different F-convex polyhedron is then obtained.
In particular, for any $i$ there exists $\eta$ with $H(\eta)=\langle \eta,p_i\rangle_-$, so
 $p_i\in\partial_s K$. Note that by Lemma~\ref{lem: acausal}, $p_i-p_j$ is space-like $\forall i,j$.
This last property is not a sufficient condition on the $p_i$ to define an F-convex polyhedron, as the example
$p_i=iv$ for any space-like vector $v$ and $i\in\mathbb{N}$ shows.

An F-convex polyhedron can be described more geometrically as
a ``future convex hull''. If $\mathcal{H}$ is a space-like hyperplane, we
denote by $\mathcal{H}^+$ its future side.
\begin{lemma}
 Let $K$ be an F-convex polyhedron as above.
 $K$ is the smallest F-convex set containing the $p_i$.
 
 Moreover,
 $$K=\cap \{\mathcal{H}^+ | p_i\in \mathcal{H}^+ \forall i\} .$$
\end{lemma}
\begin{proof}
 
Let $K'$ be an F-convex set containing the $p_i$.  
For any $\eta\in\F$
 $$H_{K'}(\eta)=\operatorname{sup}_{x\in K'}\langle x,\eta\rangle_- \geq \langle \eta,p_i\rangle_- $$
for all $i$ hence
$$H_{K'}(\eta) \geq \operatorname{max}_i \langle \eta,p_i\rangle_- =H(\eta) $$
hence $K\subset K'$. 
 Let $A=\cap \{\mathcal{H}^+ | p_i\in \mathcal{H}^+ \forall i\} $. 
 $K$ is an intersection of the future side of space-like hyperplanes (namely its support planes), 
 which all contains the $p_i$'s, hence
 $A\subset K$. By Lemma~\ref{lem: conv dans Fconv}, $A$ is an F-convex set, hence $K \subset A$ by the previous
property. 
\end{proof}

Let $K$ be an F-convex polyhedron as above. It gives a decomposition of $\H^d$ by sets 
$$O_i=\{\eta\in\H^d|H(\eta)=\langle\eta,p_i\rangle_-\}.$$

\begin{lemma}
 The $O_i$'s are convex sets and $O_i\cap O_j$ is contained in a totally geodesic hypersurface if not empty.
\end{lemma}
\begin{proof}
 Let us denote by $C(O_i)$ the cone over $O_i$ in $\F$. We have to prove that $C(O_i)$ is convex in $\R^{d+1}$.
 Let $\eta_1,\eta_2\in C(O_i)$. Then, for $t\in[0,1]$, as extended support functions are convex,
 $$H((1-t)\eta_1+t\eta_2)\leq (1-t)H(\eta_1)+tH(\eta_2)=\langle (1-t)\eta_1+t\eta_2,p_i\rangle_-\leq H((1-t)\eta_1+t\eta_2) $$
hence  $H((1-t)\eta_1+t\eta_2)=\langle (1-t)\eta_1+t\eta_2,p_i\rangle_-$ 
that means that $(1-t)\eta_1+t\eta_2\in C(O_i)$.

For any $\eta \in O_i\cap O_j$ we get $\langle \eta,p_i-p_j\rangle_-=0$ that is the equation of a 
time-like vector hyperplane. 
 \end{proof}

A part $F$ of  $O_i\subset \H^d$ is a \emph{$k$-face}, $k=0,\ldots,d$, 
if $k$ is the smallest integer such that $F$ can be written as an intersection of $(d+1-k)$ $O_j$.
A $0$-face is a \emph{vertex},  a $(d-1)$-face is a \emph{facet} and a $d$ face is a \emph{cell} $O_i$ of the decomposition
$\{O_i\}$.
Let $\eta\in\H^d$ and $\mathcal{H}$ be the support plane of $K$ with normal $\eta$. If 
$\eta$ belongs to the interior of a $k$-face $F$, it is easy to see that $\mathcal{H}\cap K$ 
does not depend on $\eta\in F$ but only on $F$.
The set $\mathcal{H}\cap K$ is called a \emph{$(d-k)$-face} of $K$. 
As an intersection of convex sets, the faces of $K$ are convex. By construction 
a $(d-k)$-face contains at least $(d-k+1)$ of the $p_i$. As the normal vectors of the 
hyperplane containing it span a $k+1$ vector space,  the  $(d-k)$-face is contained in 
a plane of dimension $(d-k)$, and is not contained in a plane of lower dimension.

A $0$-face is a \emph{vertex}, a $1$-face is an \emph{edge} and a $d$-face 
is a \emph{facet} of $K$.
The vertices are exactly the $p_i$.
An F-convex polyhedron must have vertices, but maybe no other $k$-faces as the example of the future cone of a point shows.

From Proposition~9.9 and Remark~9.10 in \cite{Bon05}, the decomposition given by the $O_i$ is locally finite
(each $\eta\in\H^d$ has a neighborhood intersecting a finite number of $O_i$).
Nevertheless the cells $O_i$ can have an infinite number of sides (see Figure~3.6 in \cite{Mar07} where 
 the lift of a simple closed geodesic on a punctured torus is drawn). In this case, the decomposition of
$\partial_s K$ into faces is not locally finite, for example a vertex can be the endpoint of an infinite number of 
edges.

We call an F-convex polyhedron $K$ 
 a \emph{space-like F-convex polyhedron} if the $O_i$ are compact convex hyperbolic polyhedra (each with finite number
of faces). Each vertex of the decomposition corresponds to a space-like facet of $K$, which is a compact convex
polyhedron. Moreover $\partial_s K$ is locally finite for the decomposition in facets.
It must have an infinite number of faces.

\begin{example}{\rm
Let $x\in\F$. Then the convex hull of $\Gamma x$ is a  space-like Fuchsian
convex polyhedron, because fundamental
 domains for $\Gamma$ gives a tessellation of $\H^d$ by compact convex polyhedra \cite{NP91}.
A dual construction consists of considering the orbit of a space-like hyperplane \cite{FF}.
}\end{example}

\begin{figure}[ht]
\begin{center}
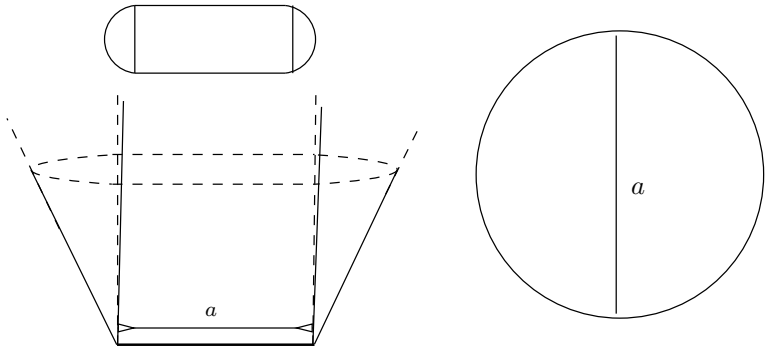
\end{center}
\caption{The elementary example in $d=2$.}\label{fig:wei}
\end{figure}


Let us now consider the case of an F-regular domain $K$.
From \cite{Bon05}, the image by the Gauss map $G$ of points of
$\partial_s K$ gives a decomposition of $\H^d$ by convex sets which are convex hulls of
points on $\partial_{\infty}\H^d$. Of course, if $p\in \partial_{s}K$, the support function 
$H$ of $K$ is equal to $\langle \cdot,p\rangle_-$ on $G(p)$. Hence $K$ is polyhedral in our sense
if $K$ has a discrete set of vertices (points $p$ of $\partial_s K$ such that $G(p)$ has non empty interior).
Following \cite{Bon05}, we call them \emph{F-regular domains with simplicial singularity}.

\begin{example}[\textbf{The elementary example}]\label{ex: elemntary1}{\rm
    Figure~\ref{fig:wei} is an  elementary examples of F-regular domains with simplicial singularity.
 The letters on the F-convex sets are edge-lengths. The letters on the 
cellulation of $\H^2$ are measures that will be introduced later. 
Actually we will call this example  (the union the the future cones 
of points on a space-like
segment) ``the'' elementary example, which is the simplest one, right after the future cone over a point.
 }
\end{example}

\subsection{Duality}\label{sub:duality}

The notion of duality has  interest in its own, but here it will only be used as a tool in the
proof of Proposition~\ref{lem: base}. See \cite{Ber10} for a previous introduction.
Let $A$ be a set which does not contain the origin. The \emph{dual} of $A$ is 
$$A^*=\{x\in\R^{d+1}| \langle x,a\rangle_- \leq -1, \forall a\in A\}. $$

It is immediate that $A^*$ is a closed convex set which does not contain the origin, 
that $A\subset A^{**}=:(A^*)^*$ and that $A\subset B$ implies $B^*\subset A^*$,  
see \cite[1.6.1]{Sch93}.
Note that as an F-convex set contains the future cone of its points,
it meets any future time-like ray
from the origin.

\begin{lemma}
Let $K$ be an F-convex set which does not contain the origin. Then
 $K^*$ is  contained in  $\overline{\F}^{\star }$
\end{lemma}
\begin{proof}
 Let $x\notin\overline{\F}$. 
Then there exists a $k \in K$ such that $\langle x,k\rangle_-\geq 0$, so $x\notin K^*$. 
As by definition $0\notin K^*$, we have $K^*\subset \overline{\F}^{\star }$. 
\end{proof}

In the compact case, duality is defined for convex bodies 
with the origin in their interior, that is equivalent 
to say that the Euclidean support functions are positive, and 
we get the fundamental property that the dual of the dual is the 
identity.

The lemma above says that in our case, even if $0\notin K$, 
we can take $K\nsubseteq \overline{\F}^{\star }$, and then 
$K^*\subset \overline{\F}^{\star }$ and 
$(K^*)^*\subset \overline{\F}^{\star }$, so $(K^*)^*\not= K$.
Actually the genuine analog to the compact case
is that the support function is negative.
By \eqref{eq: cone 0} this is equivalent to say that $K$ is contained in $\overline{\F}^{\star }$. 

\begin{lemma}
Let $K$ be an F-convex set contained in $\overline{\F}^{\star }$.  
Then $K^*$ is a  F-convex set and $(K^*)^*=K$.
\end{lemma}
\begin{proof}

$K^*$ is a closed convex set, so it is determined by its total support function. For $\eta\in \F$ let us consider 
$\tilde{H}_{K^*}(\eta)=\mathrm{sup}\{\langle \eta,x\rangle_- \forall x\in K^*  \}$. 
There exists $\lambda >0$ such that $\lambda \eta \in K$, so as
$\tilde{H}_{K^*}(\eta)=\frac{1}{\lambda}\tilde{H}_{K^*}(\lambda \eta)$ and by
definition $\langle \lambda \eta,x\rangle_-\leq -1$  $\forall x\in K^*$, 
$\tilde{H}_{K^*}$ has finite values on $\F$. 
As for two future vectors $u,v$ we have $\langle u,v\rangle_-<0$
and $K\subset \overline{\F}^{\star }$, 
 if $x\in K^*$ then $x+\overline{\F}\subset K^*$, so $\tilde{H}_{\C(x)}\leq \tilde{H}_{K^*}$, hence
$\tilde{H}_{K^*}$ is infinite outside of $\overline{\F}$. So
$$K^*=\{x\in\R^{d+1}| \langle x,\eta\rangle_-\leq \tilde{H}_{K^*}(\eta), \forall \eta \in \overline{\F}\} $$
and by Lemma~\ref{lem cl}  we have 
$$K^*=\{x\in\R^{d+1}| \langle x,\eta\rangle_-\leq \tilde{H}_{K^*}(\eta), \forall \eta \in\F\} $$
that says exactly that $K^*$ is an F-convex set.

To prove that $(K^*)^*=K$ one has to prove that
 $ (K^*)^*\subset K$.
Let $z\notin K$. There exists a support plane of $K$, orthogonal to 
some  $\eta\in\F$, which separates $z$ from $K$ \cite[1.3.4]{Sch93}.
Hence there exists $\alpha$ with $\langle z,\eta\rangle_- >\alpha$ and $\langle k,\eta\rangle_-<\alpha$ for all $k\in K$.
From \eqref{eq: cone 0}, $\alpha <0$.
On the other hand, for any $k\in K$, $\langle k,\eta\rangle_-\leq \alpha$,
which can be written $\langle k,\frac{\eta}{-\alpha}\rangle_-\leq -1 $, hence
$\frac{\eta}{-\alpha}\in K^*$. But  $\langle z,\frac{\eta}{-\alpha}\rangle_- >-1$,
so $z\notin (K^*)^*$.
\end{proof}

Let $K$ be an F-convex contained in $\overline{\F}^{\star }$.
The \emph{radial function}
of $K$ is the function from $\overline{\F}^{\star }$ to $\R^+\cup\{+\infty \}$ defined by
$$R_K(\eta):=\mathrm{inf}\{ s >0 \vert s \eta\in K\}.  $$
$R_K$ has always finite values on $\F$. If $K$ does not meet the light-like ray 
directed by $\ell$, then $R_K(\ell)=+\infty$.
In particular $\forall\eta$, $R_K(\eta)\eta\in\partial K$, $R_K$ is homogeneous of
degree $-1$ and 
$$K=\{\eta\in\overline{\F}^{\star }|R_K(\eta)\leq 1\} .$$

\begin{lemma}\label{lem:radial dual}
 Let $K$ be an F-convex set contained in $\overline{\F}^{\star }$. Then on $\overline{\F}^{\star }$,
the total support function $\tilde{H}_{K^*}$ of $K^*$ satisfies
$$\tilde{H}_{K^*}=\frac{-1}{R_K}. $$
\end{lemma}

As $K\subset \overline{\F}^{\star }$, $\tilde{H}_{K^*}$ and $\tilde{H}_{K}$ have finite non-positive values on $\overline{\F}$.

\begin{proof}
We define $X=\{x\in\overline{\F}^{\star }|\tilde{H}_{K^*}(x)\leq -1 \}.$
Let us first prove that $K=X$.

Let $x\in K\cap \F$. There exists $v\in K^*$ such that
$\tilde{H}_{K^*}(x)=\langle v,x\rangle_-$. But by definition of $K^*$, $\langle v,x\rangle_-\leq -1$ hence
$x\in X$. If $K\cap \partial \F$ is empty, we have $K\subset X$. If not, for $x\in K\cap \partial \F$
 the result is obtained from the previous case using Lemma~\ref{lem cl}.

Let $x\in X$. By definition of the support function, 
for any $v\in K^*$ we have $\langle x,v\rangle_-\leq \tilde{H}_{K^*}(x)$. On the other hand, as $x\in X$,
$\tilde{H}_{K^*}(x)\leq -1$ hence $x\in (K^*)^*=K$. We proved that $K=X$.

Let us suppose that there exists $x$
with  $\tilde{H}_{K^*}(x)>\frac{-1}{R_K(x)}$. As $\tilde{H}_{K^*}$ and 
$\frac{-1}{R_K}$ are $1$-homogeneous, one can find $\lambda>0$ such that 
$\tilde{H}_{K^*}(\lambda x)>-1 >\frac{-1}{R_K(\lambda x)}$. So 
$\lambda x \in K \setminus X$, that is impossible.
\end{proof}

\begin{example}\label{ex:cotenu light}{\rm
 Let $p\in\overline{\F}^{\star }$. Then $\C(p)^*$ is the intersection of $\overline{\F}$ with the
half-space $\{\langle x,p \rangle_-\leq -1\}$.

 The dual of $K(\H_t)$ is $K(\H_{1/t})$. More striking
 is the dual of $K(\H)+\C(e_{d+1})$.
It is not hard to see that on $\H^d$, $R_{K(\H)+\C(e_{d+1})}(\eta)=2\eta_{d+1}$, so
$H_{(K(\H)+\C(e_{d+1}))^*}(\eta)=\langle\eta,\eta\rangle_-/(-2\langle e_{d+1},\eta\rangle_-)
$, see Figure~\ref{fig:dual1}.
Note that on $\partial\F^{\star }$, $R_{(K(\H)+\C(e_{d+1}))^*}=1$. So 
$K\subset \F^{\star }$ does not imply $K^*\subset \F^{\star }$.
}\end{example}
\begin{example}{\rm
 The dual of a $\Gamma$ invariant F-convex set is a $\Gamma$ invariant F-convex set.
}\end{example}

\begin{figure}[ht]
\centering
\includegraphics[scale=0.2]{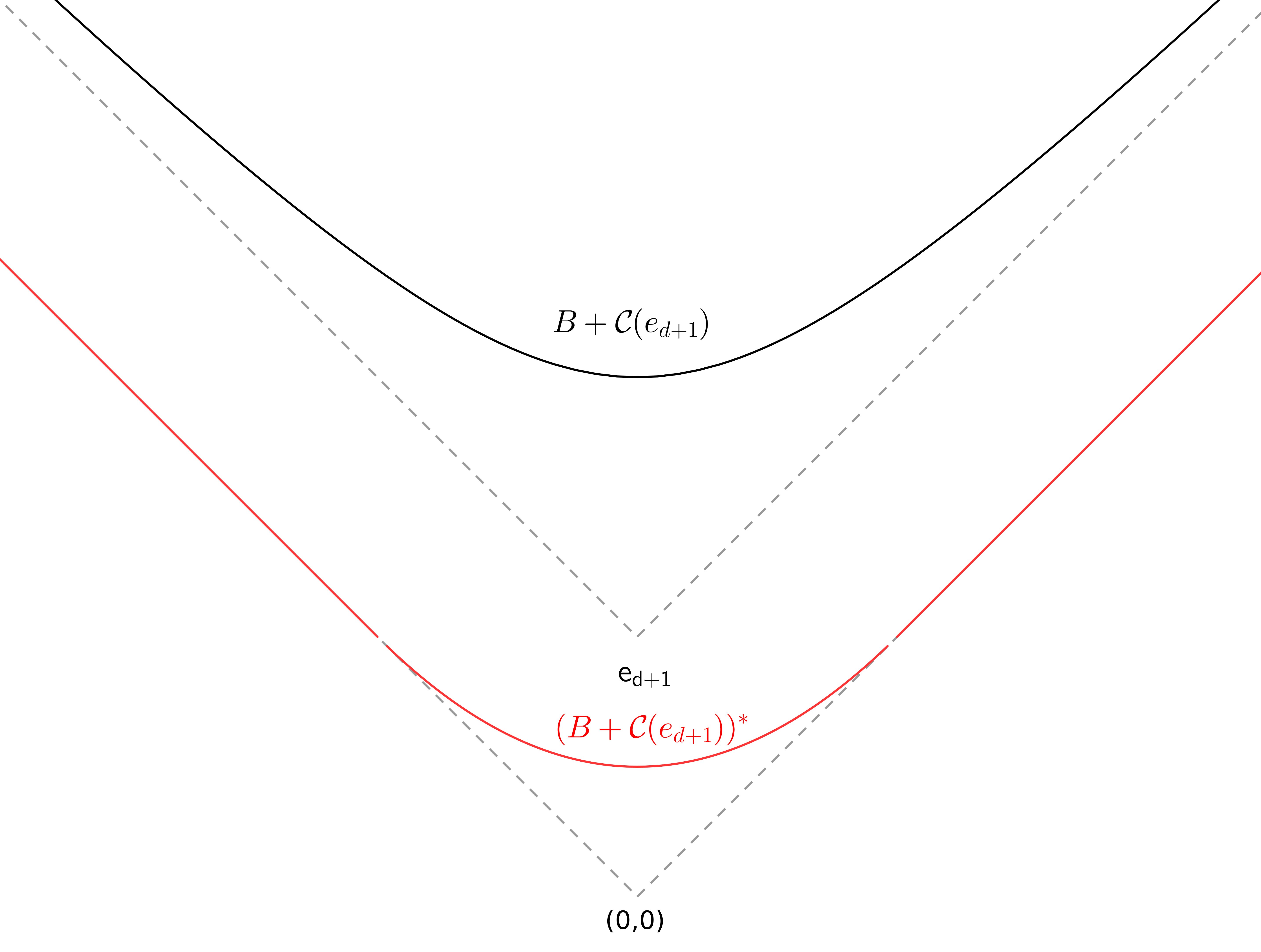}  
\caption{$K(\H)+\C(e_{d+1})$ and its dual.
\label{fig:dual1}} 
  \end{figure}

\subsection{First order regularity}\label{sub:first order}

\begin{lemma}\label{lem: grad}
Let $\eta\in \F$, $K$ be an F-convex set and $\mathcal{H}$ be the space-like support plane of $K$ with normal $\eta$. 
The intersection of $K$ and $\mathcal{H}$ is reduced to one point $p$ if and only if 
$H$ is differentiable at $\eta$. In this case $p$ is equal to the gradient $\grad_{\eta}H$ (for $\langle \cdot,\cdot\rangle_-$)
of $H$ at $\eta$. 
\end{lemma}
This result is a classical fact for convex bodies in the Euclidean space \cite{Sch93}, and 
the adaptation of the proof is straightforward. See \cite{FF}, where this property is 
checked for Fuchsian convex sets, but the group invariance does not enter the proof.

An F-convex set is said to be 
$C^k$ if $\partial_{s}K$ is a $C^k$ submanifold of $\R^{d+1}$.

\begin{lemma}\label{lemme:base 1}
Let $K$ be an F-convex set with support function $h_K$ and extended support function $H_K$.
 \begin{enumerate}[nolistsep,label={\bf(\roman{*})}, ref={\bf(\roman{*})}]
  \item\label{lem princ: KC1} $K$ is $C^1$ if and only if it has a unique support plane at each boundary point.
  \item\label{lem princ: linearity cond} If there exist $\eta,\eta'\in\F$ with $H_K(\eta+\eta')=H_K(\eta)+H_K(\eta')$ then
there exists $k\in K$ with two support planes. In particular $K$ is not $C^1$.
\item\label{lem princ: hC1}  $h_K$ is $C^1$ if and only if  $\partial_{s}K$ is strictly convex
(i.e.~the intersection of $K$ with any  space-like
support plane is reduced to a point).
\item\label{lem princ: hC1b} If $K$ is strictly convex, then $h_K$ is $C^1$ and $\partial_{s}K=\partial K$. 
\item\label{lem princ: hC1t} If $h_K$ is $C^1$ and $\partial_{s}K=\partial K$, then $K$ is strictly convex.
\item\label{lem princ: gauss1} If $K$ is $C^1$ then the  Gauss map is a well defined continuous map.
If 
$\partial_{s}K$ is strictly convex, 
then the Gauss map has a well defined inverse. 
 \end{enumerate}
\end{lemma}

It follows that if $\partial_{s}K$ is $C^1$ and strictly convex, the Gauss map 
is a continuous bijection (it is surjective by assumption). We will see in Subsection~\ref{sec: normal} that in this case it is actually a homeomorphism.

\begin{proof}
 
\ref{lem princ: KC1} is a general property of closed convex sets, see \cite{Sch93} p.~104.
Suppose that the hypothesis of \ref{lem princ: linearity cond} holds. Let 
$k,x,x'$ be points of $K$ with respectively $\langle k,\eta+\eta' \rangle_-=H_K(\eta+\eta')$,
$\langle x,\eta \rangle_-=H_K(\eta)$, $\langle x',\eta' \rangle_-=H_K(\eta')$. By assumption we get
$\langle k,\eta\rangle_-=\langle x,\eta\rangle_-+\langle x'-k,\eta'\rangle_-$, and 
$\langle x'-k,\eta'\rangle_-\geq 0$ so $\langle k,\eta\rangle_-\geq \langle x,\eta\rangle_-=H_K(\eta)$,
so $H_K(\eta)=\langle k,\eta\rangle_-$, that means that the support plane directed by $\eta$ contains $k$, which
is also in the support plane directed by $\eta+\eta'$. By \ref{lem princ: KC1} $K$ is not $C^1$.

From Lemma~\ref{lem: grad}   the intersection of $K$ with any of its space-like support planes is reduced to a point if and only if
$H_K$ is differentiable, that occurs if and only if $H_K$ is $C^1$, as $H_K$ is convex (see e.g.~\cite[p.~189]{HUL93}).
This is \ref{lem princ: hC1}, that implies  \ref{lem princ: hC1t}.  \ref{lem princ: hC1b} 
follows because if $K$ is strictly convex it has only space-like 
support planes due to \ref{base:ray} of Lemma~\ref{lem: base convexe}.

From \ref{lem princ: KC1}, the Gauss map is well defined if $K$ is $C^1$.
In this case, the volume form of $\partial_sK$ is continuous, 
and, by the identification induced by  $\langle \cdot,\cdot\rangle_-$
between  $d$-forms and vectors on $\mathbb{R}^{d+1}$, normal vectors to 
$\partial_sK$ are continuous.

If 
$K$ is strictly convex, the inverse of the Gauss map is clearly well defined. 
\end{proof}

\begin{example}{\rm
If $H$ is the extended support function of $\C(p)$, it is immediate that
$\grad_{\eta}H=p, \forall \eta\in \F$.
 It is important to not confuse $\partial_{s}\C(p)$ (the single point $p$) and $\partial \C(p)$
(the boundary of the cone), as $H$ is 
 $C^1$
but $\C(p)$ is not strictly convex.
}\end{example}

\subsection{Orthogonal projection}\label{subsub:brush}

Let $K$ be an F-convex set. We recall some facts which are contained in
 \cite{Bon05}, especially
Proposition~4.3. 
For any point $k\in K$, there exists a unique point $r(k)$ on $\partial K$
which is contained in the closure of the past cone of $k$ and which maximizes the Lorentzian distance. 
 The hyperplane orthogonal to $(k-r(k))$ is a support plane 
of $K$ at $r(k)$. In particular 
$r(K)=\partial_{s}K.$ The map $k\mapsto r(k)$ is the Lorentzian analogue of the Euclidean
orthogonal projection onto a convex set, see Figure~\ref{fig:projection}.
 The \emph{cosmological time}
of $K$ is $T(k)=d_L(k,r(k))$ for any $k\in K$. This is the analogue of the distance between a point
and a convex set in the Euclidean space. 

The \emph{normal field} of $K$
is the map $N : K\setminus\partial K\rightarrow \H^d$ defined by $N(k)=\frac{1}{T(k)}(k-r(k))$.
The normal field is well-defined and continuous, because equal to minus the Lorentzian gradient of $T$,
and $T$ is a $C^1$ submersion on the interior of $K$. 
$N$ is surjective by definition
of F-convex set. Note that  $\|\operatorname{grad}T\|_-=1$.

\begin{figure}[ht]
\begin{center}
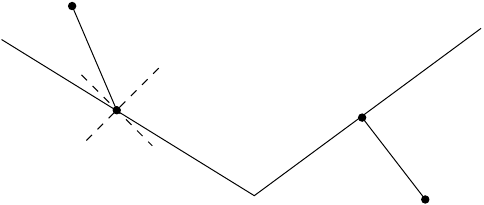
\end{center}
\caption{For the Euclidean metric, orthogonal projection onto a convex set is well-defined.
For the Lorentzian metric, orthogonal projection onto the complementary of a space-like convex set is well-defined.}\label{fig:projection}
\end{figure}

Let $\omega$ be a Borel set of $\H^d$ and $I$ be a non-empty interval of positive numbers 
(maybe reduced to a point). 
We introduce the following sets, see Figure~\ref{fig:brush}:
\begin{eqnarray*}
\ && K_I=T^{-1}(I), \\
\ && K_{I}(\omega)=K_I \cap N^{-1}(\omega), \\
\ && K(\omega)=G^{-1}_{K}(\omega).
\end{eqnarray*}
We have some immediate properties:
\begin{itemize}[nolistsep]
\item $K_t$ is  the boundary of the F-convex set $K+t K(\H)$. If $h$ is the  support function of $K$
and $H$ is its $1$-extension, then the support function of $K_t$ is
 $h-t$ and its extended support function is $H-t\|\cdot\|_-$.  It follows easily that $K$ and $K_t$ have the same light-like support planes at infinity. 
Finally,  $K_t$ has no light-like support plane.
\item For $t>0$, the restriction of the normal field to $K_t=T^{-1}(t)$ is equal to the Gauss map $G_{K_t}$
of $K_t$. In particular, $K_t(\omega)=G^{-1}_{K_t}(\omega)$ and $\partial K_t$ is a $C^1$ space-like hypersurface (it is actually $C^{1,1}$ \cite[4.12]{Bar05}).
\item The restriction of the normal field to $K_t$ is  a proper map \cite[4.15]{Bon05} (as $K_t$ is 
$C^1$, the Gauss map is well-defined).
Hence, if $\omega\subset \H^d$ is compact, $K_t(\omega)$ is compact.
\item The map $r:K\rightarrow \partial_{s}K$ is continuous \cite[4.3]{Bon05}.
\item If  $\omega\subset \H^d$ is compact then $K(\omega)$ is compact, by the two previous items and because 
$r(K_t(\omega))=K(\omega)$.
\end{itemize}

This allows to prove that $\partial_s K$ determines $K$ in the following sense.

\begin{lemma}
 Let $K$ be an F-convex set. Then
$$K=\bigcup_{k\in \partial_sK}\mathcal{C}(k).$$
\end{lemma}
\begin{proof}
 Because of \ref{base:cone} of Lemma~\ref{lem: base convexe},
 $\bigcup_{k\in \partial_sK}\mathcal{C}(k)\subset K.$
Because $r(K)=\partial_{s}K$, for any $p\in K$ there exists $k\in \partial_{s}K$ such that 
$p\in \mathcal{C}(k)$.
\end{proof}

\begin{figure}[ht]
\centering
\includegraphics[scale=0.1]{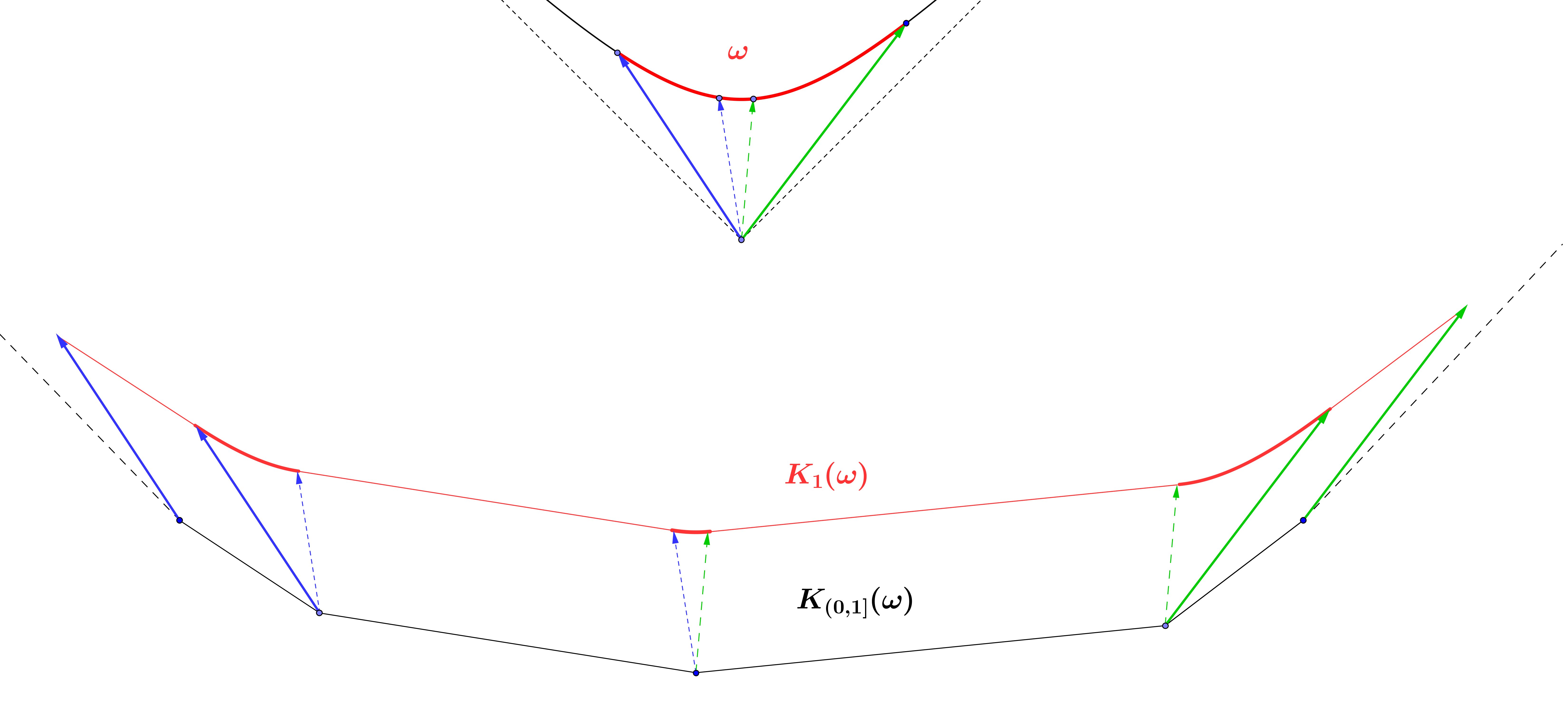}  
\caption{Notations, see Subsection~\ref{subsub:brush}.}\label{fig:brush}
  \end{figure}

\begin{remark}{\rm
 A property of F-convex sets is that the restriction of the normal field to $K_t$ is a proper map. Consider as $K$ the future of a line (an angle formed by the future of two
 light-like planes) in $\R^3$. The image of the Gauss map is a line $l$ in $\H^2$. The pre-image of any compact segment of $l$ by the Gauss map of any $K_t$  is not bounded.
}\end{remark}

\begin{remark}\label{eq:inv Komega}{\rm
Let  $K$ be a $\tau$-convex set.
It is easy to see \cite[4.10]{Bon05}
that, if $\gamma\in\Gamma_{\tau}$, with linear part $\gamma_0$, then $r\circ \gamma=\gamma\circ r$,
$N\circ \gamma=\gamma_0\circ N$, hence $T\circ \gamma =T$. It follows that
\begin{equation*}
 \gamma K_{(0,\epsilon]}(\omega)=K_{(0,\epsilon]}(\gamma_0\omega).
\end{equation*}
} 
\end{remark}

\begin{lemma}\label{lem: propr tau cvxe} Let $\tau$ be a cocycle 
and let $h_{\tau}$ be the support function of $\overline{\Omega_{\tau}}$ (see Example~\ref{ex: domain}).
\begin{enumerate}[nolistsep,label={\bf(\roman{*})},ref={\bf(\roman{*})}]
 \item\label{item: def gamma inv} 
An F-convex set $K$ which is (setwise) invariant for the action of $\Gamma_{\tau}$ 
is contained in $\overline{\Omega_{\tau}}$.
\item\label{item: infinite value supp} All $\tau$-F-convex sets have the same 
light-like support planes at infinity
than $\overline{\Omega_{\tau}}$.
\item\label{item: space-like support} A $\tau$-F-convex set contained in  $\Omega_{\tau}$
has only space-like support planes.
\item\label{item: bord omega} Let $K$ be a $\tau$-F-convex set. 
If $ K \cap \partial \Omega_{\tau}\not=\emptyset$,
then  $ K \cap \partial_s \Omega_{\tau}\not=\emptyset$.
\item\label{item: sup dedans} Let $h$ be the support function of a $\tau$-F-convex set $K$.
If $h<h_{\tau}$, then $K\subset \Omega_{\tau}$.
\end{enumerate}
\end{lemma}
\begin{proof}
Let $K$ as in \ref{item: def gamma inv}  with extended support function $H$, 
 and let $H_{\tau}$ be the extended support function
of $\overline{\Omega_{\tau}}$.
As $H-H_{\tau}$ is $0$-equivariant,  its restriction to $\H^d$ reaches a  minimum $a$ and a maximum $b$.
Hence $H_{\tau}+a\|\cdot\|_-\leq H \leq H_{\tau}+b\|\cdot\|_-$, so clearly 
$H$ and $H_{\tau}$ have the same limit on any path $\ell+t(\eta-\ell)$. From 
 Lemma~\ref{lem cl}, both sets have the same light-like support planes at infinity. 
(This proves \ref{item: infinite value supp}  if $K$ is a $\tau$-convex set.)
In particular $K$ is contained in the intersection of the future side of those planes, 
but this intersection is precisely $\Omega_{\tau}$ \cite[Corollary 3.7]{Bon05},
so $K\subset\overline{\Omega_{\tau}}$.

\ref{item: space-like support} We know from Lemma~\ref{lem: base convexe} that $K$ has 
no time-like support plane.
  Let us suppose that $K$ has a light-like support plane $L$, and let $x\in K\cap L$. 
  Then by \ref{item: infinite value supp}
$L$
is a support plane at infinity of $\Omega_{\tau}$, but $x\in \Omega_{\tau}$ so
$L$ is a support plane
of $\overline{\Omega_{\tau}}$. In particular,  $x\in\partial\Omega_{\tau}$, that is impossible 
as $K$ is supposed to be in $\Omega_{\tau}$, which is open.

\ref{item: bord omega} Suppose that  $ K \cap \partial_s \Omega_{\tau}=\emptyset$. 
By cocompactness, $H_{\tau}-H$ (the extended support functions of $\overline{\Omega_{\tau}}$ and $K$)
is bounded from below by a positive constant $c$. So $S_{c/2}$, the level set of the cosmological time
of $\Omega_{\tau}$ for the value $c/2$, contains $K$. But $S_{c/2}$ has no light-like support plane, so by
\ref{item: infinite value supp},  $S_{c/2}\cap \partial \Omega_{\tau}=\emptyset$, hence $ K \cap \partial \Omega_{\tau}=\emptyset$.

\ref{item: sup dedans} If $h<h_{\tau}$, then $ K \cap \partial_s \Omega_{\tau}=\emptyset$
and the result follows from \ref{item: bord omega}.
\end{proof}

\begin{remark}\rm{
 Lemma~\ref{lem: deter Hcs} and Lemma~\ref{lem: propr tau cvxe}
imply that there exists a $\tau$-equivariant convex function $H_{\tau}$ such that, for any  $\tau$-equivariant convex function $H$, then $H\leq H_{\tau}$.
}\end{remark}

\begin{example}\label{rem: restr fuchs}\rm{
 Let  $h$ be the support function 
of a $\Gamma$ invariant F-convex set $K$. 
 Lemma~\ref{lem: propr tau cvxe}  says that $K\subset \overline{\F}$.
Suppose that $K\not=\overline{\F}$.
From Lemma~\ref{lem: petit o}, $K\subset \overline{\F}^{\star }$, 
and by Lemma~\ref{lem: propr tau cvxe}, $K\subset \F$.

}
\end{example}

\subsection{The normal representation}\label{sec: normal}

Let $\O$ be an open set of $\H^d$ and let $h:\O\rightarrow \R$ be a $C^1$ map with $1$-extension $H$. We call \emph{normal representation}
of $h$ the map $ \chi$ from $\O\rightarrow \R^{d+1}$ defined by 
$\chi(\eta)=\grad_{\eta}H $, that is, for any space-like vector $v$,
\begin{equation}\label{eq:norm1}\langle \chi(\eta),v\rangle_-=D_{\eta}H(v) \end{equation}
and by Euler's Homogeneous Function Theorem
\begin{equation}\label{eq:norm2}\langle \chi(\eta),\eta\rangle_-=H(\eta).\end{equation}

The equation above defines a space-like hyperplane with normal $\eta$ containing the point $\chi(\eta)$.
Lemma~\ref{lem: grad} says that if $H$ is the support function of an F-convex set $K$, then
 $\chi(\H^d)=\partial_{s}K$.
If an F-convex set $K$ is $C^1$ and $\partial_{s}K$ is strictly convex, 
we know from Lemma~\ref{lemme:base 1} that the Gauss map is a continuous bijection. 
But from \ref{lem princ: hC1} 
its support function has normal representation, which is clearly the inverse of the Gauss map, which
is then a homeomorphism.

Now let $h:\O\rightarrow \R$ be $C^2$. Then $\chi$ is  $C^1$. Differentiating \eqref{eq:norm2} 
 in the direction of a space-like vector $v$, and using \eqref{eq:norm1}, 
we get that $\langle \eta,D_{\eta}\chi(v)\rangle_-=0 $, so if $\eta$ is a regular point, 
the space-like hyperplane $\langle \cdot ,\eta\rangle_-=H(\eta)$ is tangent to $\chi(\O)$ at $\chi(\eta)$. 
The differential $S^{-1}$ of $\chi$ is called the \emph{reverse shape operator}, because $\chi$ is the inverse of the Gauss map, and the differential of the 
Gauss map is the shape operator. $S^{-1}$
is considered as an endomorphism of $T_{\eta}\H^d$, by identifying this space with 
the support plane of $\chi(\eta)$ with normal $\eta$. This
 allows to
define the \emph{reverse second fundamental form} of $H$: $\forall X,Y \in T_{\eta}\H^d$,
\begin{equation}\label{eq:reversed so}\II^{-1}(X,Y):=\langle S^{-1}(X),Y\rangle_-=\operatorname{Hess}_{\eta}H(X,Y)
 \stackrel{\eqref{eq:II-1}}{=} \nabla^2 h (X,Y)- h g(X,Y). \end{equation}
As  $\II^{-1}(X,Y)=\hess_{\eta}H_K(X,Y)$,  $\II^{-1}$ is symmetric and the eigenvalues $r_1,\ldots,r_d$
of $S^{-1}$ are real.
They are the \emph{principal radii of curvature} of $h$.
If they are not zero, the Gauss map is a $C^1$ diffeomorphism, and then the $r_i$ are the inverse of the principal curvatures of the space-like hypersurface
$\chi(\O)$. 

\subsection{Second order regularity}

 An F-convex set $K$ is called $C^2_+$ if $\partial_sK$ is
$C^2$ and its Gauss map is a $C^1$ diffeomorphism. This implies that  $\partial_{s}K$ is strictly convex,
but $K$ is not necessarily strictly convex, as can be seen on Figure~\ref{fig:dual1}.

\begin{lemma}\label{lem: basebis1}
 \begin{enumerate}[nolistsep,label={\bf(\roman{*})}, ref={\bf(\roman{*})}]
 \item\label{lem princ: radneg}  If $h_K$ is $C^2$, then the radii of curvature are real non-negative numbers.
\item\label{lem princ: c2 deter}  If a $C^2$ function $h$ on $\H^d$ satisfies \begin{equation}\label{eq:hess supp}
(\nabla^2h- h g )\geq 0.
\end{equation} then it is the support function of an F-convex set. 
\item\label{lem princ: radiipos}  If $K$ is $C^2_+$ then $h_K$ is $C^2$, the radii of curvature are positive (hence equal to 
the inverses of the 
principal curvatures).
 \end{enumerate}
\end{lemma}
\begin{proof}
 
We already know that the eigenvalues
of $S^{-1}$ are real. As $H_K$ is convex, its Hessian is positive semidefinite, so \ref{lem princ: radneg} holds.

Let $h$ be a function as in \ref{lem princ: c2 deter}. Then its one homogeneous extension to 
$\R^{d+1}$ has a positive semidefinite Hessian, and \ref{lem princ: c2 deter} follows by Lemma~\ref{lem: deter Hcs}.

Let us prove \ref{lem princ: radiipos}. If the Gauss map $G$ is a $C^1$ diffeomorphism, 
its inverse is the normal representation $\chi$,
 which is then $C^1$. As $\chi$ is the gradient of 
$H_K$, $H_K$ is $C^2$. Moreover the shape operator (the differential of the Gauss map) is the inverse of the reverse
shape operator, and both are positive definite, because they are both  positive semidefinite and invertible.

\end{proof}

\begin{proposition}\label{lem: base}
Let $K$ be an F-convex set with support function $h_K$.
   If $h_K$ is $C^2$ and the principal radii of curvature are positive, i.e.~$(\nabla^2 h_K - h_Kg) >0$, then $K$ is $C^2_+$.
\end{proposition}

This proof of this proposition is the content of the next subsection.

\begin{corollary}\label{cor:base}
 Let $K$ be an F-convex set with support function $h_K$.
 \begin{enumerate}[nolistsep,label={\bf(\arabic{*})}, ref={\bf(\arabic{*})}]
 \item\label{lem princ: c2+ deter}  If a $C^2$ function $h$ on $\H^d$ satisfies \begin{equation}\label{eq:hess supp +}
(\nabla^2h- h g )> 0.
\end{equation} then it is the support function of a $C^2_+$ F-convex set. 
\item\label{lem princ: epsilon} If $h_K$ is $C^2$, then for any $\epsilon >0$, $K+\epsilon K(\H)$ is $C^2_+$.
 \end{enumerate}
\end{corollary}
\begin{proof}

\ref{lem princ: c2+ deter} follows from Proposition~\ref{lem: base}  and \ref{lem princ: c2 deter} of Lemma~\ref{lem: basebis1}.
Let $h_K$ be $C^2$. Then $(\nabla^2h- h g )\geq 0$, and for any $\epsilon>0$, the support function of $K+\epsilon K(\H)$
is $h_K-\epsilon$ and $(\nabla^2(h-\epsilon)- (h-\epsilon) g )>0$, and \ref{lem princ: epsilon}
follows from \ref{lem princ: c2+ deter}

\end{proof}

\begin{remark}\label{remark:-h}{\rm
 Let $h$ be a $C^2$ function on $\H^d$ such that $(\nabla^2h- h g )\leq 0$, with $1$-extension $H$.
 Then by the proposition above,
$-H$ is the extended support function of an F-convex set $K$, and $\grad (-H)=-\grad H$ is the normal representation of $\partial_s K$.
 Hence $\grad H$  is the normal representation of $\partial_s (-K)$, and 
$-K$ is a P-convex set, see Remark~\ref{rem:Pconvex}. 
}\end{remark}

\begin{example}{\rm
 The future cone of a point is at the same time an F-convex polyhedron and an F-convex set with $C^2$ support
function. This is the only case where it can happen.
}\end{example}

\begin{example}[\textbf{F-convex sets not contained in the future cone of a point}] \label{ex alpha}{\rm 
 Let us define, for $x\in\H^d$, $\rho=\rho(x)$ the hyperbolic
distance to $e_{d+1}$, and

$$F_{\alpha}^+(x)=\cosh (\rho)^{\alpha},   \alpha\geq 1, F_{\alpha}^-(x)=-\cosh (\rho)^{\alpha},  -1\leq \alpha\leq  1$$
whose degree one extensions on $\F$ are respectively
$$\frac{x_{d+1}^{\alpha}}{(-\langle x,x\rangle_-)^{(\alpha-1)/2}},-\frac{x_{d+1}^{\alpha}}{(-\langle x,x\rangle_-)^{(\alpha-1)/2}}.$$

As $\cosh \rho$ is the restriction to $\H^d$ of the map 
$x\mapsto x_{d+1}$, using \eqref{restr hessien} and the fact that
for $f:\mathbb{R}\rightarrow\mathbb{R}$ one has 
$\nabla^2(f\circ \rho)=(f'\circ \rho)\nabla^2\rho+(f''\circ\rho)\mathrm{d}\rho\otimes \mathrm{d}\rho $,
we compute easily that
\begin{equation}\label{eq: hessien rho}\nabla^2 \rho = 
 \frac{\cosh \rho}{\sinh \rho}\left ( g-\mathrm{d}\rho\otimes \mathrm{d}\rho \right )
\end{equation}

and finally
\begin{equation}\label{eq:nabla cosh}\nabla^2\cosh^\alpha \rho=[\alpha\cosh^{\alpha} \rho] g + [\alpha(\alpha-1)\cosh^{\alpha-2}\rho\ \sinh^2 \rho]
\mathrm{d}\rho\otimes \mathrm{d}\rho.\end{equation}

It follows that $(\nabla^2-g)(F^+_{\alpha})$ and $(\nabla^2-g)(F^-_{\alpha})$ are semi-positive definite, 
hence
$F_{\alpha}^+$ and $F_{\alpha}^-$ are support functions of  F-convex sets. 
Note that $F^-_0$ is the support function of $K(\H)$, and $F^-_1$ and $F^+_1$ are support functions
of the future cones of $e_{d+1}$ and $-e_{d+1}$ respectively.
From Lemma~\ref{lem:h infini}, for $\alpha>1$, $F^+_{\alpha}$ has no light-like support plane at infinity.
See Figure~\ref{fig:dual2}.

\begin{figure}[ht]
\centering
\includegraphics[scale=0.2]{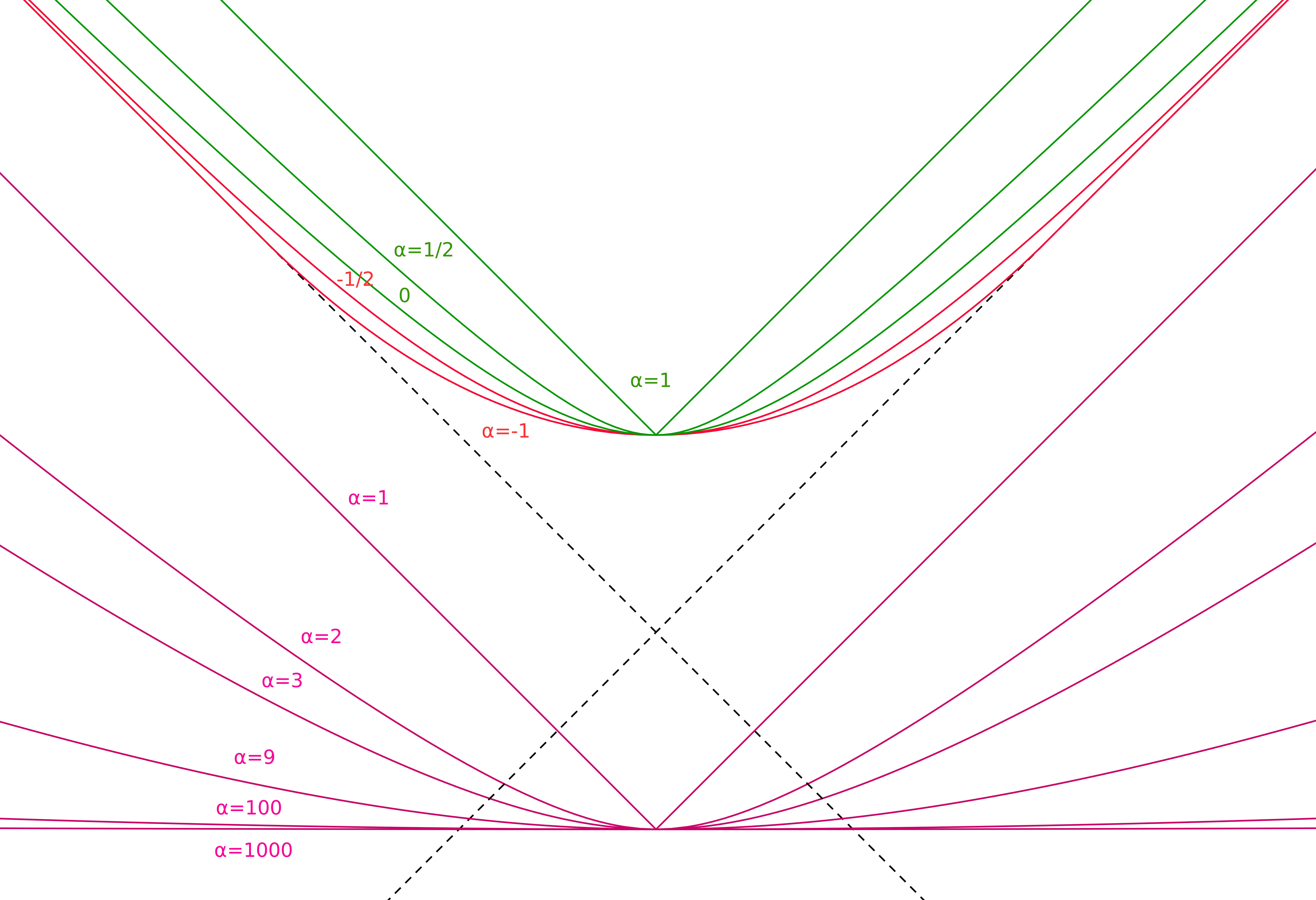}  
\caption{To Example~\ref{ex alpha}.
\label{fig:dual2}} 
  \end{figure}
}\end{example}

\subsection{Proof of Proposition~\ref{lem: base}}

As $h_K$ is $C^2$, $H_K$ is $C^2$, the normal representation is $C^1$, 
and this is a regular map as the principal radii of curvature (the eigenvalues of its differential) are positive, so 
 $\partial_s K$ is $C^1$. Moreover as the Gauss map is the inverse of the normal representation, it is a $C^1$ diffeomorphism.
It remains to prove the non-trivial result that $\partial_sK$ is actually $C^2$.

First suppose that $K$ is contained in the future cone of a point. Up to a translation, we can consider that this is
the future cone of the origin. From Lemma~\ref{lem:radial dual} and the properties of $H_K$,  $K^*$ is $C^2$.
At the point $R_{K^*}(\eta)\eta$ of the boundary of $K^*$, the Gauss map is $\chi(\eta)/(\sqrt{-\langle \eta,\eta\rangle_-})$,
so a $C^1$ diffeomorphism, and then $K^*$ is $C^2_+$. By \ref{lem princ: radiipos} of Lemma~\ref{lem: basebis1},
$h_{K^*}$ is $C^2$ and its principal radii of curvatures are positive. Repeating the argument,
we get that $K=(K^*)^*$ is $C^2$.

Now suppose that $K$ is not contained in any future cone of a point. We will need the following:

\emph{Fact: For any $k\in \partial_s K$, there exists a neighborhood $V$ of $k$ in $ \partial_s K$ and 
an F-convex set $K_V$ such that: $V$ is a part of the boundary of $K_V$, $K_V$ is contained in the future cone of a point, 
has $C^2$ support function and positive principal radii of curvature.}

From the preceding argument, it will follow that the boundary of $K_V$ is $C^2$, hence each point of $\partial_s K$
has a $C^2$ neighborhood, hence $K$ is $C^2$. Let us prove the fact.
We need the following local approximation result.

\begin{lemma}\label{lem: approx}
 Let $K$ be an F-convex set with support function $h_K$, $\omega\subset \H^d$ be compact and  $\epsilon >0$. 
Then there exists an F-convex set $A(K,\omega,\epsilon)=:A$ with support function $h_A$ such that
\begin{itemize}[nolistsep]
\item $A$ is $C^2_+$,
 \item $\mathrm{sup}_{\eta\in\omega}|h_K(\eta)-h_{A}(\eta)|<\epsilon,$
\item $A$ is contained in the future cone of a point.
\end{itemize}
\end{lemma}
\begin{proof}[Proof of Lemma~\ref{lem: approx}]
The argument is an adaptation of \cite{Fir74}.
The intersection of  $K_{\epsilon/4}(\omega)$ with 
$\bigcup_{k\in \partial_s K} \C(k)$
is an open covering of the compact set $K_{\epsilon/4}(\omega)$ (see Subsection~\ref{subsub:brush}).
 From it we get a finite covering
 $\bigcup_{i=1}^N \C(k_i)$.
Let $E$ be the convex hull of  $\cup_i\C(k_i)$. It
has extended support function $H_E(x)=\mbox{max}_{i=1\cdots N} \langle x,k_i\rangle_-$, and is 
an F-convex set due to Lemma~\ref{lem: deter Hcs}.
 As $k_i\in K$, $\C(k_i)\subset K$ hence $E\subset K$ and $H_E\leq H_K$ on $\omega$.
By construction $K_{\epsilon/4}(\omega)\subset E$ hence $H_K-\epsilon/4 \leq H_E$ on $\omega$, and
finally
$\mbox{sup}_{x\in\omega}|H_K(x)-H_{E}(x)|<\epsilon/3$.

The statement of the lemma and the computation above are true up to translations. 
 We imply that we performed a translation such that $k_1,\ldots,k_N$ are contained in the past
 cone of the origin, so $\langle x,k_i\rangle_->0$ for all $x\in \F$.
The functions
$$H_p(x)=\left(\sum_{i=1}^N \langle x,k_i\rangle_-^p/N\right)^{\frac{1}{p}}  $$
are extended support functions of F-convex sets by Minkowski inequality and Lemma~\ref{lem: deter Hcs}. $H_p$
is  clearly $C^2$ (actually analytical), and $H_p(x)$ converges to $H_E(x)$ 
when $p\rightarrow \infty$. Let us choose $p=p_{\epsilon}$ such that $\mbox{sup}_{x\in\omega}|H_p(x)-H_{E}(x)|<\epsilon/3$.
 From Lemma~\ref{lem cl}, the extension $\tilde{H}_p$ of $H_p$ to $\partial \F$ is a continuous function
with finite values.
Let $\F(1)$ be the subset of $\overline{\F}$ made of vectors with last coordinate equal to 
one. It is a compact set and let $M$ be the maximal value for $\tilde{H}_p$.
By homogeneity, we have, $\forall \eta \in \F$, 
$$H_p(\eta)=\eta_{d+1}H_p(\eta/\eta_{d+1})\leq M \eta_{d+1}=M\langle \eta, -e_{d+1}\rangle_-,$$
hence
the F-convex set supported by $H_p$ is contained in $\C(-Me_{d+1})$.
If $h_p$ is the restriction of $H_p$ to $\H^d$, we define
$h_A:=h_p-\epsilon/3$. Then:
\begin{itemize}[nolistsep]
\item from \ref{lem princ: c2 deter} of Lemma~\ref{lem: basebis1} $h_A$ is the support function of an F-convex set $A$,
\item $A$ is contained in the future cone of a point,
 \item hence \ref{lem princ: epsilon} of Corollary~\ref{cor:base} holds, and $A$ is a $C^2_+$ F-convex set,
 \item finally
 $\mathrm{sup}_{\eta\in\omega}|h_K(\eta)-h_{A}(\eta)|<\epsilon$,
\end{itemize}
so $A$ is the aimed $A(K,\omega,\epsilon)$.

\end{proof}

Let $k\in \partial_s K$, $G_K(k)\in \omega_0\subsetneq \omega$,
where $\omega_0$ and $\omega$ are two compact subsets of $\H^d$, and 
$V=\chi(\stackrel{\circ}{\omega_0})$. Let us also introduce a bump function $\psi\in C^{\infty}(\H^d)$, 
$0\leq \psi \leq 1$,
with $\mathrm{supp} \psi\subset \omega$ and $\psi=1$ in $\stackrel{\circ}{\omega_0}$. 
Let $\epsilon >0$, $A(K,\omega,\epsilon)$ be the F-convex set given by Lemma~\ref{lem: approx}, and let 
$h_{\epsilon}$ be its support function. We proceed as in \cite{Gho02} for example.
The function 
$$\overline{h}=\psi h_K+(1-\psi) h_{\epsilon} $$
is a $C^2$ function on $\H^d$. It satisfies \eqref{eq:hess supp +}
on $\stackrel{\circ}{\omega_0}$ and outside of $\omega$. On the remaining part of $\H^d$ we have
$$(\nabla^2-g)\overline{h}=
\psi (\nabla^2-g)h_K+(1-\psi)(\nabla^2-g)h_{\epsilon}+(h_K-h_{\epsilon})\nabla^2\psi$$
$$+\mathrm{d}\psi\otimes \mathrm{d}(h_K-h_{\epsilon})+\mathrm{d}(h_K-h_{\epsilon})\otimes \mathrm{d}\psi.$$
We have $\psi (\nabla^2-g)h_K>0$ 
and $(1-\psi)(\nabla^2-g)h_{\epsilon}>0$. Moreover the choice of $\epsilon$ is independent of 
$\psi$. 
On one hand $(h_K-h_{\epsilon})$ is arbitrarily small by
 Lemma~\ref{lem: approx}. On the other hand, as $h_K$ and $h_\epsilon$ are both $C^1$, they are arbitrarily close 
 in $C^1(\stackrel{\circ}{\omega})$ (this is true for the convex $1$-homogeneous extensions of the functions 
 on a suitable subset of $\F$ \cite[25.7]{Roc97}). 
So $(\nabla^2-g)\overline{h}>0$ for a well chosen $\epsilon$.
As $\overline{h}=h_{\epsilon}$ outside of a compact set, $\overline{h}$ is the support function
of an F-convex set contained in the future cone of a point, which is the wanted $K_V$.

Proposition~\ref{lem: base} is proved.

\subsection{The $d=1$ case}\label{sub:1d}

The relations between an F-convex set and its support function can be made more explicit in the case of the plane.
Let $h$ be $C^1$ and let us use
the coordinates $(r\sinh \rho,r\cosh \rho)$ on $\F$. We have
$$H(r\sinh\rho,r\cosh\rho)=rH(\sinh\rho,\cosh\rho)=:rh(\rho).$$
Computing the gradient in those coordinates, we can write $\partial_s K$ as a  curve in terms of the 
support function, that has a clear geometric meaning, see Figure~\ref{fig:nabla}:
\begin{equation}\label{eq:eq courbe}
c(\rho)=h'(\rho)\binom{\cosh \rho}{\sinh \rho} -h(\rho)\binom{\sinh \rho}{\cosh \rho}. 
\end{equation}
Note that if $h$ is $C^2$ then $c'(\rho)=(h''(\rho)-h(\rho))\binom{\cosh \rho}{\sinh \rho}$, so the curve 
is indeed space-like, and regular if $h''-h\not= 0$.

\begin{figure}
\centering 
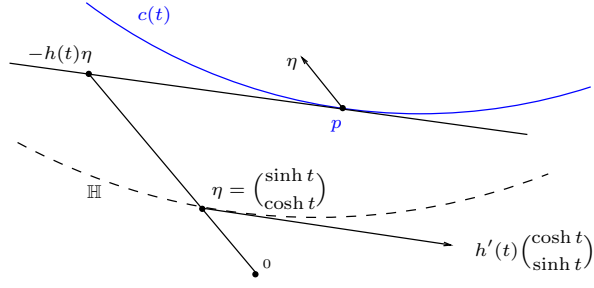  
\caption{Planar case: recovering the curve from the
support function (subsection~\ref{sub:1d}).  
\label{fig:nabla}} 
  \end{figure}

From Corollary~\ref{cor:base}, a $C^2$ function $h:\R\rightarrow\R$ is the support function of an F-convex curve 
(F-convex set in the plane) if and
only if $h''-h\geq 0$. If $h''-h>0$, then the curve has finite curvature. It will be useful to have a 
more general characterization of
convexity. The compact analogue of the lemma below appeared in 
\cite{Kal74}.

\begin{lemma}\label{lem:1d}
 A real function is the support function of an F-convex curve if and only if it is
continuous and satisfies, for any real $\alpha$,
\begin{equation}\label{eq:hyp convex}
 h(\rho+\alpha)+h(\rho-\alpha)\geq 2 \cosh (\alpha)  h(\rho).
\end{equation}
\end{lemma}
\begin{proof}
The condition is necessary due to Lemma~\ref{lem:convexite radiale}.
Now let $h$ be a continuous function and let $H$ be its homogeneous extension. We suppose that $H$
is not convex on $\F$.

\emph{Fact:} There exists unitary $u$ and $v$ such that $H(u+v)>H(u)+H(v)$.

If the fact is true, we see from \eqref{eq:conv d1} that \eqref{eq:hyp convex} is false. 
Now let us prove the fact. We know that there exists $u,v\in\F$ and $0<\lambda<1$ such that
$$H(\lambda u+(1-\lambda)v)>\lambda H(u)+(1-\lambda)H(v). $$
By continuity, this holds in a neighborhood of $\lambda$. Up to a reparametrization of $\lambda$, we can
consider that this holds for any $0<\lambda<1$. Then it suffices to take  $\lambda=\frac{\|v\|_-}{\|u\|_-+\|v\|_-}$
and multiply both sides of the equation above by $\frac{\|u\|_-+\|v\|_-}{\|u\|_-\|v\|_-}$.
\end{proof}

\begin{remark}[\textbf{Osculating hyperbola}]{\rm 
 We can give a geometric interpretation of the radius of curvature for F-convex curves in the plane.
Computations are formally the same as in the Euclidean case, see e.g.~the first pages of \cite{Spi79}, so
we skip them.
Let $\gamma$ be the boundary of a strictly convex F-convex set in the Minkowski plane, seen as a curve parametrized
by arc length (for the 
induced Lorentzian metric). Let $p_1,p_2,p_3$ be three points on $\gamma$, with
$p_2$ between $p_1$ and $p_3$. There exists a unique upper hyperbola passing through those points (the center of this 
hyperbola is the intersection between the two time-like lines passing through the middle, and orthogonal to,
the space-like segments $p_1p_2$ and $p_2p_3$). When $p_1$ and $p_3$ approaches $p_2$, the hyperbolas
converges to a hyperbola with radius $\displaystyle \frac{1}{\|\gamma'' \|_-}$.
Now let $c$ as in \eqref{eq:eq courbe}. We have $c=\gamma\circ s$, with $s$ the arc length of $c$:
$$s(\rho)=\int_0^{\rho} h''(t)-h(t)\mathrm{d}t $$
and $\gamma$ parametrized
by arc length. A computation shows that $\displaystyle\langle \gamma'',\gamma''\rangle_-=-\frac{1}{(h''-h)^2}$.

}
\end{remark}

\subsection{Hedgehogs}\label{subsec:hedg}

Both  spaces of support functions of F-convex sets and of
P-convex set of $\R^{d+1}$ form a convex cone in the space of
continuous functions on $\H^d$. They span a vector space, the vector space of differences of support functions. 
Such functions were known for a long under different names (see Remark~\ref{rem:herissons} and Remark~\ref{rem:mess pol}) and called  \emph{hedgehogs} since \cite{LLR}.

To simplify we restrict to the case of $C^2$ support functions.
It follows from the classical theory of difference of convex functions 
 that the vector space spanned by $C^2$ support functions is the whole space of $C^2$ functions on $\H^d$ \cite{ale12,har59,hu85,ves87}.
In the classical compact case, this is straightforward by compactness,
writing any $C^2$ function $h$ on $\mathbb{S}^d$ as $(h+r)-r$ for any sufficiently large constant $r$.
 The same
argument occurs in the quasi-Fuchsian case (see Lemma~\ref{lem:tau herisson} below). This also gives another natural 
motivation to introduce hedgehogs: level surfaces of the cosmological time 
outside of an F-convex set
are hedgehogs. Moreover, if $\tau$ is a cocycle, the following lemma says that all the $C^2$ $\Gamma_\tau$ 
invariant hedgehogs are obtained in this way. We will call such functions \emph{($C^2$) $\tau$-hedgehog}.
See Figure~\ref{fig:herissons}.

\begin{lemma}\label{lem:tau herisson}
 Let $h$ be a $C^2$ $\tau$-hedgehog. There exists positive constants
$c_1$ and $c_2$ such that $h-c_1$ bounds a $\tau$-F-convex set and $h+c_2$
bounds a $\tau$-P-convex set. For any positive constant $c$, $h+c_2+c$ (resp. $h-c_1-c$) bounds a $C^2_+$ $\tau$-F-convex (resp. $\tau$-F-convex). 
\end{lemma}
\begin{proof}
From Lemma~\ref{lem: hess gamma inv}, as $\H^d/\Gamma$ is compact, we get the constants $c_1$ and $c_2$ such that 
$\nabla^2h-gh$ is either positive semi-definite or negative semi-definite.  
The result follows from Corollary~\ref{cor:base} and Remark~\ref{remark:-h}.
\end{proof}

Note that to speak about ``F-hedgehogs'' is not relevant, as they are also  ``P-hedgehogs''. If $h$ is $C^k$ we will speak about $C^k$ hedgehog.
$C^2$ hedgehogs have a natural geometric representation  via the normal representation of $h$, see Subsection~\ref{sec: normal}.
Sometimes we will also call hedgehog the surface $\chi(h)$. 
Note that if $h$ is $\tau$-equivariant, by 
 \eqref{eq:eq grad} $\chi(\H^d)$ is setwise invariant for the action of $\Gamma_{\tau}$.

In the classical case, when  $h$ is the support function of  a convex body,
 the normal representation of $h$ is the boundary of the convex body with support function $h$.
Things are not so simple in our case, as if $h$ is the support function 
of an F-convex set, the normal representation of $h$ describes only  
$\partial_s K$.
For example, the normal representation of the null function is the origin, and
not the future light cone. Anyway we will be mainly interested in  $\tau$-hedgehogs.
 From Lemma~\ref{lem: propr tau cvxe}, if such a function is the support function of an F-convex set and is 
strictly less than $h_{\tau}$, then the image of the normal representation is the boundary of 
 the F-convex set.

\begin{figure}[ht]
\centering
\includegraphics[scale=1]{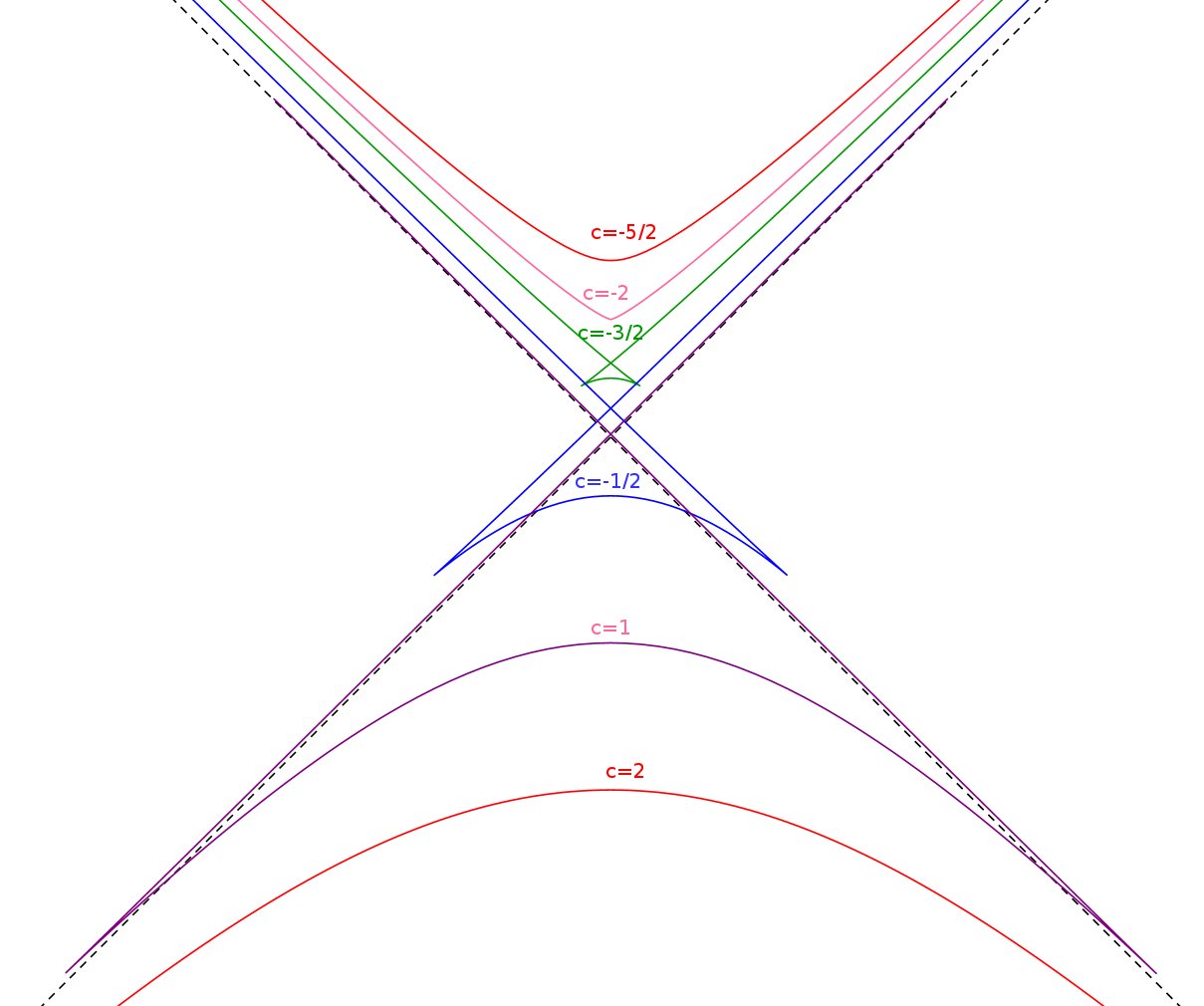}  
\caption{Plane $C^2$ hedgehogs with support function $h(t)=\cos(t)+c$ (curves are drawn thanks to 
\eqref{eq:eq courbe}). If $c$ is sufficiently small or large, the hedgehog bounds an F-convex set or
a P-convex set.
\label{fig:herissons}} 
  \end{figure}

\subsection{Elementary volume computations}

For a space-like $C^1$ hypersurface $S$, we denote by 
$\mathrm{d}(S)$  the  volume form of $S$ for the  Riemannian metric
induced on $S$ by the ambient Lorentzian metric.

\begin{lemma}\label{lemma:vol aire}

Let $A$ be an open set of $\mathbb{R}^{d+1}$ and let 
$l:A\rightarrow \mathbb{R}$ be a $C^1$ function with non-vanishing gradient.
Suppose that the level hypersurfaces $A_t:=l^{-1}(t)$ are space-like. Then
$$V(A)=\int \int_{A_t} \frac{1}{\|\operatorname{grad}_xl \|_-}\mathrm{d}(A_t)(x)  \mathrm{d}t.$$
\end{lemma}

The Lorentzian coarea formula formula above is 
certainly well-known in more general versions, nevertheless we provide a proof,
 just following the classical one, see e.g.~\cite{Schwartz93}. 
The key elementary remarks are:
1) if we take $d$ space-like vectors with last coordinates equal to $0$ and a vertical
vector, the computation of the volume of the resulting box is obviously the same for the Euclidean 
metric and for the Minkowski metric 2)  linear Lorentzian isometries 
have determinant modulus equal to $1$ so they preserve the volume.

\begin{proof}

The Lorentzian gradient of $l$ is a non-zero time-like vector. 
Without loss of generality we suppose that it is past directed.
Moreover  at a point 
$x_0\in A$ we have  $\frac{\partial l}{\partial x_{d+1}}(x_0)\not=0$.
Up to a add a constant to $l$, let us suppose that $l(x_0)=0$.
By the implicit function theorem, locally there exists a $C^1$ map $g$
such that $x_{d+1}=g(x_1,\ldots,x_d,t) $ and
 $$l(x_1,\ldots,x_d,g(x_1,\ldots,x_d,t))=t.$$ We define a  $C^1$
diffeomorphism $\Phi$ from an open set $O\times (-\epsilon,\epsilon)$,
$O\subset\R^{d}$, to  $A$ by
 $$(x_1,\ldots,x_d,t)\mapsto (x_1,\ldots,x_d,g(x_1,\ldots,x_d,t) ).$$
(Up to decompose $A$
into suitable open sets, we suppose for simplicity that the 
 image of $\Phi$  is the whole $A$.) 
Let us denote $X_i=\frac{\partial \Phi}{\partial x_i}$, $i=1,\ldots,d$
and $X_{d+1}=\frac{\partial \Phi}{\partial t}$. Then \cite[6.2.1]{Schwartz93}
$$V(A)=\int_{-\epsilon}^{\epsilon}\int_O | \operatorname{det} (X_1,\ldots,X_{d+1})|\mbox{d}x_1
\cdots \mbox{d}x_d \mbox{d}t.$$
The vectors $X_1,\ldots,X_d$ belong
to the space-like tangent space $L$ to $A_t$.
Let $f_1,\ldots,f_d$ be an orthonormal basis (for $\langle \cdot,\cdot\rangle_-$)
of $L$, and $f_{d+1}$ be the unit past time-like vector orthogonal to $L$.
We have
$$ \operatorname{det} (X_1,\ldots,X_{d+1})= \operatorname{det} 
\left(\langle X_i,f_j\rangle_-\right)_{i,j=1,\ldots,d+1} $$
 (this is easy to see using a Lorentz linear isometry   sending  
$f_1,\ldots,f_{d+1}$ to $e_1,\ldots,e_d,-e_{d+1}$ with $\{e_i\}$ the standard Euclidean basis ---this isometry has determinant
$1$). As $\langle X_i,f_{d+1}\rangle_-=0$ for $i=1,\ldots,d$, 
$$ \operatorname{det} (X_1,\ldots,X_{d+1})= \langle X_{d+1},f_{d+1}\rangle_- 
\operatorname{det} (\langle X_i,f_j\rangle_-)_{i,j=1,\ldots,d}.
$$

On one hand, 
$$ \langle X_{d+1},f_{d+1}\rangle_-=\left \langle \frac{\partial \Phi}{\partial t} ,
\frac{\grad l}{\|\grad l\|_-}\right\rangle_-
=\frac{1}{\| \grad l\|_-}\left\langle 
\left( \begin{array}{c} 0 \\ \vdots \\ 0 \\ \frac{\partial g}{\partial t}\end{array} \right),\grad h \right\rangle_-
$$ $$= \frac{1}{\|\grad l\|_-} \frac{\partial g}{\partial t}\frac{\partial l}{\partial x_{d+1}}
= \frac{1}{\|\grad l\|_-}.
$$ 
On the other hand, $$D:=\operatorname{det} (\langle X_i,f_j\rangle_-)_{i,j=1,\ldots,d}=
-\operatorname{det} M$$
with 
$$M={}^t \!(X_1,\ldots,X_d,f_{d+1})J(f_1,\ldots,f_d,f_{d+1}).$$
 Note that $D=\operatorname{det}MJ$. So
$$-D^2=\operatorname{det} MJ\times {}^tM=  
\operatorname{det} {}^t \!(X_1,\ldots,X_d,f_{d+1}) J (X_1,\ldots,X_d,f_{d+1}) 
=\operatorname{det} (\langle X_i,X_j\rangle_-)_{i,j=1,\ldots,d}, $$
 finally $|D|=\sqrt{\left|\mathrm{det}\left(\langle \frac{\partial \Phi}{\partial x_i},
\frac{\partial \Phi}{\partial x_j}\rangle_-\right)_{i,j=1,\ldots,d}\right| }$
and $|D| \mbox{d}x_1\cdots \mbox{d}x_d$ is the volume form on 
$A_t$ for the metric induced by the Lorentzian metric.
\end{proof}

\section{Area measures}\label{sec:area}


\subsection{Definition of the area measures}

\subsubsection{Main statement}

The notation is the one of Subsection~\ref{subsub:brush}.
Let $\omega\subset\H^d$ be a Borel set. 
The normal field $N$ is continuous, and if we denote by $N_{\epsilon}$ its restriction to
$K_{(0,\epsilon]}$,
$K_{(0,\epsilon]}(\omega)=N_{\epsilon}^{-1}(\omega)$,  so
$K_{(0,\epsilon]}(\omega)$ is measurable for the  Lebesgue measure, and we denote 
by $V_{\epsilon}(K,\omega)$ its volume.
In other terms,  $V_{\epsilon}(K,\cdot)$
is the push forward of the restriction to $K_{(0,\epsilon]}$ of the Lebesgue measure, which is a Radon measure, and
as $N_{\epsilon}$ is continuous,  $V_{\epsilon}(K,\cdot)$ is a Radon measure on $\H^d$.
All results concerning measure theory in this section are elementary and can be found for example in \cite{Tao10}
or in the first pages of \cite{mat95}. Actually we mainly use these well known facts:
\begin{itemize}[nolistsep]
 \item Radon measures on $\H^d$ are the (unsigned) Borel measures which are finite on any compact,
 \item a Radon measure $\mu$ has the inner regularity property: for any Borel set 
 $\omega$ of $\H^d$, 
 $$\mu(\omega)=\operatorname{sup}\{\mu(K)|K\subset \omega, K \mbox{ compact } \}, $$
 \item for any positive linear functional $I$ on the space of real continuous compactly supported
functions on $\H^d$,
 there exists a unique Radon measure $\mu$ on $\H^d$ such that $I(f)=\int_{\mathbb{H}^d}f\operatorname{d}\mu$ 
(Riesz representation
 theorem).
\end{itemize}

The aim of this subsection is to prove the following result. 
\begin{theorem}\label{prop:def area}
 Let $K$ be an F-convex set in $\mathbb{R}^{d+1}$.
There exist Radon measures $S_0(K,\cdot),\ldots,S_d(K,\cdot)$
on $\H^d$ such that, for any Borel set $\omega$ of $\H^d$ and any $\epsilon >0$,
\begin{equation}\label{eq: mu}V_{\epsilon}(K,\omega)=\frac{1}{d+1}\sum_{i=0}^d \epsilon^{d+1-i}\binom{d+1}{i} S_i(K,\omega).\end{equation}
$S_i(K,\cdot)$ is called the \emph{area measure of order $i$} of $K$.
We have  that $S_0(K,\cdot)$ is given by the volume form of $\H^d$.
\end{theorem}

Two of those measures deserve special attention. $S_d(K,\cdot)$ may be called  ``the'' area measure of $K$,
for a reason which will be clear below. The problem of prescribing this measure is the Minkowski problem.
In this paper we will focus on $S_1(K,\cdot)$.

\begin{example}{\rm
For any $p\in \R^{d+1}$ let us consider $K=\C(p)$. Actually
$V_{\epsilon}(\C(p),\omega)$ is invariant under translations, so it suffices to
compute it for  $p=0$. From Lemma~\ref{lemma:vol aire}, using the cosmological time of the future cone
 (the Lorentzian distance to the origin),
which has Lorentzian gradient equal to $1$, 
\begin{equation}\label{eq:vol cone}V_{\epsilon}(\C(p),\omega)=\frac{\epsilon^{d+1}}{d+1} S_0(K,\omega), \end{equation}
that expresses the fact that
all space-like hyperplanes meet $\C(p)$ only at $p$, so the ``curvatures'' are supported only at a single point.
}\end{example}

After some basics results on the $C^1$, $C^2_+$ and polyhedral cases, we will prove a statement 
close to Theorem~\ref{prop:def area} in the Fuchsian case. After that we will prove that,
up to a translation, any compact part of the boundary of an F-convex set can be considered as
a part of a Fuchsian convex set. The proof of Theorem~\ref{prop:def area} 
will follow from the following elementary remark.

\begin{lemma}\label{lem: loc meas}
The area measures defined in Theorem~\ref{prop:def area}  are uniquely defined.
They are even defined locally: if $K$ and $K'$ are two F-convex sets such that the statement 
of Theorem~\ref{prop:def area} holds, and
if 
$\omega$ is a Borel set of $\H^d$ with $K(\omega)=K'(\omega)$, then 
 $S_i(K,\omega)=S_i(K',\omega)$.
\end{lemma}
\begin{proof}
The uniqueness of the $S_i(K,\cdot)$ follows because  
\eqref{eq: mu} says that $V_{\epsilon}(K,\omega)$ is a polynomial in $\epsilon$.
 $K(\omega)=K'(\omega)$ clearly implies $K_{(0,\epsilon]}(\omega)=K'_{(0,\epsilon]}(\omega)$ hence
$V_{\epsilon}(K,\omega)=V_{\epsilon}(K',\omega)$, which are polynomials by  Theorem~\ref{prop:def area},
hence they have equal coefficients.
\end{proof}

\begin{remark}{\rm
 Due to their local nature, the area measures can be defined for more general convex sets than F-convex sets.
 What is needed is that the restriction of the normal map to level sets of the cosmological time is a proper map
, see Subsection~\ref{subsub:brush}.
}\end{remark}

\begin{remark}{\rm
 From \eqref{eq: mu} we get a definition \`a la Minkowski for  the area measure of an F-convex set:
\begin{equation*}\label{eq: mink area}
\underset{\epsilon\downarrow 0}{\mathrm{lim}}
\frac{V_{\epsilon}(K,\omega)-V_{0}(K,\omega)}{\epsilon}=\underset{\epsilon\downarrow 0}{\mathrm{lim}}
\frac{V_{\epsilon}(K,\omega)}{\epsilon}=S_d(K,\omega).
\end{equation*}}\end{remark}

 \begin{remark}{\rm
 Let $K$ be a $C^1$ F-convex set
 and let $\mathrm{d}K$ be the  volume form on $\partial_{s} K$ given by the Riemannian metric
 induced on $\partial_{s}K$ by the ambient Lorentzian metric.
 Let us denote by $\mbox{Area}(K,\omega)$ the measure (for  $\mathrm{d}K$)
  of the set of points 
 of $\partial_{s} K$ whose support vector belongs to $\omega$, i.e.~ $\mbox{Area}(K,\omega)$ is the push-forward of $\mathrm{d}K$ 
  on $\H^d$:
 $$\mbox{Area}(K,\omega)=\mathrm{d}K(G_K^{-1}(\omega))=\mathrm{d}K(K(\omega))=(G_K)_*\mathrm{d}K(\omega),$$
 and $\mbox{Area}(K,\cdot)$ is a Borel measure because $G_K$ is continuous (Lemma~\ref{lemme:base 1}).
 It is even a Radon measure as finite on any compact set, because if $\omega$ is compact then $K(\omega)$ is compact (see Section \ref{subsub:brush}).
 Now, the cosmological time $T$ of any F-convex set $K$ is $C^1$, with Lorentzian gradient equal to
 $1$, so from Lemma~\ref{lemma:vol aire}:
 \begin{equation}\label{eq:vol int}V_{\epsilon}(K,\omega)=\int_0^{\epsilon} \mathrm{Area}(K_t,\omega) \mbox{d}t.\end{equation}
 }\end{remark}

\begin{remark}\label{rem: gamma 0 inv mes}{\rm With the notation of Remark~\ref{eq:inv Komega}:
\begin{equation}\label{eq:inv mu}
 V_{\epsilon}(K,\gamma_0\omega)=V_{\epsilon}(K,\omega).
\end{equation}
}
\end{remark}

\subsubsection{The $C^2_+$ case}

 Let $K$ be a $C^2_+$ F-convex set. We denote by $s_i$ 
the $i$th elementary symmetric function of the radii of curvature of $K$, i.e.
$$s_i=\binom{d}{i}^{-1}\sum_{1\leq j_1<\cdots<j_i \leq d}r_{j_1}\cdots r_{j_i}.$$
In particular $s_0=1$, $s_1=\frac{1}{d}(r_1+\cdots+r_d)=\frac{1}{d}\mathrm{Trace}(S^{-1})$
and $s_d=r_1\cdots r_d=\mathrm{det} (S^{-1})$, where $S^{-1}$ is the reverse 
shape operator of $\partial K$.

\begin{lemma}\label{lem: cas C2}
 Let $K$ be a $C^2_+$ F-convex set. Then the statement of Theorem~\ref{prop:def area} holds. Moreover 
$$S_i(K,\cdot)= s_i\mathrm{d}\mathbb{H}^d(\cdot).$$
\end{lemma}
\begin{proof}
 $K_t$ is  the boundary of $K+tK(\H)$, which is $C^2_+$ by \ref{lem princ: epsilon}
of Corollary~\ref{cor:base}. The Gauss map is a $C^1$ diffeomorphism hence
\begin{equation}\label{eq: aire c2+}
\int_{K_t(\omega)}\mbox{d}K_t=\int_{\omega} \mathrm{det}(S^{-1}_t) \mbox{d}\mathbb{H}^d 
\end{equation}
where $S^{-1}_t$ is the reverse shape operator of the boundary of  $K+tK(\H)$.
Moreover from \eqref{eq:reversed so}  $S^{-1}_t=S^{-1}+t\mathrm{Id}$. The result follows using \eqref{eq:vol int} and
$$\mathrm{det}(S^{-1}+t\mathrm{Id})=\sum_{k=0}^d t^k \binom{d}{k} s_{d-k}. $$
\end{proof}
\begin{remark}{\rm
\eqref{eq: aire c2+} can be written $\mbox{Area}(K,\omega)=S_d(K,\omega)$, that explains the terminology
for ``the'' area measure $S_d$.
}
\end{remark}

\subsubsection{The polyhedral case}\label{subsub:pol}

The following characterization of the area measures for the compact case seemlingly appeared in 
\cite{zel70}, see also \cite{Fir70}.
 Let $P$ be a polyhedral F-convex set.  For a $i$-face $e_i$, we denote by $\lambda_i(e_i)$ the $i$-dimensional volume
of $e_i$ in the Euclidean space isometric to the support plane containing $e_i$.
We also denote by $\nu_n$ the $n$-dimensional Hausdorff measure of $\H^d$.

\begin{lemma}\label{lem: aire poly}
 Let $P$ be a polyhedral F-convex set. Then the statement of Theorem~\ref{prop:def area} holds. Moreover,
for any Borel set $\omega\subset\H^d$,
\begin{equation}\label{eq:area pol}
 S_i(P,\omega)= \binom{d}{i}^{-1}\sum_{e_i}\lambda_i(e_i) \nu_{d-i}(\omega\cap G_P(e_i)).
\end{equation}
where the sum is on all the open $i$-faces $e_i$ of $P$ and $G_P$ is the Gauss map of $P$. 
\end{lemma}

\begin{proof}
 Let $e_i$ be an open $i$-face of $P$ and let $ \omega$ be a Borel subset in the relative interior of $G_P(e_i)$. 
 We have
$$V_{\epsilon}(P,\omega)= \lambda_i(e_i)\frac{\nu_{d-i}(\omega)\epsilon^{d+1-i}}{d+1-i}.$$
 Indeed, up to a volume preserving Lorentzian isometry, 
we can suppose that the hyperplane containing $e_i$ is an horizontal hyperplane,
for which the induced metric for the Euclidean or the Lorentzian structure of  $\R^{d+1}$ are the same.
By Fubini Theorem, 
$$V_{\epsilon}(P,\omega)=V((e_i)_{(0,\epsilon]},\omega)=\int_{e_i} V_{d+1-i}(\C(x)_{(0,\epsilon)}(\omega))) 
\mathrm{d}V_{i}(x)$$
where $V_{k}$ is the volume in $\R^{k}$.
The relation \eqref{eq:vol cone} gives that  
$V_{d+1-i}(\C(x)_{(0,\epsilon)}(\omega))) =\frac{\epsilon^{d+1-i}}{d+1-i}\nu_{d-i}(\omega)$, which is independent
of  $x$.

Now, if $e_i$ and $e_j$ are distinct open faces of $P$, then for any $\omega_i\subset G_P(e_i)$
and $\omega_j\subset G_P(e_j)$, for any positive $\epsilon$, the interiors of $P_{(0,\epsilon]}(\omega_i) $
and $P_{(0,\epsilon]}(\omega_j) $
are disjoint. 
On one hand, $V_{\epsilon}(P,\cdot)$ and $\nu_{d-i}$
are measures on $\H^d$.
On the other hand, the cell decomposition of $\H^d$ given by $P$ has a countable number of cells,
and each face is defined as the intersection of a finite number of cells, hence the decomposition
has a countable number of faces. By the property of countable additivity of measures, we get, for any Borel set $\omega\subset \H^d$:
$$V_{\epsilon}(P,\omega)=\sum_{i=0}^d\frac{1}{d+1-i}\sum_{e_i}\lambda_i(e_i) \nu_{d-i}(\omega\cap G_P(e_i))\epsilon^{d+1-i}.$$
The lemma follows by comparing the coefficients with \eqref{eq: mu}. 
\end{proof}

\subsection{The Fuchsian case}

We prove a ``quotiented'' version of
Theorem~\ref{prop:def area}. By the strong analogy between
Fuchsian convex sets and convex bodies, the argument is a straightforward adaptation of 
Chapter 4 of \cite{Sch93}. 

Let $\Gamma$ be a group of hyperbolic isometries, such that $\H^d/\Gamma$ is a compact hyperbolic manifold. Let
$K$ be a Fuchsian convex set (for the group $\Gamma$). Recall that $V_{\epsilon}(K,\cdot)$ is 
then $\Gamma$ invariant. This permits to introduce a canonical projected Radon measure $V_{\epsilon}^{\Gamma}(K,\cdot)$ on the Borel sets of
$\H^d/\Gamma$. Namely, it is the only measure on $\H^d/\Gamma$ such that if $\omega\subset \H/\Gamma$ is a Borel set and $\psi : \omega \to \H^d$ is a measurable section of the covering projection $\pi:\H^d\to\H^d/\Gamma$, then $\psi_\ast(V_{\epsilon}^{\Gamma}(K,\cdot))= V_{\epsilon}(K,\cdot)$, \cite[Section 3.4]{BF}. In particular it satisfies
$$V_{\epsilon}^{\Gamma}(K,\pi(\omega))= V_{\epsilon}(K,\omega)$$ 
each time $\omega$ meets at most once each orbit of $\Gamma$.

Let us denote by $\mathcal{K}(\Gamma)$ the set 
of $\Gamma$-F-convex sets.
Recall that for $K,K'\in\mathcal{K}(\Gamma)$, the 
Hausdorff distance between them is \cite{FF}
$$ d_H(K,K')=\min\{\lambda \geq 0 |  K'+\lambda K(\H) \subset K,   K+\lambda K(\H)\subset K' \}. $$
If $K\in \mathcal{K}(\Gamma)$,
the \emph{covolume} of $K$,  $\mathrm{covol}_{\Gamma}(K)$, is the volume of $(\F\setminus K)/\Gamma$. 
Note that
\begin{equation}\label{eq:mu et covol}
 V_{\epsilon}^{\Gamma}(K,\H^d/\Gamma)=\mathrm{covol}_{\Gamma}(K_{\epsilon})-\mathrm{covol}_{\Gamma}(K). 
\end{equation}

\begin{lemma}\label{lem:weak conv}
 Let $(K(n))_n$ be a sequence of $\Gamma$-convex sets converging (for $d_H$) to a  $\Gamma$-convex set $K$.
Then $V_{\epsilon}^{\Gamma}(K(n),\cdot)$ weakly converges to $V_{\epsilon}^{\Gamma}(K,\cdot)$.
\end{lemma}
\begin{proof}
We have to prove that
\begin{enumerate}[nolistsep]
\item  $V_{\epsilon}^{\Gamma}(K(n),\H^d/\Gamma)$ converges to  $V_{\epsilon}^{\Gamma}(K,\H^d/\Gamma)$,
 \item for any open set $\omega$ of $\H^d/\Gamma$ then $\operatorname{Liminf}_{n\rightarrow +\infty}
V_{\epsilon}^{\Gamma}(K(n),\omega)\geq V_{\epsilon}^{\Gamma}(K,\omega)$.
\end{enumerate}

Note that $K_{\epsilon}(n)=K(n)+\epsilon K(\H)$ so by continuity of the Minkowski addition,
$K_{\epsilon}(n)$ converges to  $K_{\epsilon}$. By continuity of the covolume \cite{FF}, 
the first point follows from \eqref{eq:mu et covol}. Let us prove the second point.
Let $\omega$ be an open set of $\H^d/\Gamma$,  $\tilde{\omega}$ be any of its lift
and let $x\in K_{(0,\epsilon)}(\tilde{\omega})$.

\emph{Fact: for $n$ sufficiently large, $x\in K(n)_{(0,\epsilon]}(\tilde{\omega})$. }

Let us suppose that the Hausdorff distance between $K$ and $K(n)$ is $\delta$,
the orthogonal projection of $x$ onto $\partial K$ is $p$ and $d_L(x,p)=t<\epsilon$.  Let us denote by
$\eta \in\tilde{\omega}$ the vector $(x-p)/t$. As $K+\delta K(\H)\subset K(n)$, the point
$q=p+\delta\eta$ belongs to $K(n)$.  We can suppose that $\delta$ is  small enough so that 
 $\delta < t $
and then $x=p+t\eta$ belongs to $K(n)$. We denote by 
$p_n$ the orthogonal projection of $x$ onto $\partial K(n)$.
By maximization property, 
$d_L(p_n,x)\geq d_L(x,q)=t-\delta$. Note that $p$ and $p_n$ are both in the past cone of $x$.
Up to a translation  we can suppose
that $x=0$. The last equation writes $\|p_n\|_-\geq t-\delta $.
The property  $K(n)+\delta K(\H)\subset K$  implies
$\langle p_n,\eta \rangle_-\leq H_K(\eta)+\delta$ with $H_K$ the extended support function of $K$,
that can be written $\langle p_n,\eta \rangle_-\leq  t+\delta$. 

 We want to show that 
$-p_n/\|p_n\|_-$ is arbitrary close to $\eta$ if $n$ is sufficiently large (recall that
$p_n$ is a past vector), i.e.~that 
$\cosh d_{\H^d}(-p_n/\|p_n\|_-,\eta)$ is close to $1$, i.e.~that
$\langle p_n/\|p_n\|_-,\eta\rangle_- $ is close to $1$. But
$$ \frac{\langle p_n,\eta\rangle_-}{\|p_n\|}\leq \frac{t+\delta}{t-\delta} $$
that goes to $1$ when $\delta$ goes to $0$. On the other hand,
$\|p_n\|_- \leq \langle p_n,\eta\rangle_- $ as it can be easily checked. 
As $\tilde{\omega}$ is open, for $n$ sufficiently large $-p_n/\|p_n\|_-\in \tilde{\omega}$.
Moreover $$\|p_n\|_- \leq \langle p_n,\eta\rangle_-\leq t+\delta  $$ 
that is less than $\epsilon$ if $\delta$ is sufficiently small because $t<\epsilon$,
so $d_L(p_n,x)=\|p_n\|_- <\epsilon$. The fact is proved.

The fact says that $K_{(0,\epsilon)}(\tilde{\omega})\subset \operatorname{Liminf}_{n}K(n)_{(0,\epsilon)}(\tilde{\omega})$,
hence $$V(K_{(0,\epsilon)}(\tilde{\omega}))\leq V( \operatorname{Liminf}_{n}K(n)_{(0,\epsilon)}(\tilde{\omega}))
\leq \operatorname{Liminf}_{n}V(K(n)_{(0,\epsilon)}(\tilde{\omega}))$$ that implies point 2 because the boundary of a 
convex set has zero Lebesgue measure.
\end{proof}

\begin{lemma}\label{lem: conv pol}
 Let $K$ be a $\Gamma$ convex set. Then there exists a sequence of 
$\Gamma$-convex polyhedra converging to $K$.
\end{lemma}
\begin{proof}
 Let $\epsilon >0$,  $h$ be the support function of $K$ and $k_i\in \partial_s K$.
 There exists $\eta\in\H^d$ such that
$\langle k_i,\eta\rangle_-=h(\eta)$. By continuity there exists an open neighborhood $V_i$ of
$\eta$ in $\H^d$ such that $|\langle k_i,\eta' \rangle_--h(\eta')|<\epsilon,\, \forall \eta'\in V_i$.
By cocompactness of $\Gamma$, there exists a finite number of neighborhood $V_i$ as above
such that $\{\Gamma V_i\}$ covers $\H^d$. 
The associated set of points $\{\Gamma k_i\}$ is discrete as  discrete orbits of a finite number of points.

Let us introduce $h_{\epsilon}(\eta)=\operatorname{max}_i\langle k_i,\eta\rangle_-$.
It is easy to see that if $\eta\in V_i$ and $\eta\notin V_j$, then $\langle \eta,k_j\rangle_-<\langle \eta,k_i\rangle_-$.
Moreover each $\eta$ belongs to a finite number of $V_i$ (the tessellation of $\H^d$ by fundamental domains for 
$\Gamma$ is locally finite), hence $h_{\epsilon}$ is well defined.
It is also clearly $\Gamma$ invariant, hence it is the support function 
of  a $\Gamma$ convex polyhedron and by construction, on $\H^d$, $|h_{\epsilon}(\eta)-h(\eta) |<\epsilon$.
\end{proof}

\begin{proposition}\label{prop: area cas fuchs}
 Let $K$ be a $\Gamma$ convex set.
There exists finite  Radon measures $S^{\Gamma}_0(K,\cdot),\ldots,S^{\Gamma}_d(K,\cdot)$
on $\H^d/\Gamma$ such that, for any Borel set $\omega$ of $\H^d/\Gamma$ and any $\epsilon >0$,
\begin{equation}\label{eq: mu fuchs}
V_{\epsilon}^{\Gamma}(K,\omega)=\frac{1}{d+1}\sum_{i=0}^d \epsilon^{d+1-i}\binom{d+1}{i} S_i^{\Gamma}(K,\omega),
\end{equation}
and $S_0^{\Gamma}(K,\cdot)$ is given by the volume form on $\H^d/\Gamma$.

Moreover, if $K(n)$ converges to $K$, then $S_i^{\Gamma}(K(n),\cdot)$
weakly converges to $S_i^{\Gamma}(K,\cdot)$.
\end{proposition}
\begin{proof}
 If $P$ is a $\Gamma$ Fuchsian polyhedron, then \eqref{eq: mu fuchs} is a consequence
of \eqref{eq:area pol}, applied to any lifting of $\omega$.
By polynomial interpolation,
for $d+1$ distinct reals numbers $n_0,\ldots,n_{d}$,  \eqref{eq: mu fuchs}  applied with $\epsilon=n_i$ can be considered as
a solvable system of $d+1$ linear equations with unknowns
$S_0^{\Gamma}(P,\omega),\ldots,S_d^{\Gamma}(P,\omega)$. So there exists
real numbers $a_{im}$ with
$$S_i^{\Gamma}(P,\omega)=\sum_{m=0}^{d}a_{im}V_{n_m}^{\Gamma}(P,\omega).$$
Now let $K$ be any $\Gamma$-convex set. We define 
$$S_i^{\Gamma}(K,\cdot):=\sum_{m=0}^{d}a_{im}V_{n_m}^{\Gamma}(K,\cdot). $$
Clearly $S_i^{\Gamma}(K,\cdot)$ is a finite signed Radon measure on
$\H^d/\Gamma$. From Lemma~\ref{lem: conv pol} we can consider a sequence $P(n)$ of
$\Gamma$-convex polyhedra converging to $K$, and from Lemma~\ref{lem:weak conv},
for any continuous function  $f$ on $\H^d/\Gamma$,
$$\int_{\H^d/\Gamma} f\operatorname{d}S_i^{\Gamma}(P(n),\cdot) \rightarrow 
\int_{\H^d/\Gamma} f\operatorname{d}S_i^{\Gamma}(K,\cdot).$$
It follows that  $f\mapsto \int_{\H^d/\Gamma} f\operatorname{d}S_i^{\Gamma}(K,\cdot)$
is a positive linear functional, hence  $S_i^{\Gamma}(K,\cdot)$ is a Radon measure. 

The statement about weak convergence is clear. Using again polyhedral approximation and the fact that
 \eqref{eq: mu fuchs} is true in the polyhedral case, we see that the functionals on the continuous
functions of $\H^d/\Gamma$ given by integrating with respect to each side of  \eqref{eq: mu fuchs} are equal, hence
the measures are equal by the uniqueness part of the Riesz representation theorem.
We also get the remark on $S_0^{\Gamma}$ from Lemma~\ref{lem: aire poly}.
\end{proof}

\begin{remark}[\textbf{A Steiner formula}]\label{rem:steiner}{\rm 
Let us introduce
$$W_i^{\Gamma}(K)=\frac{1}{d+1}S_{d+1-i}^{\Gamma}(K,\hG)$$
and $W_0^{\Gamma}(K):=\operatorname{covol}_{\Gamma}(K)$. They are
the \emph{$ \Gamma$-quermass integrals} of $K$. Then  \eqref{eq: mu fuchs}  gives the following 
Steiner formula for $ \Gamma$ convex sets:
$$V_{\epsilon}(K)=\sum_{i=1}^{d+1}\epsilon^i \binom{d+1}{i} W_i(K).$$

Note that $S_0^{\Gamma}(K,\H^d/\Gamma)=(d+1)W_{d+1}^{\Gamma}(K)$ is nothing but the volume of $\H^d/\Gamma$,
which is itself related to the Euler characteristic of $\H^d/\Gamma$ if $d$ is even by the Gauss--Bonnet formula \cite{Rat06}.
In the compact Euclidean case, up to a dimensional constant the quermass integrals are the intrinsic volumes, and their sum  
has an integral representation known as Wills functional, see \emph{e.g.}~\cite{kam69}.
}\end{remark}

\begin{remark}[\textbf{Mixed-area}]{\rm  Recall that $\mathcal{K}(\Gamma)$ is the set 
of $\Gamma$-convex sets. 
The \emph{mixed-covolume} $\mathrm{covol}(\cdot,\ldots,\cdot)$ is the unique symmetric $(d+1)$-linear form
on $\mathcal{K}(\Gamma)$, continuous on each variable, such that \cite{FF}
$$\mathrm{covol}(K,\ldots,K)=\mathrm{covol}(K).$$  
For given $K_1,\ldots,K_d\in \mathcal{K}(\Gamma)$, we get an additive functional
$$\mathrm{covol}(\cdot,K_1,\ldots,K_d) : \mathcal{K}(\Gamma)\rightarrow \R, K\mapsto \mathrm{covol}(K,K_1,\ldots,K_d). $$

If we identify the $\Gamma$-convex sets
with their  support functions, we can consider $\mathcal{K}(\Gamma)$ as a subset of $C^0(\H^d/\Gamma)$, 
the set of continuous functions on 
$\H^d/\Gamma$.
Following the classical arguments of the compact case \cite{Ale37},
one can show that 
 $\mathrm{covol}(\cdot,K_1,\ldots,K_d)$ can be  extended to a positive linear functional on $C^0(\H^d/\Gamma)$.
The  first step is to extend $\mathrm{covol}(\cdot,K_1,\ldots,K_d)$
to the subset of $C^0(\H^d/\Gamma)$ of functions which are difference of 
support functions: if $Z=h_1-h_2$ where $h_1$ and $h_2$ are support functions of
$\Gamma$ convex sets, then we define
$$\mathrm{covol}(Z,K_1,\ldots,K_d)=\mathrm{covol}(h_1,K_1,\ldots,K_d)-\mathrm{covol}(h_2,K_1,\ldots,K_d).$$
By the Stone--Weierstrass theorem, any continuous function on $\H^d/\Gamma$ can be uniformly approximated by a
$C^2$ function. Moreover
 any $C^2$  function $Z$ on $\H^d/\Gamma$ is the difference of two support functions:
for $t$ sufficiently large, $Z+t$ satisfies \eqref{eq:hess supp}. Hence 
any continuous function on $\H^d/\Gamma$ can be uniformly approximated by the difference of two support functions.
From this it can be checked that $\mathrm{covol}(\cdot,K_1,\ldots,K_d)$ can be  extended to 
$C^0(\H^d/\Gamma)$ with the required properties.

By the  Riesz representation theorem  there exists 
a unique Radon measure on $\H^d/\Gamma$, the \emph{mixed-area measure}, denoted by $S(K_1,\ldots,K_d;\cdot)$, 
such that, for any $f\in C^0(\H^d/\Gamma)$,
$$\mathrm{covol}(f,K_1,\ldots,K_d)=-\frac{1}{d+1}\int_{\mathbb{H}^d/\Gamma} f(u) \mathrm{d}S(K_1,\ldots,K_d;u). $$ 

The mixed-area measures are generalization of the area measures in the Fuchsian case.
Let us sketch the proof of this fact. Following \cite{FJ38}, p.~29, one can prove 
that 
$$S(K,\cdots,K;\omega)= \underset{\epsilon\downarrow 0}{\mathrm{lim}}\frac{V_{\epsilon}(K,\omega)}{\epsilon}$$

It is clear that $K_{(0,\epsilon+t]}(\omega) $ is the disjoint union 
of         $K_{(0,\epsilon]}(\omega) $ and of $(K_{\epsilon})_{(0,t]}(\omega)$, in particular
$$V_{\epsilon+t}(K,\omega)=V_{t}(K,\omega)+ V_{\epsilon}(K_t,\omega)$$
 hence
the equation above can be written
$$\underset{\epsilon\downarrow 0}{\mathrm{lim}}
\frac{V_{\epsilon+t}(K,\omega)-V_{t}(K,\omega)}{\epsilon}=S(K+tK(\H),\cdots,K+tK(\H),\omega), $$
 with $K(\H)$ the $\Gamma$-convex set bounded by  $\H^d$,  in other terms
$$
 S(K+tK(\H),\cdots,K+tK(\H),\omega)=\frac{d}{d\epsilon}\left(V_{\epsilon}(K,\omega) \right)(t).
$$

On the other hand, 
by properties of the mixed-covolume, $S(K_1,\ldots,K_d;\cdot)$ is linear in each variable,
in particular,
 $$S(K+t K(\H),\ldots,K+t K(\H);\cdot)=\sum_{i=0}^d t^{d-i} \binom{d}{i} S(\underbrace{K,\ldots,K}_{i},K(\H),\ldots,K(\H);\cdot), $$

Integrating the two equations above between $0$ and $\epsilon$ with respect to 
$t$ leads to

$$V_{\epsilon}(K,\omega)=\frac{1}{d+1}\sum_{i=0}^d \epsilon^{d+1-i}\binom{d+1}{i} S(\underbrace{K,\ldots,K}_{i},K(\H),\ldots,K(\H);\omega), $$

Comparing the coefficients with \eqref{eq: mu fuchs} leads to 
$$S(\underbrace{K,\ldots,K}_{i},K(\H),\ldots,K(\H);\cdot)=S_i^{\Gamma}(K,\cdot). $$

}\end{remark}

\begin{remark}\label{rem:meanwidth fuch}{\rm 
 With the notations of Remark~\ref{rem:steiner}

$$W_d(K)=\frac{1}{d+1}S_{1}^{\Gamma}(K,\hG) 
=\int_{\hG}\mathrm{d}S(K,K(\H),\ldots,K(\H))=\covol(K(\H),K,K(\H),\ldots,K(\H))=$$
$$=\covol(K,K(\H),\ldots,K(\H))=-\int_{\hG}h\mathrm{d}S(K(\H),\ldots,K(\H))
=-\int_{\hG}h\mathrm{d}\H^d/\Gamma
$$
(in the $C^2_+$ case, writing the first area measure with the help of the Laplacian --- see \eqref{eq_hn} ---,
it appears that the
formula above is nothing but the Green Formula $\int_{\hG} h\Delta f=\int_{\hG} f\Delta h$ applied to $f=-1$). See also Subsection~\ref{sub mean width}.
 }\end{remark}

\begin{remark}[\textbf{Mean radius of curvature and Hessian of the covolume}]{\rm 
As in the compact Euclidean case, the Hessian of the covolume 
of  $C^{\infty}_+$ Fuchsian convex sets, at the point $K(\H)$, is $( S_1(\cdot),\cdot)$, where
$(\cdot,\cdot)$ is the $L^2$ scalar product on $\H^d/\Gamma$, see
 \cite{FF} --- it acts on the space of $C^{\infty}$ functions on  $\H^d/\Gamma$, i.e.~on the space of 
 $C^{\infty}$ $\Gamma$-hedgehogs.
 }\end{remark}

\subsection{Fuchsian extension}\label{subsub: fuch ext}

\begin{lemma}\label{lem: fuchs ext}
 Let $K$ be an F-convex set and $\omega\subset\H^d$ be a bounded Borel set.
Up to a translation, there exists a Fuchsian convex set $\tilde{K}_{\omega}$ such that, 
for any subset $\omega'$ of $\omega$, 
$$K(\omega')=\tilde{K}_{\omega}(\omega').$$
\end{lemma}

$\tilde{K}_{\omega}$ is a \emph{$\omega$-Fuchsian extension} of $K$. 

\begin{proof}

Let $ \boldsymbol{\omega}$ be a compact set of $\H^d$ containing $\omega$ in its interior.
 As $K(\boldsymbol{\omega})$ is compact (see Subsection~\ref{subsub:brush}), up to a translation, we 
 suppose $K(\boldsymbol{\omega})\subset \F$. This implies that the support function $h_K$ of $K$ is negative on  $ \boldsymbol{\omega}$
(for $x\in K(\boldsymbol{\omega})$ with support vector $\eta\in \boldsymbol{\omega}$ we have $h_K(\eta)=\langle \eta,x\rangle_-<0)$.
  Let $h_0$ be the infimum of $h_K$  on $ \boldsymbol{\omega}$. Let $B_{\rho}$ be the  closed ball of $\H^d$ of radius $\rho$ centered at $e_{d+1}$.

\emph{Fact:  $\exists \rho>0, \forall x\in K(\boldsymbol{\omega}), 
\forall \eta \in \H^d\setminus B_{\rho},
\langle x,\eta\rangle_-\leq h_0$.}

The condition $\langle x,\eta\rangle_-\leq h_0$ can be written
$$\cosh d_{\H^d}\left(\frac{x}{\|x\|_-},\eta\right)\geq \frac{|h_0|}{\|x\|_-}. $$
As $K(\boldsymbol{\omega})$ is compact  and contained in $\F$,
$\{x/\|x\|_- | x\in K(\boldsymbol{\omega}) \}$ is a compact set of $\H^d$, say contained in
$B_r$. Any $\rho$ larger than $\displaystyle r+\frac{|h_0|}{\underset{x\in K(\boldsymbol{\omega})}{\mathrm{inf}}(\|x\|_-)}$
satisfies the wanted condition. The fact is proved.

Let $\Gamma$ be a group of isometries such that $\H^d$ is compact and containing $B_{\rho}$ in a fundamental domain 
(this is always possible, see page 74 of  \cite{Far96}). 
We define
 \begin{equation}\label{eq: fuch ext def}
 \tilde{K}_{\omega}:= \{x\in\R^{d+1}|\langle x,\eta\rangle_-\leq h_K(\gamma_0^{-1}\eta), \forall \eta=\gamma_0\eta_0,\gamma_0\in\Gamma,
\eta_0\in\boldsymbol{\omega}\},
 \end{equation}
i.e.~$\tilde{K}_{\omega}$ is the intersection of the future side of the support planes of 
$K(\omega)$ and of their orbits for the action of $\Gamma$. 
Because of the choice of $\Gamma$, $ K( \boldsymbol{\omega})\subset \tilde{K}_{\omega}$.
Moreover it is clear that the support planes of $K( \boldsymbol{\omega})$ are support planes
of $\tilde{K}_{\omega}$, hence $ K( \boldsymbol{\omega})\subset \tilde{K}_{\omega}( \boldsymbol{\omega})$
(note that the inclusion may be strict). 
Finally  $\tilde{K}_{\omega}$ is different from $\overline{\F}$,  it is 
$\Gamma$-invariant and it is an F-convex set (Lemma~\ref{lem: conv dans Fconv}) 
hence it is a $\Gamma$-F-convex set.

Finally, we prove below  that $ K(\stackrel{\circ}{\boldsymbol{\omega}})=\tilde{K}_{\omega}(\stackrel{\circ}{\boldsymbol{\omega}})$.
Obviously this implies that for any subset $\omega'$ of $\stackrel{\circ}{\boldsymbol{\omega}}$, 
we have $K(\omega')=\tilde{K}_{\omega}(\omega')$.

Suppose that  $ K(\stackrel{\circ}{\boldsymbol{\omega}})\not=\tilde{K}_{\omega}(\stackrel{\circ}{\boldsymbol{\omega}})$.
As $ K(\stackrel{\circ}{\boldsymbol{\omega}})\subset\tilde{K}_{\omega}(\stackrel{\circ}{\boldsymbol{\omega}})$, this means that 
there exists $y\in \partial \tilde{K}_{\omega}$, $y\notin K(\stackrel{\circ}{\boldsymbol{\omega}})$ and
$\eta\in G_{\tilde{K}_{\omega}}(y)\cap \stackrel{\circ}{\boldsymbol{\omega}}$. 
Let $\mathcal{H}$ be the support hyperplane of $K$ orthogonal to $\eta$. 
$\mathcal{H} \cap K=K(\{\eta\})$ is a convex compact set (see Lemma~\ref{lem:pas half space}). Let $x$ be the orthogonal projection
(in $\mathcal{H}$)
of $y$ onto $\mathcal{H} \cap K$. Let us denote by $v$ the normalization of the space-like vector $y-x$ and 
by  $H$ the extended support function  of $K$ (which is equal to the extended support function
of $\tilde{K}_{\omega}$ on $ \boldsymbol{\omega}$). We also denote by 
$H'(\eta;v)$ the one-sided directional derivative of $H$ at $\eta$ in the direction $v$.
By \cite[1.7.2]{Sch93} (see the proof of \cite[3.1]{FF} for the Lorentzian version) 
it is equal to the total support function of  $\mathcal{H} \cap K$ evaluated at $v$, hence it is 
equal to $\langle x,v\rangle_-$. 

As $\eta$ is in the interior of $ \boldsymbol{\omega}$, for small positive $\varepsilon$, 
 the projection of $  \eta+\varepsilon v$ onto $\H^d$ is in $ \boldsymbol{\omega}$.
We want to find the non-negative $\lambda(\varepsilon)$, depending on $\varepsilon$, such that
\begin{equation}\label{eq: supp lem tech}
 \langle x+\lambda(\varepsilon) v,\eta+\varepsilon v\rangle_-=H(\eta+\varepsilon v).
\end{equation}

We get (recall that $\langle v,\eta\rangle_-=0$)
$$\lambda(\varepsilon)=\frac{H(\eta+\varepsilon v)-\langle x,\eta\rangle_-}{\epsilon}-\langle x,v\rangle_-. $$
which is non-negative because  
$$\langle x,\eta+\varepsilon v\rangle_-\leq H(\eta+\varepsilon v),$$
(that only means that $x$ belongs to $K$) and clearly continuous on positive $\varepsilon$.

Moreover
$$\lambda(\varepsilon)=\frac{H(\eta+\varepsilon v)-H(\eta)}{\epsilon}-\langle x,v\rangle_-
 $$
and $\underset{\varepsilon \downarrow 0}{\mathrm{lim}} \lambda(\varepsilon)=H'(\eta;v)-\langle x,v\rangle_-=0$.

Hence one can find $\varepsilon$ such that 
$x+\lambda(\epsilon) v$ is between $x$ and $y$ and different from $y$. 
But \eqref{eq: supp lem tech} says that $x+\lambda(\varepsilon) v$
is the intersection between 
the support hyperplane of $K(\boldsymbol{\omega})$ orthogonal to $\eta+\varepsilon v$  and the line between $x$ and $y$:
$y$ is not on the same side of a support plane of $K(\boldsymbol{\omega})$ (and hence of 
$\tilde{K}_{\omega}$) than $x$, that is impossible.
\end{proof}

\begin{lemma}\label{lem: fuchs ext conv}
Let $K_1, K_2$ be two F-convex sets with extended support functions $H_1, H_2$ 
and $\boldsymbol{\omega}$ a compact set of $\H^d$ and  $\epsilon >0$
with $\mathrm{sup}_{\eta\in\boldsymbol{\omega}}|H_1(\eta)-H_2(\eta)|<\epsilon$. 

Then for any compact set $\omega$ in the interior of $\boldsymbol{\omega}$ there exists an isometry group $\Gamma$ acting cocompactly on $\H^d$ and $\omega$-$\Gamma$ extensions
$\tilde{K_1}$ and $\tilde{K_2}$ of respectively $K_1$ and $K_2$ such that 
$d_H(\tilde{K_1},\tilde{K_2})<\epsilon$.
\end{lemma}
\begin{proof}
 Using the same notations as in the proof of Lemma~\ref{lem: fuchs ext}, we take as $\Gamma$
a  group containing $B_{\mathrm{max}(\rho_{K_1},\rho_{K_2})}$ in a fundamental domain.
Let $y\in \tilde{K_1}+\epsilon K(\H)$: $y=x+\epsilon b$ with $x\in \tilde{K_1}$ and $b\in K(\H)$.
Let  $\eta\in\H^d$ such that $\eta=\gamma_0\eta_0$ with $\gamma_0\in\Gamma$
and $\eta_0\in\boldsymbol{\omega}$. We have
$$\langle y,\eta\rangle_-=\langle x,\eta\rangle_-+\epsilon \langle b,\eta\rangle_-
\leq H_1(\eta_0)-\epsilon\leq H_2(\eta_0)=H_2(\gamma_0^{-1}\eta),$$
because $\langle b,\eta\rangle_-\leq -1$, hence $y\in\tilde{K_2}$ by definition \eqref{eq: fuch ext def}, and  $\tilde{K_1}+\epsilon K(\H)\subset \tilde{K_2}$. In the same way 
$\tilde{K_2}+\epsilon K(\H)\subset \tilde{K_1}$, so by definition $d_H(\tilde{K_1},\tilde{K_2})<\epsilon$.
\end{proof}

\subsection{Proof of Theorem~\ref{prop:def area}}

Let $\omega\subset \H^d$ be compact, and consider $\tilde{K}_{\omega}$ as in Lemma~\ref{lem: fuchs ext}
(clearly, $V_{\epsilon}(K,\omega)$ is invariant under translation).
Let us define the following Radon measures on $\omega$: for any Borel
set $b$ contained in $\omega$ and $i\in\{1,\ldots,d\}$
\begin{equation}\label{def si avec quoti}S_i^{\omega}(K,b):= S^{\Gamma}_i(\tilde{K}_{\omega},\overline{b}),\end{equation}
with $\overline{b}$ the image of $b$ for the projection $\H^d\rightarrow\H^d/\Gamma$ and 
$S^{\Gamma}_i(\tilde{K}_{\omega},\overline{b})$ given by Proposition~\ref{prop: area cas fuchs}.
From  Lemma~\ref{lem: loc meas}, this definition does not depend on $\tilde{K}_{\omega}$.

Let $C_c^0(\H^d)$ be the space of continuous functions with compact support on $\H^d$. Let $f\in C_c^0(\H^d)$,
and define
$$F_i(K)(f)=\int_{\omega} f \mathrm{d}S_i^{\omega}(K,\cdot), $$
where  $\omega$ is a compact set such that  $\mathrm{supp} f\subset \omega$. 
It is well defined, because if $\omega'$ is another compact set with 
 $\mathrm{supp} f\subset \omega'$, then $S_i^{\omega}(K,\cdot)$ and 
 $S_i^{\omega'}(K,\cdot)$ coincides on $\omega \cap \omega'$, that follows again from Lemma~\ref{lem: loc meas}.

For any $i\in\{1,\ldots,d\}$, it is easy to see that $F_i(K)$  is a linear functional on $C_c^0(\H^d)$. 
It is moreover a positive functional, so by the
 Riesz representation theorem  there exists a unique Radon measure on 
$\H^d$, that we denote by $S_i(K,\cdot)$, such that $F_i(K)(f)=\int_{\H^d}f \mathrm{d}S_i(K,\cdot)$.

Let $f$ with support in the compact set $\omega$. Recall that $\omega$ is contained in a fundamental domain 
for the action of the group $\Gamma$ fixing $\tilde{K}_{\omega}$. Let us denote by $\bar{\omega}$ (resp. $\overline{f}$) the image  
of $\omega$ (resp. $f$) for the canonical projection $\H^d\rightarrow \H^d/\Gamma$. By Lemma~\ref{lem: fuchs ext},
$V_\epsilon(K,\cdot)=V_\epsilon(\tilde{K}_\omega,\cdot)$ on $\omega$, 
then
$$\int_{\omega} f \mathrm{d}V_\epsilon(K,\cdot)= 
\int_{\omega} f \mathrm{d}V_\epsilon(\tilde{K}_\omega,\cdot)
=\int_{\bar{\omega}} \overline{f} \mathrm{d}V^\Gamma_\epsilon(\tilde{K}_\omega,\cdot)
\stackrel{\eqref{eq: mu fuchs}}{=}\int_{\bar{\omega}} \overline{f} \left(\frac{1}{d+1}\sum_{i=0}^d \epsilon^{d+1-i}\binom{d+1}{i}   S_i^{\Gamma}(\tilde{K}_\omega,\overline{\omega})\right) $$
$$\stackrel{\eqref{def si avec quoti}}{=} \int_{\omega} f \mathrm{d}\left(\frac{1}{d+1}\sum_{i=0}^d \epsilon^{d+1-i}\binom{d+1}{i}   S_i(K,\omega) \right)
$$
and \eqref{eq: mu} follows by the  uniqueness part of the Riesz representation theorem. 

\subsection{Characterizations of the first area measure}

\subsubsection{Distribution characterization}

Let $K$ be a  F-convex set of $\R^{d+1}$ with $C^2$ support function $h$. The \emph{mean radius of curvature} 
$S_1(h)$ of $K$ is the sum of the principal radii of curvature divided by $d$:
\begin{equation}
 \frac{1}{d}\Delta h -  h =S_1(h) \tag{\ref{eq_hn}}
\end{equation}
where $\Delta$ is the Laplacian on the hyperbolic space.

\begin{example}\label{ex: mesure cone}{\rm 
Let $K$ be the future cone of a point $p$. The Hessian of its extended support function is
the null matrix, hence, as expected, its mean radius of curvature is zero. 
}\end{example}

We will generalize \eqref{eq_hn}. For any F-convex set $K$ with support function $h$, 
or more generally for any continuous function $h$ on $\H^d$,
we define $S_1(h)$ by \eqref{eq_hn} considered in the sense of distributions: $\forall f\in C_c^{\infty}(\H^d),$

\begin{equation}\label{eq:distribution action} (S_1(h),f)=\int_{\H^d}f \left(\frac{1}{d}\Delta -1\right)h \mathrm{d}\H^d
:=\int_{\H^d}h \left(\frac{1}{d}\Delta -1\right)f \mathrm{d}\H^d.\end{equation}
Note that $S_1$ is linear with respect to $h$.

\begin{lemma}\label{lem: aire distribution}
 If $h$ is the support function of $K$, then  $S_1(h)=S_1(K,\cdot)$ in the sense of distributions.
\end{lemma}
\begin{proof}

Let  $f\in C_c^{\infty}(\H^d)$ and suppose that $\mathrm{supp} f\subset \omega$ with $\omega$ compact.
From Lemma~\ref{lem: approx} we know that there exists $C^2$ support function 
$h_n$ of $C^2_+$ F-convex sets $K_n$ converging  to $h$ uniformly on $\omega$. Hence
$\int_{\H^d}h_n \left(\frac{1}{d}\Delta -1\right)f \mathrm{d}\H^d$
converges to $\int_{\H^d}h \left(\frac{1}{d}\Delta -1\right)f \mathrm{d}\H^d$.

Let us consider $\omega$-Fuchsian extensions of $K_n$ and $K$ 
converging for the Hausdorff distance (Lemma~\ref{lem: fuchs ext conv}).
From
Proposition~\ref{prop: area cas fuchs}, the corresponding first area measures weakly converge. But on 
$\omega$ they are equal to the first area measures of $K_n$ and $K$ respectively  (Lemma~\ref{lem: loc meas}), hence
$\int_{\H^d}f \mathrm{d}S_1(K_n,\cdot)$ converge to $\int_{\H^d}f \mathrm{d}S_1(K,\cdot)$.

By Lemma~\ref{lem: cas C2} we know that for all $n$, 
$\int_{\H^d}h_n \left(\frac{1}{d}\Delta -1\right)f \mathrm{d}\H^d=\int_{\H^d}f \mathrm{d}S_1(K_n,\cdot)$.
This proves the lemma.

\end{proof}

\subsubsection{Polyhedral case}\label{subsubsub pol area}

Let $P$ be an F-convex polyhedron, inducing a decomposition
$C$ of $\H^d$. 
From Lemma~\ref{lem: aire poly}, the first area measure of $P$
is a weight on each facet  $\zeta$ of $C$, equal to
$\frac{1}{d}$ times $\lambda(\zeta)$, the length of the corresponding edge of $P$.
There is a necessary condition on the weights, if there exist $(d-2)$-faces of $C$.
Let $\eta$ be a $(d-2)$-face  contained in a facet $\zeta$ of $C$.
We denote by $u(\eta,\zeta)$ the unit tangent vector (of $\H^d$) orthogonal to $\eta$ and contained in $\zeta$.
We also denote by $u(\eta,\zeta)$ the corresponding space-like vector of Minkowski space.
For any $(d-2)$-face $\eta$,
\begin{equation}\label{eq:sum poly}
\sum  \lambda(\zeta)u(\eta,\zeta)=0 \end{equation}
where the sum is on the facets $\zeta$ containing $\eta$.
A  $(d-2)$-face of $C$
is the set of normal vectors to a $2$-dimensional face $F$ of $P$, say contained in a plane $\mathcal{H}$. 
In $\mathcal{H}$, $F$ is a compact convex polygon, and 
by construction 
$u(\eta,\zeta)$ is an outward unit normal of the edge of $F$ of length $\lambda(\zeta)$. 
The condition stated is 
then well-known: the sum of the weighted sum of the vectors orthogonal to $u(\eta,\zeta)$  (the edges of the polygon)
must close up. 

We will call \emph{polyhedral measure of order one} a Radon measure $\varphi$ on $\H^d$ satisfying the 
properties above, namely:
\begin{enumerate}[nolistsep,label={\bf(\roman{*})}, ref={\bf(\roman{*})}]
 \item the support of $\varphi$ is the set of  facets of a numerable decomposition $C$ of $\H^d$ by 
compact convex polyhedra.
\item For any facet $\zeta$ of $C$, there exists a positive number $\lambda(\zeta)$
such that $\varphi(\omega)=\lambda(\zeta) \nu_{d-1}(\omega)$, for any Borel set $\omega$ of $\zeta$,
\item\label{pol: 3}  for any $(d-2)$-face $\eta$,
\eqref{eq:sum poly} is satisfied.
\end{enumerate}

From Lemma~\ref{lem: aire distribution}, the first area measure of an F-convex polyhedron can also be written 
as in \eqref{eq_hn} in the sense of distribution. 
Let us check it below on the most elementary example.

\begin{example}[\textbf{The elementary example}]{\rm Let $K$ be the elementary example of Example~\ref{ex: elemntary1} (note that
this example is easily generalized in all dimensions). $p_1$ and $p_2$ are two points in $\R^3$, 
related by a space-like segment of length $a$. Let $\gamma^{\bot}$ be the time-like plane orthogonal to
$p_1-p_2$. $\gamma^{\bot}$ separates $\F$ into two regions $\tilde{\O}_1$ and $\tilde{\O}_2$,
such that $p_1-p_2$ is pointed towards $\tilde{\O}_2$. The extended support function of $K$ is the restriction of
$H_i=\langle \cdot,p_i\rangle_-$ on  $\tilde{\O}_i$. Let us denote by $\O_i$ the intersection of
$\tilde{\O}_i$ with $\H^2$, and by $h_i$ the restriction of $H$ to $\O_i$. 
Let $\nu_1$ et $\nu_2$ be the exterior normals to $\O_1$ and
$\O_2$ respectively and note that $\nu_1=-\nu_2$ on
$\partial\O_1=\partial\O_2=\gamma$. Then, for $f\in C_c^{\infty}(\H^d)$,

\begin{eqnarray*}
\ d(S_1(h),f) &=& d (\frac{1}{d}\Delta h - h , f) = (\Delta h - dh , f) =
 \int_{\mathbb H^d} h (\Delta f-df) \\
\ &=& -\int_{\mathbb H^d} dhf - \int_{\O_1} \left\langle \nabla h_1
, \nabla f \right\rangle + \int_{\partial\O_1} h_1 \left\langle
\nabla f , \nu_1 \right\rangle  - \int_{\O_2} \left\langle \nabla
h_2 , \nabla f \right\rangle + \int_{\partial\O_2} h_2
\left\langle \nabla f , \nu_2 \right\rangle \\
\  &=& -\int_{\mathbb H^d} dhf + \int_{\O_1}  f\Delta h_1 -
\int_{\partial\O_1}  f \left\langle \nabla h_1 , \nu_1
\right\rangle + \int_{\O_2}  f\Delta h_2 -
\int_{\partial\O_2}  f \left\langle \nabla h_2 , \nu_2
\right\rangle \\
\ &=& \int_{\O_1}  f (\Delta h_1 - dh_1) + \int_{\O_2}  f
(\Delta h_2 - dh_2) + \int_{\partial\O_1}  f \left\langle \nabla
h_1 - \nabla h_2 , \nu_1 \right\rangle
= \int_{\partial\O_1=\gamma}  f \left\langle \nabla (h_1 - \nabla h_2)
, \nu_1 \right\rangle,
\end{eqnarray*}
because $(\Delta h_i - dh_i)=0$ (see Remark~\ref{ex: mesure cone}).

As 
$$(H_1-H_2)(\eta)=\langle p_1-p_2,\eta\rangle_-,\qquad  \operatorname{grad}_{\eta}(H_1-H_2)=p_1-p_2,$$
and from \eqref{eq:nablanabla}, $$\operatorname{grad}_{\eta}(H_1-H_2)=\nabla_{\eta}(h_1-h_2)-(h_1-h_2)(\eta)\eta.$$
Note that $p_1-p_2=a\nu_1$, so if $\eta\in\gamma$, 
$$(h_1-h_2)(\eta)=\langle p_1-p_2,\eta\rangle_-=0.$$
Finally,
$\grad_{\eta}(h_1-h_2)=p_1-p_2$, and 
$$\langle \nabla (h_1 -  h_2)
, \nu_1 \rangle=\langle p_1 - p_2
, \nu_1 \rangle_-=a,$$ and as expected 
$$(S_1(h),f)=\frac{1}{d}a\int_{\gamma}f.$$

 }
\end{example}

\begin{remark}[\textbf{Relation with measured geodesic laminations}]\label{subsub: strat}
{\rm
It is proved in Proposition~9.1 in \cite{Bon05} that the first area measure of a 
F-regular domain with simplicial singularity is a particular case of so-called 
\emph{measured geodesic stratification}, which are transverse measures generalizing in any dimension 
measured geodesic laminations on $\H^2$ (geodesic stratifications are more general than 
geodesic laminations in any dimension, see Remark~4.18 in \cite{Bon05}).
Those measures are associated to some F-regular domains, 
but it is not known if any F-regular domain gives a transverse measure on $\H^d$.
The reciprocal is true, see Remark~\ref{rem:mess pol}.
}\end{remark}

\section{The Christoffel problem}\label{sec:sol}

Let $\mu$ be a positive Radon measure on $\H^d$. We have seen in Section~\ref{sec:area} that $\mu$ is the first area measure of an F-convex set 
$K$ if and only if the restricted support function $h_K$ of $K$ is a continuous function which satisfies
\begin{equation*}
\frac{1}{d}\Delta h_K -  h_K = \mu
\end{equation*}
in the sense of distribution on $\H^d$, and such that its $1$-homogeneous extension 
$H_K(\eta)=\|\eta\|_-h_K(\eta/\|\eta\|_-)$ 
is a  convex function on $\mathcal F$.
 In this section we will discuss the existence of explicit solutions to the  equation above, 
as well as possible conditions which guarantee the convexity and the uniqueness of the solution.
Those solutions will be compared to a polyhedral construction of a convex solution in \ref{sub:sol pol}.

Due to its specificity, the $d=1$ case will be treated at the end of this section, so all the reminder concerns the $d>1$ case.

\subsection{Regular first area measures}\label{sub:regsol}

Here we look for an explicit solution to \eqref{eq_hn} when $\mu=\phi \mathrm{d}\H^d$ for some function $\phi\in C^{\infty}_c(\H^d)$.
 
We define $k:(0,+\infty)\to(-\infty,0)$ as
\begin{equation}\label{small_k}
	k(\rho)=\frac{\cosh\rho}{v_{d-1}}\int_{+\infty}^{\rho}\frac{\mathrm{d}t}{\sinh^{d-1}(t)\cosh^2(t)},
\end{equation}
with $v_{d-1}$ the area of $\mathbb{S}^{d-1}\subset \R^d$, and we observe that $k$ is solution of the ODE
\begin{equation}\label{ODE-gamma}
	\ddot k(\rho)+\frac{\dot{A}(\rho)}{A(\rho)}\dot k(\rho)-d k(\rho)=0,
\end{equation}
where 
\begin{equation*}
	A(\rho)=\int_{\partial B_\rho(x)}\mathrm{d}A_\rho=v_{d-1}\sinh^{d-1}\rho
\end{equation*}
is the area of the (smooth) geodesic sphere 
\[
\partial B_\rho=\{y\in\hh^d : d_{\hh^d}(x,y)=\rho\}
\]
centered at any point $x\in\hh^d$ and $\mathrm{d}A_\rho$ is the $(d-1)$-dimensional volume measure on 
$\partial B_\rho$. Finally, 
we introduce the kernel function $G:\hh^d\times\hh^d\to\rr\cup\{\infty\}$ given by
\begin{equation}\label{green}
	G(x,y)= k(d_{\hh^d}(x,y)).
\end{equation}
For later purposes observe that there exists positive constants $C_1$ and $C_2$ depending on $d$ such that
\begin{equation}\label{asym_k}
-k(\rho)\stackrel{\rho\to\infty}{\sim} C_1e^{-d\rho},\qquad -k(\rho)\stackrel{\rho\to 0}{\sim} \begin{cases}C_2\rho^{2-d},&\textrm{if }d>2\\-C_2\log(\rho),&\textrm{if }d=2,\end{cases}
\end{equation}
and 
\begin{equation}\label{asym_A}
A(\rho)\stackrel{\rho\to\infty}{\sim} \frac{v_{d-1}}{2^{d-1}} e^{(d-1)\rho},\qquad A(\rho)\stackrel{\rho\to 0}{\sim} v_{d-1}\rho^{d-1}.
\end{equation}
Accordingly, for each fixed $x\in\hh^d$
\begin{equation}\label{GL1}
\int_{\hh^d}|G(x,y)|\mathrm{d}\H^d(y)=\int_0^\infty |k(\rho)|A(\rho)\mathrm{d}\rho <+\infty,	
\end{equation}
so that if $\psi:\hh^d\to\rr$ is a measurable bounded function we can write
\begin{equation}\label{fubini}\begin{aligned}
	\int_{\hh^d}G(x,y)\psi(y)\mathrm{d}\H^d(y)&=
	\int_0^{+\infty}\left(\int_{\partial B_\rho(x)}G(x,z)\psi(z)\mathrm{d}A_\rho(z)\right)
 \mbox{d}\rho \\
	&= \int_0^{+\infty} k(\rho)\int_{\partial B_\rho(x)}\psi(z)\mathrm{d}A_\rho(z)\mathrm{d}\rho\nonumber.
\end{aligned}\end{equation}

\begin{theorem}\label{lem_hn}
Let $\phi\in C^{\infty}_c(\hh^d)$.
Then, a particular solution to \eqref{eq_hn} is given by the function $h_\phi\in C^{\infty}(\hh^d)$ defined as
\begin{equation}\label{th}
	h_\phi(x)=d\int_{\hh^d}G(x,y)\phi(y)\mathrm{d}\H^d(y).
\end{equation}
\end{theorem}

\begin{remark}\label{rem: helga firey et co}{\rm
We proceed for the proof as in \cite{So81} (similar computations was performed also in \cite{LS06}).
 Actually all these proofs are essentially based on the work of Helgason \cite{He59}, which gave 
the solution of the Poisson problem $\Delta h=\phi$ on $\hh^d$ for compactly supported data $\phi$.\\
In the regular compact case, a different approach was proposed by Firey in \cite{Fir67}. 
Let $\chi_K:\Sph^d\to\R^{d+1}$ 
be the normal representation of a compact convex set $K$ with $C^2$ support function
($\chi_K(\Sph^d)$ is the boundary of $K$ and $\eta$ is an outer normal to $K$ at $\chi_K(\eta)$, 
namely  $\chi_K$ is the Euclidean gradient of the extended support function of $K$). 
Then, once we have defined $\Phi$ the $(-1)$-homogeneous extension of $\phi$, we get that
 $\chi_K$ satisfies the system of uncoupled Poisson equations $\Delta_{\Sph^d}\chi^i=\partial_i\Phi$. Actually, the techniques introduced by Firey to solve this problem seem hardly generalizable to the study of F-convex sets, due to the non compactness of $\hh^d$. Nevertheless, one could try to reproduce Firey's approach to our context, and use Helgason's analysis of the Poisson problem on $\H^d$ to get a proof of Theorem \ref{lem_hn} for smooth compactly supported $\phi$.\\
To conclude this remark, it's worthwhile to recall that Sovertkov proposed also a further method to prove the existence of a solution to \eqref{eq_hn}, \cite{So83}. However this latter is based on a compactness argument which permits to extract a function from the solutions of the problem on a sequence of compact balls exhausting $\hh^d$. Accordingly the obtained solution has no explicit expression.
}
\end{remark}

\begin{proof}
For each $\psi\in C^{2}(\hh^d)$ and for each real $\rho>0$, we introduce the mean value operator $M_\rho(\psi;x)$, defined for $x\in\hh^d$ as
\begin{equation*}
	M_\rho(\psi;x):=\frac1{A(\rho)}\int_{\partial B_\rho(x)}\psi \mbox{d}A_\rho.
\end{equation*}
More generally, one could define $\mathfrak M_\psi:\hh^d\times\hh^d\to\rr$ as 
\[
\mathfrak M_\psi(y,x):=M_{d_{\hh^d}(x,y)}(\psi;x).
\]
According to Lemma 22 in \cite{He59}, it holds that 
\begin{equation}\label{commut}
\Delta_1 \mathfrak M_\phi(x,y)=\Delta_2 \mathfrak M_\phi(x,y),	
\end{equation}
where $\Delta_1$ and $\Delta_2$ are the Laplace-Beltrami operator of $\hh^d$ acting respectively on the
first and second $\hh^d$ component of $\mathfrak M_\phi$. Choosing on $\hh^d$ spherical coordinates $(\rho,\mathbf{\theta})$ 
centered at $x$, standard computations show that (see e.g.~\cite[X.7.2]{Hel62}), for every $\psi\in C^2(\hh^d)$, it holds
\begin{equation}\label{rad_lapl}
\Delta_{\hh^d}\psi(\rho,\mathbb\theta)=\partial_{\rho\rho}\psi(\rho,\mathbb\theta)+\frac{\dot A(\rho)}{A(\rho)}\partial_\rho\psi(\rho,\mathbb\theta)+\frac1{\sinh^2\rho}\Delta_{\partial B_{\rho}(x)}\psi(\rho,\mathbb\theta).
\end{equation}
Accordingly, since $\mathfrak M_\psi(y,x)$ depends on $d_{\hh^d}(x,y)$, but not on the angular coordinates $\mathbb\theta$ of $y$, relations \eqref{commut} and \eqref{rad_lapl} give
\begin{equation}\label{mean radial}
	\Delta_{\hh^d} M_\rho(\psi;x)=\Delta_2\mathfrak M_\psi(y,x)=\Delta_1\mathfrak M_\psi(y,x)=\partial_{\rho\rho}	M_\rho(\psi;x)
	+\frac{\dot A(\rho)}{A(\rho)}\partial_\rho	M_\rho(\psi;x),
\end{equation}
where $y=(\rho,\mathbb\theta)$ is any chosen point on $\partial B_\rho(x)$.\\
Now, from \eqref{green}, \eqref{GL1} and \eqref{th}, $h_\phi$ is well-defined and we have
\begin{equation*}
	h_\phi(x)=d\int_{0}^\infty k(\rho)\int_{\partial B_\rho(x)}\phi(y)\mbox{d}A_\rho(y)\mbox{d}\rho
=d\int_{0}^{+\infty} k(\rho)A(\rho)M_\rho(\phi;x)\mbox{d}\rho.
\end{equation*}
We claim that we can differentiate under the integral sign, i.e. that for all $x_0\in\H^d$
\begin{equation}\label{exchange}
\Delta_{\hh^d}h_\phi(x_0) 
=d\int_{0}^{+\infty} k(\rho)A(\rho)\Delta_{\hh^d}M_\rho(\phi;x_0)\mbox{d}\rho.
\end{equation}
To prove this claim, let $K\subset \H^d$ be a compact set containing $x_0$, endowed with a local coordinate chart, and let $\partial_x^\alpha$ be a partial derivative of order $0\leq|\alpha|\leq 2$. The function $h_\phi$ can be written as 
$$h_\phi(x)=d\int_0^\infty k(\rho)\frac{A(\rho)}{\mathbb{A}(\rho)}\int_{\partial \mathbb B_\rho(0)}\phi(\exp_x(y))d\mathbb A_{\rho}(y) d\rho,$$
where $\mathbb B_\rho(0)$ is the Euclidean ball of radius $\rho$, $\mathbb A_{\rho}$ its area, $d\mathbb A_{\rho}$ is its area measure and $\exp_x$ the exponential map of $\H^d$ at $x$.
Since $\phi\in C^\infty_c$, there exists a constant $R=R(K)>0$ such that 
\begin{equation*}
\int_{\partial \mathbb B_\rho(0)}\phi(\exp_x(y))d\mathbb A_{\rho}(y)=0,\quad\forall (\rho,x)\in[R,\infty)\times K.
\end{equation*}
On the other hand, by compactness there exists a constant $C>0$ such that $|\partial_x^\alpha\phi(\exp_x(y))|\leq C$ for all $(\rho,x)\in[0,R+1]\times K$, $y\in \partial \mathbb B_\rho(0)$, and $0\leq |\alpha|\leq 2$. Then
\begin{equation*}
\left|\partial_x^\alpha\int_{\partial \mathbb B_\rho(0)}\phi(\exp_x(y))d\mathbb A_{\rho}(y)\right| = \left|\int_{\partial \mathbb B_\rho(0)}\partial_x^\alpha\phi(\exp_x(y))d\mathbb A_{\rho}(y)\right|\leq C\mathbb A(\rho).
\end{equation*}
Hence, for all $(\rho,x)\in(0,\infty)\times K$, thanks to \eqref{asym_k} and \eqref{asym_A} we have
$$\left|\partial_x^\alpha\left[k(\rho)\frac{A(\rho)}{\mathbb{A}(\rho)}\int_{\partial \mathbb B_\rho(0)}\phi(\exp_x(y))d\mathbb A_{\rho}(y)\right]\right|\leq C |k(\rho)A(\rho)|\in L^1((0,\infty))$$
so that
$$
\partial_x^\alpha h_\phi(x_0) 
=d\int_{0}^{+\infty}\partial_x^\alpha\left[ k(\rho)A(\rho) M_\rho(\phi;x_0)\right]\mbox{d}\rho
=d\int_{0}^{+\infty} k(\rho)A(\rho)\partial_x^\alpha M_\rho(\phi;x_0)\mbox{d}\rho
$$
and \eqref{exchange} is proven.
Now, thanks to \eqref{mean radial},
\begin{equation*}\begin{aligned}
	&\frac1d\Delta_{\hh^d}h_\phi(x)-h_\phi(x)\\
	&=\int_{0}^{+\infty} k(\rho)A(\rho)\left[\Delta_{\hh^d}M_\rho(\phi;x)-dM_\rho(\phi;x)\right]\mbox{d}\rho\nonumber\\
	&=\int_{0}^{+\infty} k(\rho)A(\rho)\left[\partial_{\rho\rho}	M_\rho(\phi;x)
	+\frac{\dot A(\rho)}{A(\rho)}\partial_\rho	M_\rho(\phi;x)-dM_\rho(\phi;x)\right]\mathrm{d}\rho\nonumber\\
	&=\int_{0}^{+\infty} k(\rho)\partial_{\rho}\left[A(\rho)	\partial_{\rho}M_\rho(\phi;x)\right]\mathrm{d}\rho-d\int_{0}^{+\infty} k(\rho)A(\rho)M_\rho(\phi;x)\mathrm{d}\rho\nonumber
\end{aligned}\end{equation*}
An integration by parts and \eqref{ODE-gamma} yield
\begin{equation*}\begin{aligned}
	&\frac1d\Delta_{\hh^d}h_\phi(x)-h_\phi(x)\\
	&=\left. k(\rho)A(\rho)	\partial_{\rho}M_\rho(\phi;x)\right|^{\rho=+\infty}_{\rho=0}-\int_{0}^{+\infty}\dot k(\rho)A(\rho)	
	\partial_{\rho}M_\rho(\phi;x)\mathrm{d}\rho\nonumber\\
	&-d\int_{0}^{+\infty} k(\rho)A(\rho)M_\rho(\phi;x)\mathrm{d}\rho\nonumber\\
	&=\left. k(\rho)A(\rho)	\partial_{\rho}M_\rho(\phi;x)\right|^{\rho=+\infty}_{\rho=0}
	\left.-\dot k(\rho)A(\rho)M_\rho(\phi;x)\right|^{\rho=+\infty}_{\rho=0}\nonumber\\
	&	+\int_{0}^{+\infty}\partial_{\rho}\left[\dot k(\rho)A(\rho)\right]	
	M_\rho(\phi;x)\mathrm{d}\rho-d\int_{0}^{+\infty} k(\rho)A(\rho)M_\rho(\phi;x)\mathrm{d}\rho\nonumber\\
	&=\left. k(\rho)A(\rho)	\partial_{\rho}M_\rho(\phi;x)\right|^{\rho=+\infty}_{\rho=0}
	\left.-\dot k(\rho)A(\rho)M_\rho(\phi;x)\right|^{\rho=+\infty}_{\rho=0}\nonumber\\
	&	+\int_{0}^{+\infty}\left[A(\rho)\ddot k(\rho)+\dot A(\rho)\dot k(\rho)-\mathrm{d} k(\rho)A(\rho)\right]M_\rho(\phi;x)\mathrm{d}\rho\nonumber\\
		&=\left. k(\rho)A(\rho)	\partial_{\rho}M_\rho(\phi;x)\right|^{\rho=+\infty}_{\rho=0}
	\left.-\dot k(\rho)A(\rho)M_\rho(\phi;x)\right|^{\rho=+\infty}_{\rho=0}\nonumber.
\end{aligned}\end{equation*}
Now, observe that 
\begin{equation}\label{cond_dec_M}
\begin{array}{ll}&\left|\partial_\rho	M_\rho(\phi;x)\right|=
 A(1)^{-1}\left|\partial_\rho\int_{\partial B_1(x)}\phi(\exp_x(\rho\exp_x^{-1}( y))\mathrm{d}A_1(y)\right|\leq \max_{\partial B_\rho(x)}|\nabla\phi|.\\
&\left|M_\rho(\phi;x)\right| \leq \max_{\partial B_\rho(x)}|\phi|.
\end{array}\end{equation}
Moreover, applying l'H\^opital's rule, we get that $\dot k(\rho)=O(k(\rho))$ as $\rho\to\infty$ and 
\[
\lim_{\rho\to 0} k(\rho)A(\rho)=0\quad\textrm{ and }\quad\lim_{\rho\to 0}\dot k(\rho)A(\rho)=1.
\]
Since $\phi\in C^\infty_c$, \eqref{cond_dec_M} implies
\[
\lim_{\rho\to 0} k(\rho)A(\rho)	\partial_{\rho}M_\rho(\phi;x)=\lim_{\rho\to+\infty} k(\rho)A(\rho)
	\partial_{\rho}M_\rho(\phi;x) = \lim_{\rho\to+\infty}
\dot k(\rho)A(\rho)M_\rho(\phi;x) =0
\]
and
\[
\lim_{\rho\to 0} \dot k(\rho)A(\rho)M_\rho(\phi;x) =\phi(x),
\]
so that
\[
\frac1d\Delta_{\hh^d}h_\phi(x)-h_\phi(x)=\phi(x)
\]
as aimed. Finally, since $\phi\in C^\infty$, by standard elliptic regularity we get $h_\phi\in C^\infty(\hh^d)$.
\end{proof}

\begin{remark}[\textbf{Geometric interpretation}]\label{rem:herissons}{\rm
Let $\varphi$ as in Theorem~\ref{lem_hn}.
We do not know if
$h_{\varphi}$ is the support function of an F-convex set. But
the solution \eqref{th} can be written as $h_{\varphi}(x)=\langle x, \chi(x) \rangle_- $
with, for $x\in\F$,
$$\chi(x)= -\frac{d}{v_{d-1}}\int_{\H^d} y \varphi(y)\int_{+\infty}^{\mathrm{acosh}(-\langle \frac{x}{\|x\|_-},y\rangle_-)}
\frac{\mathrm{d}t}{\sinh^{d-1}(t)\cosh^2(t)}\mathrm{d}\H^d(y).$$
This is the normal representation of a $C^2$ F-hedgehog with mean radius of curvature $\varphi$,
see Subsection~\ref{subsec:hedg}. Hedgehogs appear naturally when the Christoffel problem is considered, under different names.
In the smooth setting, they are also called \emph{generalized envelopes}, see \cite{oli92} and the references inside.
See also Remark~\ref{rem:mess pol}.
}\end{remark}

\subsection{Distributions solutions}

Let $\mathcal R(\hh^d)$ be the set of the  Radon measures  $\mu$ on $\hh^d$ and define $\mathcal R^+(\hh^d)$ as the subset of
 measures satisfying the additional condition
\begin{equation}\label{cond_dec_meas}
	\int_{\hh^d\setminus B_1(x_0)}|G(x_0,y)|\mathrm{d}\mu(y)<+\infty
\end{equation}
for some (hence any) $x_0\in\hh^d$.\\
Each $\mu\in\mathcal R^+(\hh)$ can be seen as the distribution called, with a standard abuse of notation, also $\mu\in \mathcal D'(\hh)$, and whose action is given by
\begin{equation}\label{mu-dist}
(\mu,f)=\int_{\hh^d}f(x)\mathrm{d}\mu(x),\qquad \forall f\in \mathcal D(\hh^d)=C^{\infty}_c(\hh^d).
\end{equation}

\begin{remark}{\rm
We note that, in case $\mu=\phi \mathrm{d}\H^d$ is given as a $C^2$ function on $\hh^d$, thanks to \eqref{asym_k} and \eqref{asym_A}, condition \eqref{cond_dec_meas} is implied by 
\begin{equation}\label{cond_dec_e}
e^{-d_{\hh^d}(x_0,x)}\max\{|\phi(x)|;|\nabla\phi(x)|\} \in L^1(\H^d\setminus B_1(x_0)).
\end{equation}
In particular our assumption \eqref{cond_dec_meas} is weaker than the conditions required by Sovertkov \cite{So81} and Lopes de Lima and Soares de Lira \cite{LS06}.}
\end{remark}

\begin{theorem}\label{th_dist}
Let $\mu\in\mathcal{R^+}(\hh^d)$ and consider the equation
\begin{equation}\label{eq_dist}
\frac{1}{d}\Delta h - h = \mu 
\end{equation}
in the sense of distributions on $\hh^d$. Then, a particular solution to \eqref{eq_dist} is given by the distribution $h_\mu\in \mathcal D'(\hh^d)$ defined formally as
\begin{equation}\label{def_h}
h_\mu(x):=d\int_{\hh^d}G(x,y)\mathrm{d}\mu(y),
\end{equation}
and whose action is defined by \eqref{action}.
\end{theorem}

\begin{corollary}\label{coro_C0}
Let $\phi\in C^{k,\alpha}(\hh^d)$, $0\leq k$, $0\leq\alpha<1$. Assume that there exists $x_0\in\hh^d$ such that
\begin{equation*}\label{cond_dec}
	\int_{\hh^d}|G(x_0,y)|\phi(y)\mathrm{d}\H^d(y)<+\infty
\end{equation*}
Then \eqref{eq_hn} has a solution given by
\begin{equation*}
	h_\phi(x)=d\int_{\hh^d}G(x,y)\phi(y)\mathrm{d}\hh^d(y).
\end{equation*}
Moreover, $h_\phi\in C^{k+2,\alpha}(\hh^d)$ if $\alpha>0$ and $h_\phi\in C^{1,\beta}(\hh^d)$ for all $\beta<1$ if $\alpha=k=0$.
\end{corollary}
\begin{remark}{\rm
It is not hard to see (cf. \eqref{well-def}) that if $\phi$ is $\Gamma$-invariant, 
then also the solution $h_\phi$ is $\Gamma$ invariant, see Subsection~\ref{sub:fuchssol} and the proof of Theorem \ref{th_smooth} for more details. 
On the other hand if $K$ is a $C^2_+$ $\tau$-F-convex set, it follows from Lemma~\ref{lem: hess gamma inv} (or more generally from Remark~\ref{rem: gamma 0 inv mes})
that its mean radius of curvature  is $\Gamma$-invariant. 
In particular the support function of a $\tau$-F-convex set can not be recovered by Corollary~\ref{coro_C0}. 
This is a first evidence of the non-uniqueness of the solutions, that will be further discussed in the subsequent sections.
}\end{remark}

\begin{proof}[Proof of Theorem~\ref{th_dist}]
Given $\mu\in\mathcal R^+(\hh^d)$, the distribution $h\in \mathcal D'(\hh^d)$ is a solution to \eqref{eq_dist} if and only if
\[
(h,\frac{1}{d}\Delta f-f)=(\mu,f)\qquad \forall f\in \mathcal D(\hh^d).
\]
Define formally 
\begin{equation*}
h_\mu(x):=d\int_{\hh^d}G(x,y)\mathrm{d}\mu(y). 	
\end{equation*}
We claim that $h_\mu\in \mathcal D'(\hh^d)$, its action being defined by 
\begin{equation}\label{action}
(h_\mu,f):=(\mu,h_f)=\int_{\hh^d} h_f(x)\mathrm{d}\mu(x),
\end{equation}
where $h_f(x)=d\int_{\hh^d}G(x,y)f(y)\mathrm{d}{\hh^d}(y)$ is the smooth solution to $\frac{1}{d}\Delta h_f-h_f=f$ given by Theorem \ref{lem_hn}. To this end, note that for each compact set $K\in\H^d$ and $f\in \mathcal D (K)$, 
\begin{equation*}
	\left|h_f(x)\right|
	\leq\begin{cases}d|K|\left\|f\right\|_\infty\left|k(d_{\hh^d}(x),\operatorname{supp}f)\right|,&\textrm{if}\ d_{\hh^d}(x,K)>1,\\
	d\left\|f\right\|_\infty\left\|G(x,\cdot)\right\|_{L^1(\hh^d)},&\textrm{if}\ d_{\hh^d}(x,K)\leq 1\end{cases},
\end{equation*}
where $|K|$ is the hyperbolic volume of $K$. Then, choosing $x_0$ in the interior of $K$, thanks to \eqref{cond_dec_meas} and to the monotonicity of $k$ we have
\begin{equation}\label{well-defi}\begin{aligned}
	&\int_{\hh^d} h_f(x)\mathrm{d}\mu(x)\\ 
	&\leq d\left\|f\right\|_\infty\left[\mu\left(\{x:d_{\hh^d}(x,K)\leq 1\}\right)\left\|G(x,\cdot)\right\|_{L^1(\hh^d)} 
	+ |K|\int_{\{x:d_{\hh^d}(x,K)> 1\}}|k(d_{\hh^d}(x,K))|\mathrm{d}\mu(x)\right]\\
	&\leq d\left\|f\right\|_\infty\left[\mu\left(\{x:d_{\hh^d}(x,K)\leq 1\}\right)\left\|G(x,\cdot)\right\|_{L^1(\hh^d)} 
	+ |K|\int_{\{x:d_{\hh^d}(x,K)> 1\}}|G(x,x_0)|\mathrm{d}\mu(x)\right]\\
	&\leq C_K\left\|f\right\|_\infty<+\infty,
\end{aligned}\end{equation}
where the constant
$$
C_K:=d\left[\mu\left(\{x:d_{\hh^d}(x,K)\leq 1\}\right)\left\|G(x,\cdot)\right\|_{L^1(\hh^d)} 
	+ |K|\int_{\H^d\setminus B_1(x_0)}|G(x,x_0)|\mathrm{d}\mu(x)\right]
$$
is independent of $f$. Then \eqref{action} is well defined and the functional $h_\mu$ on $\mathcal D(\hh^d)$ is linear by construction and continuous because of \eqref{well-defi}. Also, it's worthwhile to observe that \eqref{action} is the natural action for $h_\mu$, 
as it is shown by the case $\mu=\phi \mathrm{d}{\hh^d}$ when $\phi\in C^2_c(\hh^d)$. \\
We want to prove that $h_\mu$ is a solution of \eqref{eq_dist}. To this end, let $f_1,f_2\in \mathcal D(\hh^d)=C^{\infty}_c(\hh^d)$ and compute
\begin{equation}\begin{aligned}\label{equiv}
	&=\int_{\hh^d} \left(d\int_{\hh^d}G(x,y)f_2(x)\mathrm{d}{\hh^d}(x)\right)\left[\frac1d\Delta f_1-f_1\right](y)\mathrm{d}{\hh^d}(y)\\
	&=\int_{\hh^d} h_{f_2}(y)\left[\frac1d\Delta f_1-f_1\right](y)\mathrm{d}{\hh^d}(y)\\
	&=\int_{\hh^d}\left[\frac1d\Delta h_{f_2}(y)-h_{f_2}(y)\right]f_1(y){\mathrm{d}\hh^d}(y)\\
	&=\int_{\hh^d}f_2(y)f_1(y)\mathrm{d}{\hh^d}(y),	
\end{aligned}\end{equation}
where we have applied Fubini and Stokes' theorems, and $h_{f_2}\in C^{\infty}(\hh^d)$ is the solution given by Theorem \ref{lem_hn}. Since $f_2\in \mathcal D(\hh^d)$ is arbitrary, \eqref{equiv} says that for a $f\in\mathcal D(\hh^d)$ one has $h_{\frac1d\Delta f-f}=f$. Then
\begin{equation}
	(\frac{1}{d}\Delta h_\mu-h_\mu,f)= (h_\mu,\frac1d\Delta f -f)
 =\int_{\hh^d} h_{\frac1d\Delta f-f}(x)\mathrm{d}\mu(x)
	=\int_{\hh^d} f(x) \mathrm{d}\mu(x) = (\mu,f).
\end{equation}
\end{proof}

\begin{proof}[Proof of Corollary~\ref{coro_C0}]
Let $\mu=\phi \mathrm{d}{\hh^d}$. Then, according to Theorem \ref{th_dist}, 
\[
h_\phi(x):=d\int_{\hh^d}G(x,y)\mathrm{d}\mu(y)=d\int_{\hh^d}G(x,y)\phi(y)\mathrm{d}{\hh^d}(y)
\]
is a distribution solution to  \eqref{eq_dist}. 
If $\phi\in C^{k,\alpha}(\hh^d)$ for $\alpha>0$, 
then the conclusion follows directly from Theorem~3.54 in \cite{Au98}. 
More generally, if $\phi\in C^{0}(\hh^d)$, then clearly $\phi\in L^p_{loc}(\hh^d)$ for 
all $p<\infty$. Applying again Theorem 3.54 in \cite{Au98} we get that $\phi\in W^{2,p}_{loc}(\hh^d)$. 
Hence, up to choose $p$ large enough, we get that $\phi\in C^{1,\beta}(\hh^d)$ 
for all $\beta<1$ thanks to Sobolev embedding (see Theorem 2.10 in \cite{Au98}).
\end{proof}

\begin{example}[\textbf{The elementary example}]\label{rem:elementary dist}{\rm 
We are given a measure on $\H^2$ which is a weight $a$ on a geodesic $\gamma$.
It separates $\H^2$ into $\O_1$  and $\O_2$.
Let us denote by $v$ the unit space-like vector orthogonal to the time-like hyperplane defining $\gamma$ and pointing
to $\O_2$. 
Let $h_\mu$ be the analytic solution proposed in \eqref{def_h}. 
Since $h_\mu|_{\O_i}$ is smooth, 
it makes sense to write
\begin{equation*}
	h_\mu|_{\O_i}(x)=\int_{\hh^d}G(x,y)\mathrm{d}\mu(y)=\int_{-\infty}^\infty a k(d_{\hh^2}(x,\gamma(t)))\mathrm{d}t.
\end{equation*}
It is clear that $h_\mu|_{\O_i}(x)$ depends only on $d_{\hh^2}(x,\gamma)$. First of all, in dimension $d=2$ by \eqref{small_k} we have the explicit expression
\[
k(\rho)=\frac1{2\pi}\left[1+\frac{\cosh(\rho)}{2}\log\left(\frac{\cosh(\rho)-1}{\cosh(\rho)+1}\right)\right]
\]
By the hyperbolic Pythagorean theorem \cite{Thurcour1}
\[
\cosh (d_{\H^2}(x,\gamma(t))=\cosh(t) b(x),
\]
where $b(x):=\cosh(d_{\mathbb{H}^2}(x,\gamma))$ is independent of $t$. Note also that $b(x)$ has the following geometric interpretation:
 $\sinh(d_{\mathbb{H}^2}(x,\gamma))=\epsilon\langle x,v \rangle_-$, where $\epsilon=1$ if $x$ and $v$ are on the 
same side of $\gamma$, and $\epsilon=-1$ otherwise. So
$$\cosh(d_{\mathbb{H}^2}(x,\gamma))=\sqrt{1+\langle x,v\rangle_-^2}. $$
Consider the halfspace model for $\H^2$, i.e. $\H^2=\{(u,w)\in\rr^2\ :\ y>0\}$ endowed with the (conformally Euclidean) 
metric $w^{-2}(du^2+dw^2)$. Without loss of generality, we can suppose that $\gamma(t)=(0,e^t)$.
 With this choice for the coordinates system and for the geodesic, it is easy to obtain 
\[
\sinh(d_{\mathbb{H}^2}((u,w),\gamma))=\epsilon\langle x,v \rangle_-=\frac{|u|}w.
\]
Then
\begin{equation*}
\begin{aligned}
	h_\mu|_{\Omega_i}(x)
	&=\int_{-\infty}^\infty a k(d_{\hh^2}(x,\gamma(t)))\mathrm{d}t\\
	&=\frac{a}{2\pi}\int_{-\infty}^\infty\left[1+\frac{\cosh(t)b(x)}{2}\log\left(\frac{\cosh(t)b(x)-1}{\cosh(t)b(x)+1}\right)\right]\mathrm{d}t\\
	&=\frac{a}{\pi}\left[(b^2(x)-1)^{1/2}\arctan\left((b^2(x)-1)^{-1/2}\right)-1\right]\\
	&=\frac{a}{\pi}\left[\langle x,v \rangle_-\arctan\left((\langle x,v \rangle_-)^{-1}\right)-1\right]\\
	&=\frac{a}{\pi}\left[\frac uw\arctan\left(\frac wu\right)-1\right].
\end{aligned}
\end{equation*}
On the one hand, using the conformal structure of $\H^2$, one can check that as expected
\[
(\Delta h_\mu - 2h_\mu)|_{\O_i}(u,w)=w^2\left(\frac{\partial^2}{\partial u^2}+\frac{\partial^2}{\partial w^2}\right)h_\mu|_{\O_i}(u,w) -2h_\mu|_{\O_i}(u,w) = 0
\]
for $i=1,2$. On the other hand, we have for instance that
\[
\begin{aligned}
(\nabla^2h_\mu|_{\mathcal{O}_i}-gh_\mu|_{\mathcal{O}_i})\left(\frac{\partial}{\partial w},\frac{\partial}{\partial w}\right)
&=\frac{\partial^2}{\partial w^2}h_\mu|_{\mathcal{O}_i}+\frac1{w}\frac{\partial}{\partial_w}h_\mu|_{\Omega_i}-\frac1{w^2}h_\mu|_{\mathcal{O}_i}\\
&=\frac{a}{\pi}\frac{(1-z^2)}{2w^2(1+z^2)^2},
\end{aligned}
\]
where $z=u/w$, and the latter expression is negative for $z$ large enough, which proves that $h_\mu$ is not the support function of a convex set.
But we know that there exists a convex solution, see Example~\ref{ex: elemntary1}. So  \eqref{def_h} 
does not reach all convex solutions. This example is continued in Example~\ref{rem:elementary pol ex} and Example~\ref{ex : surface courbure 0}.

}\end{example}

\subsection{Fuchsian solutions}\label{sub:fuchssol}
 Throughout this section we will use overlined letters to denote objects defined on the compact hyperbolic manifold 
$\hG$. For instance, given $\bar\phi:\hh^d/\Gamma\to\rr$ we can define $\phi:\hh^d\to\rr$ as 
$\phi=\Pi_{\Gamma}\circ \bar\phi$, where $\Pi_{\Gamma}:\hh^d\to\hG$ 
is the covering projection. The precise meaning of overlined symbols will be specified time by time.

\begin{theorem}\label{th_smooth}
Let $0<\bar\phi\in C^{k,\alpha}(\hG)$ for some $k\geq0$ and $0\leq\alpha<1$. Then the equation
\begin{equation}\label{main_eq}
	\frac{1}{d}\Delta \bar h - \bar h =\bar \phi
\end{equation}
on $\hh^d/\Gamma$ has a unique solution $\bar h_{\bar\phi}$ defined for all $\bar x\in \hh/\Gamma$ as
\begin{equation}\label{h}
\bar h_{\bar\phi}(\bar  x)=d\int_{\hh^d}G(x,y)\phi(y)\mathrm{d}{\hh^d}(y),
\end{equation}
where $x\in \Pi_{\Gamma}^{-1}(\bar x)$ and $\phi=\bar\phi\circ P_{\Gamma}$.\\
Moreover, $\bar h_{\bar\phi}\in C^{k+2,\alpha}(\hh^d)$ if $\alpha>0$ and $\bar h_{\bar\phi}\in C^{1,\beta}(\hh^d)$ for all $\beta<1$ if $\alpha=k=0$.
\end{theorem}

\begin{proof}
Consider $\phi=\bar\phi\circ P_{\Gamma}\in C^{k,\alpha}(\hh^d)$. By definition $\phi$ is $\Gamma$-invariant, i.e. $\phi(x)=\phi(\gamma x)$ for all $\gamma\in\Gamma$ and $x\in\hh^d$. Moreover
\[
0<\bar \phi_\ast:=\min_{\hG}\bar  \phi\leq\phi\leq\max_{\hG}\bar \phi=:\bar \phi^\ast<\infty,
\]
so that condition \eqref{cond_dec_meas} is satisfied. Let 
\[
h_\phi(x)=d\int_{\hh^d}G(x,y)\phi(y)\mathrm{d}{\hh^d}(y),
\]
be the solution to equation \eqref{eq_hn} given by Theorem \ref{lem_hn} and Corollary \ref{coro_C0}. Then $h_\phi$ is $\Gamma$-invariant. In fact, for all $x\in\hh^d$ and $\gamma\in\Gamma$ it holds
\begin{equation}\label{well-def}\begin{aligned}
	h_\phi(\gamma x)&=d\int_{\hh^d}G(\gamma x,y)\phi(y)\mathrm{d}{\hh^d}(y)\\
	&=d\int_{\hh^d}G(\gamma x,\gamma y)\phi(\gamma y)\mathrm{d}{\hh^d}(\gamma y)\qquad \textrm{(by a change of variable)}\\
	&=d\int_{\hh^d}G(x,y)\phi(\gamma y)\mathrm{d}{\hh^d}( y)\qquad \textrm{(since $\gamma$ is an isometry of $\hh^d$)}\\
	&=d\int_{\hh^d}G(x,y)\phi(y)\mathrm{d}{\hh^d}(y)\qquad \textrm{(by construction of $\phi$)}\\
	&=h_\phi(x).
\end{aligned}\end{equation}
Accordingly, $\bar h_{\bar\phi}=h_\phi\circ P_\gamma^{-1}$ is a well defined function on $\hG$, it has the form given in \eqref{h} and it is a solution of \eqref{main_eq} since $P_{\Gamma}$ is a (local) Riemannian isometry.
 
\end{proof}
Now, let $\mathcal R(\hG)$ be the set of the positive finite Radon measures on $\hG$. As for \eqref{mu-dist} 
we have \linebreak$\mathcal{R}(\hG)\subset\mathcal{D}'(\hG)$, the space of distributions on $\hG$. Then, given $\bar \mu\in\mathcal R(\hG)$, we can consider the equation
\begin{equation}\label{eq_dist_Gamma}
	\frac{1}{d}\Delta \bar h -  \bar h = \bar \mu,\qquad\textrm{in } \mathcal D'(\hG).
\end{equation}
We want to show that, as in the regular case, a solution to this latter can be obtained by projecting to $\hG$ a solution of \eqref{eq_dist}.

Let $\epsilon>0$ such that $B_\epsilon(\bar x)\subset\hG$ is a convex geodesic ball for each $\bar x\in\hG$. The compactness of $\hG$ implies that such an $\epsilon$ exists, and that the open covering $\{B_\epsilon(\bar x)\}_{\bar x\in\hG}$ admits a finite subcovering $\{B_\epsilon(\bar x_j)\}_{j\in J}$, $|J|<\infty$. 
Fix points $x_j$ in the fibers over $\bar x_j$, i.e. $P(x_j)=\bar x_j$ for all $j\in J$. Then $\{B_{\epsilon}(\gamma x_j)\}_{\gamma\in\Gamma,j\in J}$ is a locally finite open covering of $\hh^d$ such that $B_{\epsilon}(\gamma x_j)\cap B_{\epsilon}( x_j)=\emptyset$ for all $\gamma\in\Gamma$ and all $j\in J$.\\
Given $\overline{\mu}\in\mathcal R(\hG)$, we can define 
$\mu:=P_{\Gamma}^{\ast}\bar\mu\in\mathcal{R^+}(\hh^d)$ as the pull-back measure of $\bar\mu$ through the projection $P_{\Gamma}$. 
Since $P_{\Gamma}$ is a Riemannian submersion, $\mu$ is well-defined. Namely, one can first define the action of $\mu$ on Borel-measurable set $A\subset B_\epsilon (\gamma x_j)$ 
for some $\gamma\in\Gamma$ and $j\in J$, as $\mu(A)=\bar \mu(P_{\Gamma}(A))$. For general $A\subset\hh^d$, one uses the sheaf property of distributions.

We note that $\mu$ is $\Gamma$-invariant, i.e. 
\begin{equation}\label{A-inv}
\mu(\gamma A):=\mu(\{\gamma x:x\in A\})=\mu(A)	
\end{equation}
for every measurable set $A$. In fact $\gamma$ acts as an isometry on $\hh^d$, and \eqref{A-inv} is true by definition for any $A\subset B_\epsilon (\gamma x_j)$.\\
Similarly, consider a distribution $ T\in\mathcal D'(\hh^d)$ which is $\Gamma$-invariant, i.e.~such that 
$(T,f)=(T,f\circ\gamma)$ for all $\gamma\in\Gamma$ and for every $f\in\mathcal D(\hh^d)$. Then $T$ naturally induces a distribution $\bar T=P_{\Gamma,\ast} T\in\mathcal D(\hG)$ 
as follows. Let $\bar f\in\mathcal{D}(\hG)$. If $\operatorname{supp}\bar f\subset B_\epsilon(\bar x_j)$ for some $j\in J$, 
then we set $(\bar T,\bar f)=(T, \bar f\circ {P_\gamma}|_{B_\epsilon(\gamma x_j)})$ for some $\gamma\in\Gamma$. 
The definition is independent of the choice $\gamma$ because of the $\Gamma$-invariance of $T$. 
For general $\bar f\in\mathcal{D}(\hG)$ we use, as above, the sheaf property of $\mathcal D'(\hG)$.

\begin{theorem}\label{sol_fuchs_distr}
Let $\bar\mu\in\mathcal R(\hG)$. Then a distribution solution to \eqref{eq_dist_Gamma} is given by the distribution $\bar h_{\bar \mu}\in \mathcal D'(\hG)$ defined as 
\begin{equation*}
\bar h_{\bar \mu} = P_{\gamma,\ast} h_{\mu},
\end{equation*}
where $\mu=P_{\Gamma}^\ast\bar\mu\in\mathcal R^+(\hh^d)$ and $h_\mu$ is the distribution solution to equation \eqref{main_eq} given in Theorem \ref{th_dist}.
\end{theorem}

\begin{proof}
Given $\bar \mu\in\mathcal R(\hG)$, we define the "lifting" $\mu:=P^\ast_\gamma\bar\mu\in\mathcal{R^+}(\hh^d)$. Consider the equation 
\begin{equation}\label{eq_barmu}
\frac{1}{d}\Delta h- h=\mu\qquad\textrm{ in }\mathcal D'(\hh^d),
\end{equation}
and let $ h_{\mu}$ be the solution to \eqref{eq_barmu} defined in \eqref{def_h}. Such a solution exists since \eqref{cond_dec_meas} is satisfied because of the $\Gamma$-invariance of $\mu$ and \eqref{GL1}. We have that $ h_{\mu}$ is $\Gamma$-invariant. In fact, if $f\in C^{\infty}_c(\hh^d)$, then reasoning as in \eqref{well-def} we get
\begin{equation*}
h_{f\circ\gamma}(x)= d\int_{\hh^d}G(x,y)f(\gamma y)\mathrm{d}\H^d =d \int_{\hh^d}G(\gamma x,\gamma y)f(\gamma y)\mathrm{d}\H^d
= d\int_{\hh^d}G(\gamma x,y)f(y)\mathrm{d}\H^d(y) = h_{f}(\gamma x)=(h_{f}\circ \gamma)(x)
\end{equation*}
and this latter, together with \eqref{action}, yields
\begin{equation*}
	( h_{\mu},f\circ\gamma)=(\mu,h_{f\circ\gamma})=(\mu,h_{f}\circ\gamma)=(\mu,h_{f})=(h_{\mu},f),
\end{equation*}
by the $\Gamma$-invariance of  $\mu$.\\
Since $ h_{\mu}$ is $\Gamma$-invariant, we can define a distribution $\bar h_{\bar \mu}\in \mathcal D'(\hG)$ as 
\begin{equation}\label{sol_eq_dist_gamma}
\bar h_{\bar \mu} = P_{\gamma,\ast} h_{\mu}.
\end{equation}
Finally, we want to prove that $\bar h_{\bar \mu}$ is a solution to \eqref{eq_dist_Gamma}, i.e. that
\begin{equation}
	(\bar h_{\bar \mu},\frac{1}{d}\Delta\bar f -\bar f)=(\bar \mu,\bar f),\qquad\forall\bar f\in\mathcal{D}'(\hG). 
\end{equation}
To this end, suppose first that $\operatorname{supp}\bar f\subset B_\epsilon(\bar x_j)$ for some $j\in J$. 
Then $\operatorname{supp}(\frac{1}{d}\Delta\bar f-\bar f)\subset B_\epsilon(\bar x_j)$ and
\[
(\bar h_{\bar \mu},\frac{1}{d}\Delta\bar f -\bar f)= ( h_{\mu},\frac{1}{d}\Delta f -f),
\]
where $f= \bar f\circ P_{\Gamma}|_{B_\epsilon(x_j)}\in C^{\infty}_c(B_\epsilon(x_j))$. Moreover, by definition of $ h_{\mu}$
\[
(h_{\mu},\frac{1}{d}\Delta f -f)=(\mu,h_{\frac{1}{d}\Delta f -f}).
\]
Thanks to \eqref{equiv}, we know that since $f\in C^{\infty}_c(\hh^d)$ it holds $h_{\frac{1}{d}\Delta f -f}=f$. Finally, since $f$ is compactly supported in $B_\epsilon(x_j)$, we have
\[
(\mu,h_{\Delta f -df})=(\mu,f)=(\bar \mu,\bar f),
\]
which concludes the proof when $\operatorname{supp}\bar f\subset B_\epsilon(\bar x_j)$. The case of general $f\in\mathcal{D}'(\hG)$ follows by the sheaf property of distributions.
\end{proof}


\subsection{Polyhedral solution}\label{sub:sol pol}

Recall notations and definitions from Subsubsection~\ref{subsubsub pol area}.

\begin{theorem}\label{thm:pol}\textbf{(General case)} Let $\varphi$ be a polyhedral measure of order one on $\H^d$.
If the numbers $\lambda(\zeta)$ are uniformly bounded from below by a positive constant, then
$\varphi$ is the first area measure of a polyhedral F-convex set.

\textbf{(Invariant case)} Let $\overline{\varphi}$ be a Radon measure on a compact hyperbolic manifold $\H^d/\Gamma$
such that  a lift $\varphi $ of $\overline{\varphi}$ is a polyhedral measure of order one on $\H^d$.
Then there exists a cocycle $\tau$ such that $\varphi$ is the first area measure of a polyhedral $\tau$-F-convex set.
\end{theorem}

\begin{remark}[\textbf{The $d=1$ case}]\label{rem pol d=1}{\rm  In this case, the condition \ref{pol: 3} in the definition of polyhedral measure of order one
is void.  The measure is only the data of a countable numbers of points on the non-compact one 
dimensional manifold $\H^1$, with positive weights. From it we construct 
a space-like polygon with edge length the weights. The proof is then the same as the proof of
Minkowski theorem for plane convex compact polygons. See Figure~\ref{fig:pol 3}.
The invariant case is the data of a finite number of points with positive weight 
on a circle of length $t$. We construct a space-like polygon invariant under a group of
isometry whose linear part is the group $\Gamma$ generated by \eqref{eq:boost}. 
From Lemma~\ref{lem:quasi dim}, the polygon is the translate of a $\Gamma$ polygon.
}\end{remark}

\begin{figure}[ht]
\begin{center}
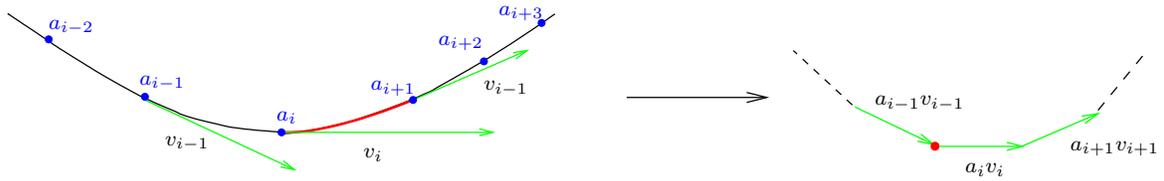
\end{center}
\caption{Proof of Theorem~\ref{thm:pol} in the $d=1$ case. The point on the right hand picture is chosen arbitrarily.
$v_k$ is a unit vector orthogonal to the point of $\H^1$ of weight $a_k$. }\label{fig:pol 3}
\end{figure}

\begin{example}[\textbf{The elementary example}]\label{rem:elementary pol ex}{\rm
Before the general proof, let us illustrate the method  with the elementary example, see Example~\ref{ex: elemntary1}.
We are given a measure on $\H^2$ which is a weight $a$ on a geodesic $\gamma$ (this generalizes immediately
to any dimension, taking a totally geodesic hypersurface instead of a geodesic).
It separates $\H^2$ into $\mathcal{O}_1$  and $\O_2$.
Let us denote by $v$ the unit space-like vector orthogonal to the time-like hyperplane defining $\gamma$ and pointing
to $\O_2$. Choose any point $p_1\in\R^3$, and let us denote by $p_2$ the point 
$p_1+av$. Then the wanted F-convex set is the union of the future cones of the points of the segment
$[p_1,p_2]$. Compare with the analytical solution, Example~\ref{rem:elementary dist}
} 
\end{example}

\begin{proof}

Choose an arbitrarily cell of $C$ and denote it by $\xi_b$.
For any other cell $\xi$, let us define the following vector of $\R^{d+1}$:
if $\xi=\xi_b$ then $X(\xi)=0$, otherwise
$$ X(\xi)=\sum_{i=1}^n \lambda(\xi_{i}\cap\xi_{i+1}) v(\xi_{i},\xi_{i+1}),$$
where
\begin{itemize}[nolistsep,label={\bf(\roman{*})}, ref={\bf(\roman{*})}]
\item[\textbullet] $(\xi_1=\xi_b,\ldots,\xi_n=\xi)$ is a path of cells of $C$, with $\xi_{i}\cap \xi_{i+1}$ a 
codimension $1$ cell of $C$;
\item[\textbullet] $v(\xi_{i},\xi_{i+1})$ is the unit space-like vector normal to the hyperplane of $\R^{d+1}$ defined by $\xi_{i}\cap \xi_{i+1}$,
pointing toward $\xi_{i+1}$.
\end{itemize}

\emph{Fact: $X(\xi)$ does not depend on the choice of the path between $\xi_b$ and $\xi$.}
As $\H^d$ is simply connected, we can go from one path to the other by a finite number of operations
as shown in Figure~\ref{fig: not pol 2}. Clearly the deformation on the left hand figure leaves $X(\xi)$ unchanged
as  $v(\xi_{i},\xi_{i+1})=-v(\xi_{i+1},\xi_{i})$. The deformation on the right hand figure consists of 
changing cells sharing a codimension $2$ cell $\zeta$ by the other cells sharing $\zeta$. Then the result follows from condition
\eqref{eq:sum poly} because  $v(\xi_{i},\xi_{i+1})$ is orthogonal to $u(\zeta,\xi_{i}\cap \xi_{i+1})$. The fact is proved.

\emph{Fact: the set of $X(\xi)$ is discrete.}
Between two points $X(\xi)$ and $X(\xi')$ there is at least a space-like segment given by a vector $\lambda v$, with
$\lambda$ greater than a given positive constant by assumption, and $v$ a unit space-like vector, so its Euclidean norm is $\geq 1$.
The fact is proved.

Let us define, for $\eta\in\F$
$$H(\eta)=\operatorname{max}_{\xi}\langle\eta,X(\xi)\rangle_-.$$
Let $\eta\in \xi_1 $ and $\eta\notin \xi_2$. For  a path of cells $\xi_i$ between $\xi_1 $ and $\xi_2 $,
as $v(\xi_{i},\xi_{i+1})$ points toward $\xi_{i+1}$, $\langle v(\xi_{i},\xi_{i+1}),\eta\rangle_-$ is negative, hence
$$\langle X(\xi_2),\eta\rangle_-=\langle X(\xi_1),\eta\rangle_-+\sum_i \lambda(\xi_{i}\cap\xi_{i+1}) \langle v(\xi_{i},\xi_{i+1}),
\eta\rangle_-<\langle X(\xi_1),\eta\rangle_- $$
so $H(\eta)=\langle\eta,X(\xi_1)\rangle_-$. This says that the decomposition of $\H^d$ induced by $H$ is $C$.
$H$ is the extended support function of the wanted polyhedron, because 
if $\xi$ and $\xi'$ share a codimension $2$ cell, then  there is an edge joining $X(\xi)$ to $X(\xi')$.
This edge is $ X(\xi)-X(\xi')= \lambda(\xi\cap\xi')  v(\xi,\xi')$ and has length $\lambda(\xi\cap\xi') $.
The general part of the theorem is proved.
Note that if the base cell $\xi_b$ is changed, the resulting polyhedron will differ from the former one by a translation.

Now suppose that the data of the cellulation and the $\lambda$ are invariant under the action of 
$\Gamma$. To each $\gamma_0\in\Gamma$ we associate  $\tau_{\gamma_0}:=X(\gamma_0\xi_b)$.
For $\mu_0\in\Gamma$, the path from $\xi_b$ to $\gamma_0\mu_0\xi_b$
is the path from $\xi_b$ to $\gamma_0\xi_b$ followed by the image under $\gamma_0$ of the path from 
$\xi_b$ to $\mu_0\xi_b$. Moreover it is easily checked from the definition of $X$ that
 $$\gamma_0 X(\xi)=\sum_{i=1}^n \lambda(\xi_{i}\cap\xi_{i+1}) \gamma_0 v(\xi_{i},\xi_{i+1})$$
$$= \sum_{i=1}^n \lambda(\xi_{i}\cap\xi_{i+1})  v(\gamma_0\xi_{i},\gamma_0\xi_{i+1}) $$
$$= \sum_{i=1}^n \lambda(\gamma_0\xi_{i}\cap\gamma_0\xi_{i+1})  v(\gamma_0\xi_{i},\gamma_0\xi_{i+1}), $$
i.e.~$\gamma_0 X(\xi)$ is the realization of the path from $\gamma_0\xi_b$ to $\gamma_0\xi$.
Hence $$\tau_{\gamma_0\mu_0}=X(\gamma_0\mu_0\xi_b)=X(\gamma_0\xi_b)
+\gamma_0X(\mu_0\xi_b)=\tau_{\gamma_0}+\gamma_0\tau_{\mu_0},$$ and the cocycle condition \eqref{eq:cocycle} is satisfied.
Finally
$$\gamma X(\xi)=\gamma_0 X(\xi)+\tau_{\gamma_0}=\gamma_0 X(\xi)+X(\gamma_0\xi_b)$$
is the sum of the realization of the path from $\xi_b$ to $\gamma_0\xi_b$ followed by the 
path from $\gamma_0\xi_b$ to $\gamma_0\xi$, i.e.~it is the realization of the path from 
$\xi_b$ to $\gamma_0\xi$, hence a vertex of $P$. The set of vertices of $P$ is $\Gamma_{\tau}$ invariant, and so
is $P$.

\begin{figure}[ht]
\begin{center}
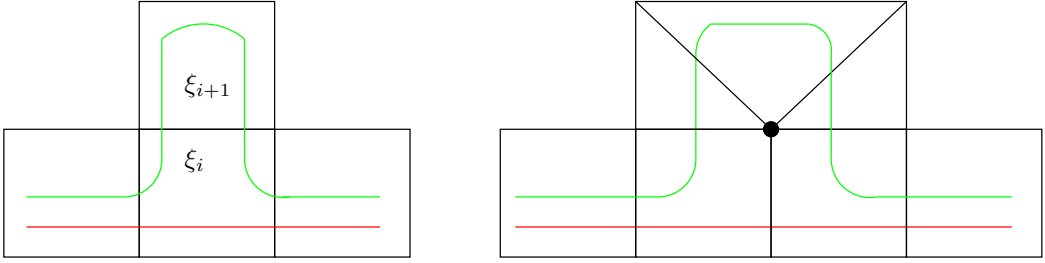
\end{center}
\caption{The two kinds of operations to go from a path of cell to another (proof of  Theorem~\ref{thm:pol}).}\label{fig: not pol 2}
\end{figure}
\end{proof}

\begin{figure}[ht]
\begin{center}
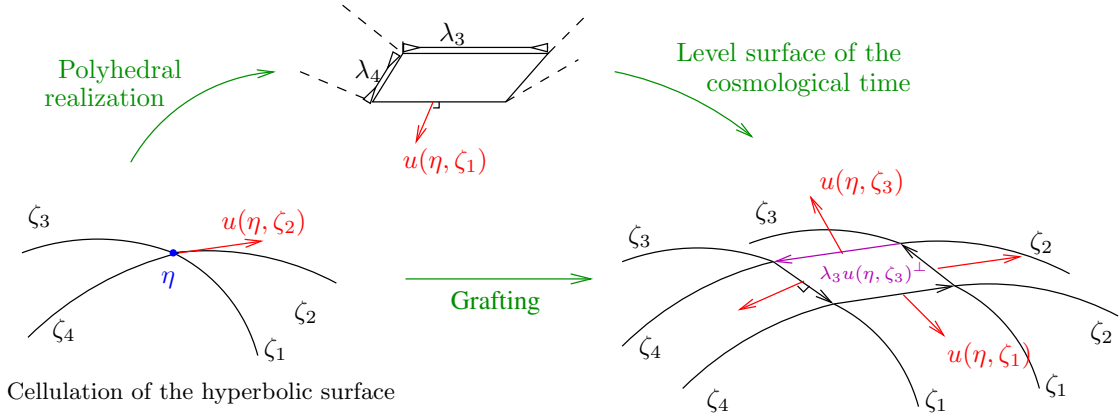
\end{center}
\caption{To Remark~\ref{rem:grafting}. Grafting and intrinsic meaning of condition \eqref{eq:sum poly}.}\label{fig: not pol 1}
\end{figure}

\begin{remark}[\textbf{A classical construction}]\label{rem:mess pol}{\rm 
The analog of Theorem~\ref{thm:pol} in the compact Euclidean case was
solved in \cite{Sch77}. We almost repeated this proof in the first part of the proof of Theorem~\ref{thm:pol} above, up to obvious changes.
Note that the argument is classical and appears in some places in polyhedral geometry, without mention
to the Christoffel problem, see \cite{FI3}.
Here polyhedral hedgehogs appear naturally under the name  \emph{virtual polytopes}, as realizations
of signed polyhedral measure of order one.

The striking fact is that the construction in the proof of Theorem~\ref{thm:pol} also appears in the following.
Inspiring on the $d=2$ construction of
G.~Mess \cite{Mes07,Mes07+},  F.~Bonsante shows in \cite{Bon05} how to construct
an F-regular domain from a measured geodesic stratification, see Remark~\ref{subsub: strat}
(in this setting, in $d=2$, the analog of condition \eqref{eq:sum poly} is void, but
it holds for $d>2$). The second part of the proof of Theorem~\ref{thm:pol} comes from those references.
Actually the basement of the construction is contained in the $d=1$ case (Remark~\ref{rem pol d=1}). 
}
\end{remark}

\begin{remark}[\textbf{Graftings}]\label{rem:grafting}{\rm 
 Let $\H^2/\Gamma$ be a compact hyperbolic surface, and let $\sigma$ be a simple closed geodesic on it.
Assign a positive weight $a$ to $\sigma$. It lifts on $\H^2$ to an infinite number of disjoint geodesics, with the same
weight $a$. From the construction mentioned above  \cite{Mes07,Bon05}, one can construct a domain $\Omega_{\tau}$.
Let $\tilde{S}_1$ be the level surface  for the cosmological time of $\Omega_{\tau}$. We get a compact surface
$\tilde{S}_1/\Gamma_{\tau}$, and this way to go from $\H^2/\Gamma$ to $\tilde{S}_1/\Gamma_{\tau}$ is a geometric realization of 
a \emph{grafting} of 
$\H^2/\Gamma$ along $\sigma$. Grafting are more generally defined along a measured geodesic lamination on a 
hyperbolic surface. The same procedure applied to a $\tau$-F-convex polyhedron
is the geometric realization of a grafting, not along disjoint geodesic but along a cellulation of the hyperbolic surface. See Figure~\ref{fig: not pol 1}.
}
\end{remark}

\begin{remark}[\textbf{Fuchsian condition}]{\rm 
The polyhedral case is absent from Theorem~\ref{thm: base general}
because  the polyhedral surface given by Theorem~\ref{thm:pol} should not be Fuchsian in
general. Are there conditions on the measure to be the first area measure of a convex Fuchsian polyhedron?
Can these conditions be stated in term of grafting in $d=2$?
}
\end{remark}

\subsection{Convexity of solutions}\label{sec:convexity}

In sections from \ref{sub:regsol} to \ref{sub:fuchssol} we have described how to obtain a general 
analytic solution to equation \eqref{main_eq}. Actually, by a geometrical point of view we are mainly interested in special solutions which are restriction to 
$\hh^d$ of convex functions on $\mathcal{F}$. 
Hence, in this section we discuss some conditions which ensure the convexity of the solution 
$h_\mu$ given in \eqref{def_h}.

A first general necessary and sufficient convexity condition for classical convex body was 
given by Firey in Theorem 2 of \cite{Fir68}. There the convexity was showed to be equivalent 
to the positivity of a particular quadratic form. As already observed in \cite{LS06}, 
Firey's approach seems unlikely generalizable to F-convex set, since it is based on applications 
of the Stokes' theorem on the compact sphere. Nevertheless, a similar condition can be given also 
in our case. We suppose here that the solution $h_\mu$ given in \eqref{def_h} is continuous 
(this is without loss of generality, since support functions of convex sets are necessarily 
continuous; see also Proposition \ref{pr_rigidity}). 
By Section \ref{sub: extended}, 
we know that $h_\mu$ is the restricted support function of a convex set if and only if 
its extended support function $H_\mu(\eta)=\|\eta\|_-h_\mu(\eta/\|\eta\|_-)$ is convex, which is 
in turn equivalent to $H_\mu$ being subadditive, i.e.
\[
H_\mu(\eta+\nu)\leq H_\mu(\eta)+H_\mu(\nu).
\] 
We note that $H_\mu$ can be written in the form
\[
H_\mu(\eta) = \int_{\hh^d} \|\eta\|_-G\left(\frac\eta{\|\eta\|_-},y\right)\mathrm{d}\mu(y) = \int_{\hh^d} \Gamma(\eta,y)\mathrm{d}\mu(y),
\]
where
\begin{equation*}\begin{aligned}
\Gamma(\eta,y) &= \|\eta\|_- k\left(d_{\hh^d}\left(\frac\eta{\|\eta\|_-},y\right)\right) = \|\eta\|_- 
k\left(\operatorname{acosh} \left(-\|\eta\|_-^{-1}\left\langle \eta,y\right\rangle_-\right)\right)\\
&= -\frac{\left\langle \eta,y\right\rangle_-}{v_{d-1}}
\int_{+\infty}^{\operatorname{acosh}\left(-\|\eta\|_-^{-1}\left\langle \eta,y\right\rangle_-\right)}
\frac{\mathrm{d}q}{\sinh^{d-1}q\cosh^2q}
\end{aligned}\end{equation*}
is defined for all $\eta\in \mathcal F$ and $y\in\hh^d\subset\mathcal F$. Hence we get the following
\begin{proposition}\label{prop: conv cond}
Let $\mu\in\mathcal R^+(\hh^d)$. Then  $h_\mu$ defined formally
as in \eqref{def_h} is the restricted support function of a $F$-convex set if and only if
$$ \left| \int_{\mathbb H^d}G(x,y)\mathrm{d}\mu(y)\right|<+\infty,\quad\forall x\in\H^d,$$
and
\begin{equation}\label{convexity_Firey}
\int_{\hh^d} \Lambda(\eta,\nu,y)\mathrm{d}\mu(y) \geq 0,
\end{equation}
for all $\eta,\nu\in \mathcal F$, where
\[
\Lambda(\eta,\nu,y)=\Gamma(\eta,y)+\Gamma(\nu,y)-\Gamma(\eta+\nu,y).
\]
\end{proposition}
\bigskip

In case $h_\mu\in C^2$, $\mu=\phi \mathrm{d}\H^d$ for some continuous function $\phi$, 
and the expression of $h=h_\phi$ is given by \eqref{h}. Thanks to Proposition~\ref{lem: base}, 
we know that $h_\phi$ is the restricted support function of a $F$-convex if and only if 
$\nabla^2h_\phi-h_\phi g\geq0$. In \cite{LS06}, the authors computed explicitly this expression. For completeness we report here, with minor changes, their computations.\\

Let $\nabla_1^2 G$ be the Hessian of $G:\hh^d\times\hh^d\to\rr$ with respect to the first component. Then
\begin{equation}\label{hess h}
	\left(\nabla^2 h|_x- h(x) g|_x\right)(X,X) = \int_{\hh^d\setminus\{x\}}\left[\nabla^2_1 G|_{(x,y)}(X,X) - |X|^2G(x,y)\right]\phi(y)\mathrm{d}\H^d(y),
\end{equation}
 for all $X\in T_x\hh^d$, with $|X|^2=g(X,X)$. 
Since $G(x,y)=k(\rho_y(x))$, we have that
\begin{equation*}
\nabla^2_1G|_{(x,y)}=\ddot k(\rho_y(x))\mathrm{d}\rho_y\otimes \mathrm{d}\rho_y|_{x}+\dot k(\rho_y(x))\nabla^2 \rho_y|_{x}.	
\end{equation*}

Computing explicitly $\dot k$ and $\ddot k$ and using \eqref{eq: hessien rho}
\begin{equation*}\begin{aligned}
\nabla^2_1G|_{(x,y)}=\left(k(\rho_y(x))+\frac{1}{v_{d-1}\sinh^d(\rho_y(x))}\right) g-\frac{d}{v_{d-1}\sinh^d(\rho_y(x))}\mathrm{d}
\rho_y\otimes \mathrm{d}\rho_y|_{x}.
\end{aligned}\end{equation*}
Accordingly, \eqref{hess h} yields
\begin{equation*}\begin{aligned}
&\left(\nabla^2 h|_x- h(x) g|_x\right)(X,X)\\
&=|X|^2\int_0^{\infty}\frac{1}{v_{d-1}\sinh^d(\rho_y(x))}\int_{\partial B_\rho(x)}\phi(y)\mathrm{d}A_\rho(y)\mathrm{d}\rho \\
&-\int_0^{\infty}\frac{d}{v_{d-1}\sinh^d(\rho_y(x))}\int_{\partial B_\rho(x)} g|_x(\nabla\rho_y,X)^2\phi(y)\mathrm{d}A_{\rho}(y)\mathrm{d}\rho\\
&=\int_0^{\infty}\frac{1}{v_{d-1}\sinh^d(\rho_y(x))}\int_{\partial B_\rho(x)}\left[|X|^2-d g_{\hh^d}|_x(\nabla\rho_y,X)^2
\right]\phi(y)\mathrm{d}A_\rho(y)\mathrm{d}\rho.
\end{aligned}\end{equation*}

\begin{proposition}\label{convexity_Lopes}
Let $\phi\in C^2(\hh^d)$. The function $h_\phi$, defined as in \eqref{h}, is the restricted support function of a $F$-convex set if and only if
\begin{equation*}
0\leq\int_0^{\infty}\frac{1}{v_{d-1}\sinh^d(\rho_y(x))}\int_{\partial B_\rho(x)}\left[|X|^2-d g_{\hh^d}|_x(\nabla\rho_y,X)^2
\right]\phi(y)\mathrm{d}A_\rho(y)\mathrm{d}\rho,
\end{equation*}
for all $x\in \hh^d$ and all $X\in T_x\hh^d$.
\end{proposition}

\begin{remark}{\rm
The last expression corresponds to the quadratic form $Q_{u'}(u'')$ computed in \cite{LS06}, 
where $u'=x$ and $u''=X$. This is easily seen using the explicit form of $k$ and the relations $\cosh \rho_y(x) = -\left\langle x,y\right\rangle$ and 
$\left\langle \nabla\rho_y(x),X\right\rangle\sinh\rho_y(x)=-\left\langle X,y\right\rangle$ obtained at page 93 in \cite{LS06}.
 Here $y\in\hh^d$ is identified with $y\in T_x\mathcal F$.
}\end{remark}

\begin{remark}[\textbf{Sufficient conditions}]\label{rem:sufficient conditions}{\rm

The convexity conditions \eqref{convexity_Firey} and 
the one of Proposition~\ref{convexity_Lopes} are sharp, but pretty involved and hard to check. 
In the case of compact convex bodies in the Euclidean space, a more direct approach was proposed by Guan and Ma \cite{GM03}, following Pogorelov, but it does not seems suitable to be adapted to our setting. 
In fact in the classical setting one has that the restricted support function $h_K$ of a regular convex body $K$ satisfies the convexity condition 
${}^{\mathbb S^d}\nabla^2 h_K + h_K g_{\mathbb S^d} \geq 0$ as a quadratic form on ${\mathbb S^d}$. 
Using the fact that the Hessian $\operatorname{Hess} (H_K)$ of the total support function $H_K(x):=|x|h_K(x/|x|)$
is $(-1)$-homogeneous and, in $\R^{d+1}$,
\[
\begin{aligned}
\hess(\hess H_K (e_i,e_i))(e_j,e_j)=\hess(\hess H_K (e_j,e_j))(e_i,e_i),
\end{aligned}
\]
one obtains the symmetry relation
\[
\begin{aligned}
&{}^{\mathbb S^d}\nabla^2\left({}^{\mathbb S^d}\nabla^2 h_K (e_i,e_i) + h_K\right)(e_j,e_j)+{}^{\mathbb S^d}\nabla^2 h_K (e_j,e_j)
\\&={}^{\mathbb S^d}\nabla^2\left({}^{\mathbb S^d}\nabla^2 h_K (e_j,e_j) + h_K\right)(e_i,e_i)+{}^{\mathbb S^d}\nabla^2 h_K (e_i,e_i),
\end{aligned}
\]
for all $i=1,\dots,d$, where $\{e_i\}_{i=1}^d$ is a local orthonormal frame in a neighborhood 
of any point $x\in{\mathbb S^d}$. Choosing the point $x$ and the direction $e_1$ such that
${}^{\mathbb S^d}\nabla^2 h_K (e_1,e_1)|_x + h_K(x)$ is a minimum of the curvature radius of $K$, an application of the maximum principle gives that 
\[
 {}^{\mathbb S^d}\nabla^2 h_K (e_1,e_1)|_x + h_K(x)\geq\phi-{}^{\mathbb S^d}\nabla^2 \phi (e_1,e_1),
\]
where $\phi$ is the mean radius of curvature of $K$. This proves that $K$ is convex provided 
\begin{equation}\label{cond_GuanMa}
  \phi(x)-{}^{\mathbb S^d}\nabla^2 \phi (e_i,e_i)|_x \geq 0, 
\end{equation}
for all $x$ and $i$.

Because of the different sign in the decomposition of the Euclidean Hessian in our setting \eqref{restr hessien}, we get instead that 
\[
 \nabla^2 h_K (e_1,e_1)|_x - h_K(x)\leq\nabla^2 \phi (e_1,e_1)+\phi,
\]
from which it seems impossible to get any useful conclusion.\\
It has to be noted that in \cite{GM03} a further sufficient condition for the existence of a convex solution is given for the classical compact problem. In particular it is there asked for $\phi^{-1}$ to be a solution of $\nabla^2_{\mathbb S^d}\phi^{-1}+\phi^{-1}g_{\mathbb S^d}\geq 0$ in the sense of quadratic form (actually the more general Christoffel-Minkowski problem is treated). Once again, the techniques used in \cite{GM03} seem require the compactness of the underlying space $\mathbb S^d$, so that a generalization of their proof to our setting seems definitely non-trivial. Nevertheless is natural to wonder whether there exist conditions on $\nabla^2_{\H^d}\phi^{-1}-\phi^{-1}g_{\H^d}$ which imply the existence of an F-convex solution to the Christoffel problem.
}
\end{remark}

\begin{remark}[\textbf{Curvatures close to a constant function}]{\rm
Finally, we note that condition \eqref{cond_GuanMa} is verified if $\phi$ is $C^2$ close to a constant function. 
In the same order of idea, suppose that $0<\bar\phi:\H^d/\Gamma\to \mathbb R$ is $C^\alpha$ close enough to 
a positive constant function $\bar\phi_\ast>0$. The unique $\Gamma$ invariant solution $\bar h_\ast$ to 
\[
 \frac{1}{d}\Delta \bar h_\ast - \bar h_\ast = \bar\phi_\ast
\]
on $\H^d/\Gamma$ is the constant function $\bar h_\ast= -\bar\phi_\ast$. 
Consider the unique $\Gamma$ invariant solution $\bar h$ to $\frac{1}{d}\Delta\bar h - \bar h = \bar\phi$, 
which, by Theorem \ref{th_smooth} and with notation introduced therein, is given as
\begin{equation*}
\bar h(\bar  x)=\int_{\hh^d}G(x,y)\phi(y)\mathrm{d}{\hh^d}(y).
\end{equation*}
Since $\bar\phi$ is $C^\alpha$ close to $\bar\phi_\ast$ and $G(x,\cdot)\in L^1(\H^d)$, also $\bar h$ is $C^0$ close to $\bar h_\ast$. 
Then, by Schauder estimates (see for instance Section 3.6.3 in \cite{Au98}), $\bar h$ is $C^{2,\alpha}$ close to $h_\ast$, and it is then the restriction to $\H^d$ of a convex function on $\mathcal F$. This proves the following

\begin{proposition}
 Let $0<\bar\phi:\H^d/\Gamma\to \mathbb R$. Fix constants $0<\alpha\leq 1$ and $\bar\phi_\ast>0$. 
There exists a constant $c=c(\alpha,\bar\phi)$ such that if $\left\|\bar\phi-\bar\phi_\ast\right\|_{C^\alpha}<c$, 
then $\bar\phi$ is the restricted support function of a $\Gamma$ invariant F-convex set.
\end{proposition}
}\end{remark}

\subsection{Uniqueness}

In this section with find conditions under which F-convex sets are uniquely determined by their first area measure. This is obviously not true in general, considering F-convex sets differing by
a translation, whose restricted support functions are given in Example~\ref{ex:restricted lin}. Below is a more elaborated example.

\begin{example}[\textbf{Fuchsian and quasi-Fuchsian F-convex sets with same mean radius of curvature}]{\rm
A  nontrivial example can be constructed as follows. Let $\tau$ be a cocycle which is not a coboundary. Let $K$ be a 
$C^2_+$ $\tau$-F-convex set, with mean radius of curvature $\phi$. $\phi$ is $\Gamma$ invariant 
and we know by Theorem~\ref{th_smooth} that there exists a $\Gamma$ invariant solution $h_0$. 
We do not know if $h_0$ is convex, but for any $t>0$, $K+tK(\H)$ is a $C^2_+$ $\tau$-F-convex set, 
with mean radius of curvature $\phi+t$, and the corresponding Fuchsian solution is
$h_0-t$. If $t$ is sufficiently large, $\nabla^2 (h_0-t)-(h_0- t)g>0$, and 
$h_0-t$ is the support function of a $\Gamma$ invariant F-convex set with same mean radius of curvature than
 $K+tK(\H)$.
 }\end{example}

\subsubsection{An elementary case}\label{unique-convex}

So far we have seen some uniqueness results for analytic solutions to equation \eqref{eq_dist}. As a matter of fact, 
we are interested in a smaller class of solutions which are restricted support functions of some F-convex set. As one expects, convexity gives further information on the uniqueness of the solution. 
A special situation occurs when the first area measure $\mu$ of some F-convex set 
$K$ is zero in some open domain $\Omega\subset\hh^d$. 
In this case we have that the restricted support function $h_K$ satisfies the homogeneous equation $\frac{1}{d}\Delta h_K - h_K =0$ 
in the sense of distributions on $\Omega$. By elliptic regularity we have that $h_K|_\Omega\in C^\infty(\Omega)$. 
In particular, it makes sense to consider the Hessian $\nabla^2h_K$ of $h_K$. Hence we have that, at each point $x\in\Omega$, 
the quadratic form $\nabla^2h_K - h_Kg$ is trace-null, and furthermore all its eigenvalues are nonnegative by convexity condition. 
This yields that $\nabla^2h_K - h_Kg\equiv 0$ in $\Omega$, which in turn gives that the extended support function 
$H_K(\eta)$ has null 
Hessian on $\{\eta\in\mathcal F : \eta/\|\eta\|_-\in \Omega\}$. 
Hence, $h_K|_\Omega$ is the restriction to $\hh^d$ of a linear function on $\mathbb{R}^{d+1}$.

The remarks above gives an elementary condition for uniqueness:
\begin{lemma}
 Let $H_1$ and $H_2$ be the extended support functions of two
F-convex sets with the same first area measure. If $H_1-H_2$ is convex, then they differ by the restriction of a linear form to $\H^d$.
\end{lemma}

This also gives the following characterization.
\begin{lemma}
 An F-convex set whose first area measure is a polyhedral measure of order one is an F-convex polyhedron.
\end{lemma}

In Section~\ref{sec:qf sol}, we will show many hypersurfaces with zero mean radius of curvature,
but they will not be explicit.

\begin{example}[\textbf{A surface with zero mean radius of curvature}]\label{ex : surface courbure 0}{\rm
From Example~\ref{rem:elementary dist}, we got a function $h$ on an open set $\mathcal{O}$ of $\H^2$ such that
its normal representation has zero mean radius of curvature. 
Up to a constant, the $1$-extension $H$ of $h$ has the form
$$H(x)= \langle x,v\rangle_-\arctan \left(\frac{\|x\|_-}{\langle x,v\rangle_-}\right)+\langle x,\frac{x}{\|x\|_-}\rangle_- $$
(here one can another time check that the wave operator of $H$ restricted to $\H^2$ is zero)
and one can compute its Lorentzian gradient restricted to $\H^2$. 
Taking for $v$ the vector with coordinates $(1,0,0)$, and using 
 the parametrization of $\mathcal{O}$ with coordinates
 $\left( \begin{array}{c}\sinh(t)\cos(\theta)\\ \sinh(t)\sin(\theta)\\ \cosh(t)\end{array} \right) $, for $t>0, -\pi/2<\theta<\pi/2$,  we get the following normal representation, drawn in Figure~\ref{fig: courb 0},
$$\chi(t, \theta)=\left( \begin{array}{c}
\arctan\left(\frac{1}{\sinh(t)\cos(\theta)}\right) \\
\frac{\sinh(t)\sin(\theta)}{1+\sinh(t)^2\cos(\theta)^2} \\
\frac{\cosh(t)}{\sqrt{1+\sinh(t)^2\cos(\theta)^2}}
\end{array} \right).$$
Note that at the points where the radii $r_i$ of curvature are not zero, multiplying by $r_1r_2$,  $r_1+r_2=0$ implies
$1/r_1+1/r_2=0$, and $1/r_1$ are the principal curvatures of the surface, hence the surface has mean curvature zero.

 }\end{example}

 \begin{figure}[htp]
   \centering
  \subfloat[The zero mean radius of curvature surface.]{\label{a}\includegraphics[scale=0.6]{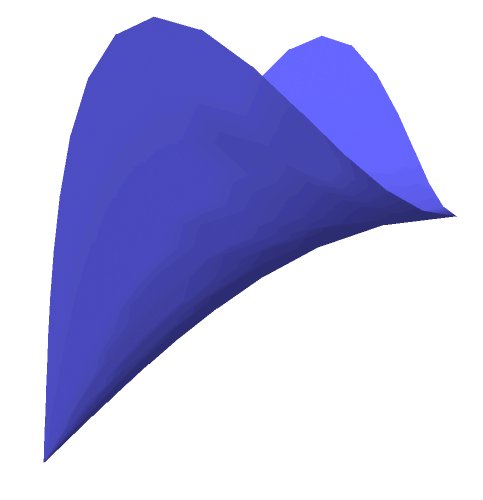} }   \\     
   \subfloat[The curve is the intersection of the surface with the $\{x_2=0\}$ plane. 
It can also be obtained from the function $h:\R^+ \rightarrow \R$, $h(t)=\sinh(t)\arctan\left(\frac{1}{\sinh(t)}\right)-1$ 
and formula \eqref{eq:eq courbe}.
Its radius of curvature is $\frac{-1}{\cosh(t)^2}$.]{\label{b}\includegraphics[scale=0.4]{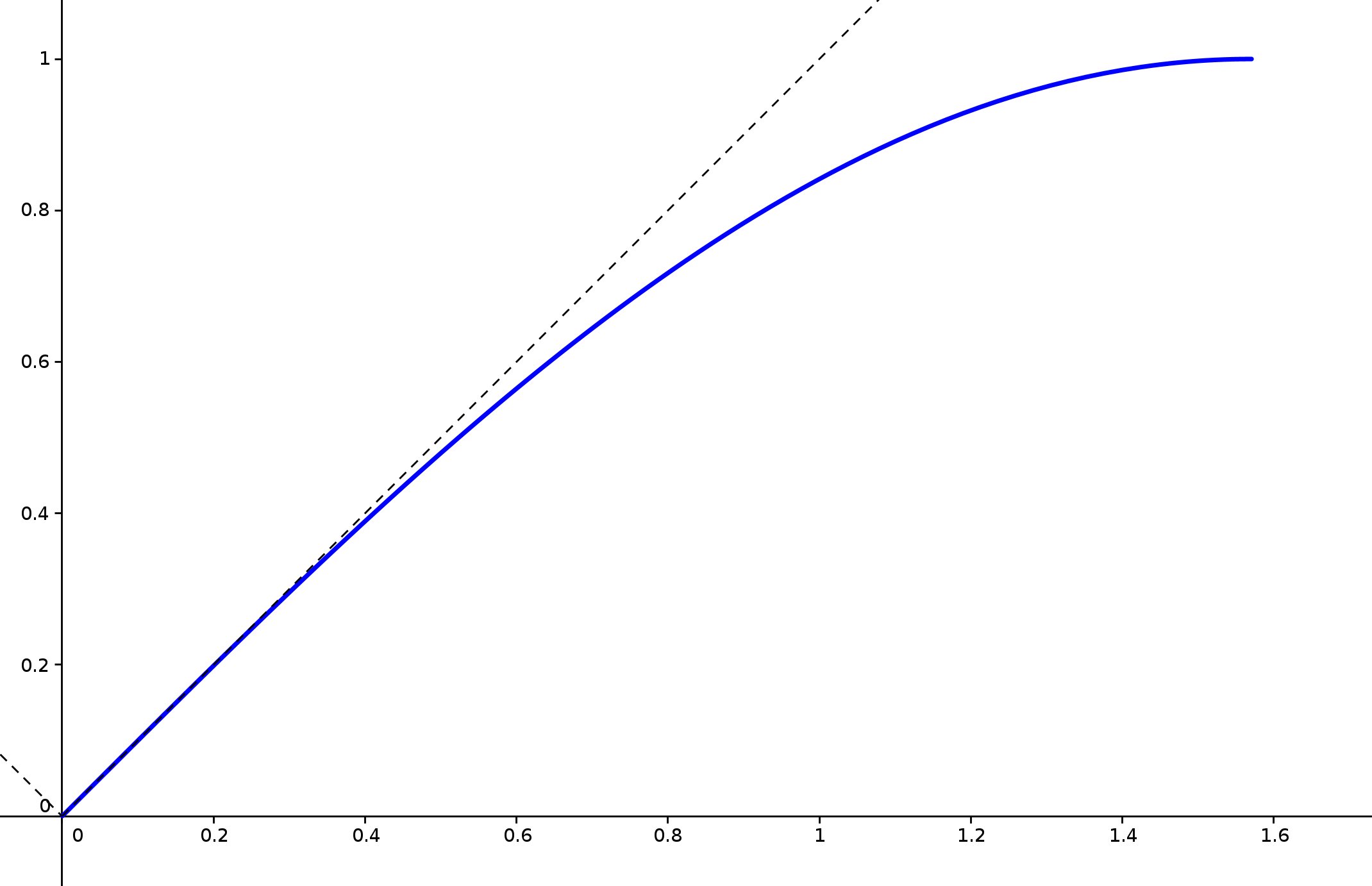}  }
   \caption{To Example~\ref{ex : surface courbure 0}.}
   \label{fig: courb 0}
 \end{figure}

\subsubsection{Sovertkov condition for uniqueness}

In \cite{So81}, the author proved the uniqueness among smooth solutions which do not grow too much. 
An easy observation gives that Sovertkov's result holds as well for distribution solutions.

\begin{theorem}\label{th_sov_uniq}
Let $\mu$ be a positive radon measure on $\hh^d$ and let $\zeta:\partial \hh^d\to \mathbb{R}$ be a function defined on the hyperbolic boundary at infinity. There is at most one continuous distribution solution $h$ to the equation \eqref{eq_dist} satisfying
\begin{equation}\label{subconical}
\forall \theta,\quad \underset{\rho\rightarrow +\infty}{\mathrm{lim}}\frac{h(\rho,\mathbf{\theta})}{\cosh (\rho)}=\zeta(\theta). 
\end{equation}
\end{theorem}
By  Lemma~\ref{lem:h infini}, the result above has a clear geometric meaning: two F-convex sets with the same
first area measure are equal if for any null direction $\ell$ they have the same support
 plane at infinity directed by $\ell$. In particular, if $\zeta$ is continuous, the two convex sets must be contained
in the future cone of a point.

\begin{proof}
Let $h_1,h_2\in \mathcal D'(\hh^n)$ be two continuous functions satisfying \eqref{subconical} and
\[
\Delta h_1 -dh_1 = \mu = \Delta h_2-dh_2
\]
in $\mathcal D'(\hh^n)$. Then, $h_3=h_1-h_2$ satisfies 
\begin{equation*}
\forall \theta,\quad \underset{\rho\rightarrow +\infty}{\mathrm{lim}}\frac{h_3(\rho,\mathbf{\theta})}{\cosh (\rho)}= 0,
\end{equation*}
by the linearity of the equation 
\begin{equation}\label{homogen}
\frac{1}{d}\Delta h_3 -h_3 = 0,
\end{equation}
and by elliptic regularity $h_3\in C^\infty (\hh^n)$, \cite{Au98}. Hence we can proceed as in \cite{So81} to prove that $h_3=0$. Namely, let $\epsilon>0$ and define the smooth functions
\[
h^{(\epsilon)}_{\pm}(\rho,\mathbf{\theta}):=\epsilon\left(\cosh(\rho)+1\right)\pm h_3(\rho,\mathbf{\theta}).
\]
Both $h^{(\epsilon)}_\pm$ satisfy 
\begin{equation*}
\underset{\rho\rightarrow +\infty}{\mathrm{lim}}h^{(\epsilon)}_\pm(\rho,\mathbf{\theta})
=\underset{\rho\rightarrow +\infty}{\mathrm{lim}}\left(\cosh(\rho)+1\right)\left(\epsilon\pm \frac{h_3(\rho,\mathbf{\theta})}{\cosh(\rho)+1}\right)>0,
\end{equation*}
for all $\mathbf{\theta}$, and 
\begin{equation*}
\frac{1}{d}\Delta h^{(\epsilon)}_\pm -h^{(\epsilon)}_\pm = -\epsilon d < 0.
\end{equation*}
By the maximum principle we thus get that $h^{(\epsilon)}_\pm$ are strictly positive for all $\epsilon$, that is
\[
|h_3(\rho,\mathbf{\theta})|<\epsilon\left(\cosh(\rho)+1\right)
\]
for all $\epsilon>0$ and $(\rho,\mathbf{\theta})\in \hh^d$. This proves the claim.
\end{proof}


\subsubsection{Non-uniqueness}\label{sec_non-uniq}

Reasoning as in the proof of Theorem \ref{th_sov_uniq}, it is possible to get a characterization of non-unique solutions. In fact, let $h_1$ and $h_2$ be two distributions solutions to the equation \eqref{eq_dist} for some positive Radon measure  $\mu$ on $\hh^d$. Then $h=h_1-h_2$ satisfies the homogeneous equation \eqref{homogen} and is hence smooth by elliptic regularity. This elementary observation easily implies the following

\begin{proposition}\label{pr_rigidity}
Let $\mu\in\mathcal R(\hh^d)^+$ and let $h_\mu$ be the distribution solution to equation \eqref{eq_dist} defined in \eqref{def_h}. If $h_\mu\in \mathcal D'(\hh^d)\setminus C^{0}(\hh^d)$, then there exists no F-convex set $K$ with $\mu$ as first area measure.
\end{proposition}

\begin{proof}
By contradiction, suppose such a convex $K$ exists. Then its restricted support function $h_K$ is a continuous solution to \eqref{eq_dist}. But $h_\mu-h_K\in C^\infty(\hh^d)$ by elliptic regularity, and this gives us a contradiction.
\end{proof}

\begin{example}{\rm Let $\mu=\delta_y$ be the Dirac distribution at the point $y\in\hh^d$. Then the solution to \eqref{main_eq} proposed in \eqref{def_h} is 
\[
h_\delta(x)= G(x,y) \in D'(\hh^d)\setminus C^{0}(\hh^d).
\]
Hence, by Proposition~\ref{pr_rigidity}, there is no F-convex set with first area measure $\delta_y$. 
On the other hand, this result is not surprising, since by Section \ref{unique-convex} we know that a continuous solution $h$ to 
$\frac{1}{d}\Delta h - h = \delta_y$ that is restriction to $\H^d$ of a convex function, has to be the restriction of a linear function on $\hh^d\setminus\{y\}$, hence on all of $\hh^d$ by continuity.
}\end{example}


 \subsection{Proof of Theorem~\ref{thm: base general}}

The uniqueness is a consequence of 
Theorem~\ref{th_sov_uniq} 
together with Lemma~\ref{lem: petit o} (it will also follows from Corollary~\ref{cor_unicity-fuchsian}).

The first part of Theorem~\ref{thm: base general} follows from Theorem~\ref{sol_fuchs_distr}, the second from
Proposition~\ref{prop: conv cond}
and the third from Theorem~\ref{th_smooth}.


\subsection{The $d=1$ case}
We specify here the analytical results of the previous section to the one dimensional setting, where an almost complete picture can be given.
Actually, the first area measure is also the last area measure, so in $d=1$ 
there is a unique Christoffel--Minkowski problem.
In fact in this case we have $\hh^1=\mathbb{R}^1$ (see Subsection~\ref{sub:1d}), 
and the first area measure is a positive Radon measure $\mu$ on $\mathbb R$. 
Accordingly, equation \eqref{main_eq} reads
\begin{equation}\label{eq_1D}
h''(t)-h(t) = \mu,\quad\textrm{in }\mathcal D'(\rr)
\end{equation}
in the sense of distributions, that is 
\[
\int_ {-\infty}^\infty h(s)(f''(s)-f(s))\mathrm{d}s = \int_ {-\infty}^\infty f(s)\mathrm{d}\mu(s),\quad\forall f\in C^{\infty}_c(\rr). 
\]
Assume that $\mu\in\mathcal R^+(\hh^1)$, that is 
\begin{equation}\label{growth_1D}
\int_{-\infty}^\infty e^{-|t|}\mathrm{d}\mu(t)<\infty.
\end{equation}
Reasoning as in the previous sections, we get that a particular solution to \eqref{eq_1D} takes the form
\[
h_\mu(t)=-\int_{-\infty}^\infty \frac{e^{-|s-t|}}2\mathrm{d}\mu(s),
\]
where the distribution $h_\mu\in\mathcal D'(\rr)$ is defined by
\[
(h_\mu,f):=(\mu,h_f)=-\int_{-\infty}^\infty\left(\int_{-\infty}^\infty \frac{e^{-|s-t|}}2f(s)\mathrm{d}s\right) \mathrm{d}\mu(t),\quad\forall f\in C^{\infty}_c(\rr),
\]
and is well-defined because of \eqref{growth_1D}. In fact an integration by parts yields
\begin{equation*}\begin{aligned}
({h_\mu}''-h_\mu,f)
&=(h_\mu,f''-f)
=-\int_{-\infty}^\infty\left(\int_{-\infty}^\infty \frac{e^{-|s-t|}}2(f''(s)-f(s))\mathrm{d}s\right) \mathrm{d}\mu(t)
=\int_{-\infty}^\infty f(t)\mathrm{d}\mu(t).
\end{aligned}\end{equation*}
We note that, thanks to condition \eqref{growth_1D}, even if the function $f:=-e^{-|s|}/2$ is not compactly supported, the convolution $h_\mu = f\ast\mu$ inherits the continuity property of $f$.\\
Considering also solutions to the homogeneous equation $h''=h$, we get that for $\mu\in\mathcal R^+(\hh^1)$ all solutions to equations \eqref{eq_1D} are continuous and can be written as
\begin{equation}\label{1D_general_sol}
h_\mu(t)=-\int_{-\infty}^\infty \frac{e^{-|s-t|}}2\mathrm{d}\mu(s)+A\cosh(t)+B\sinh(t),\quad A,B\in\rr.
\end{equation}
When $\mu=\phi(t)\mathrm{d}t$ for some $\phi\in C^0(\rr)$, then assumption \eqref{growth_1D} can be skipped. In fact the general solution to equation 
\begin{equation}\label{eq_1D_phi}
h''(t)-h(t) = \phi(t)
\end{equation}
can be also written in the form
\begin{equation}\label{1D_general_sol_2}
h_\phi = \int_1^t\sinh(t-s)\phi(s)\mathrm{d}s+C\cosh(t)+D\sinh(t),\quad C,D\in\rr,
\end{equation}
which makes sense for any continuous function $\phi$ without growth assumption. We note that, when
\[
\int_{-\infty}^\infty e^{-|t|}\phi(t)\mathrm{d}t<\infty,
\]
the expression in \eqref{1D_general_sol_2} and in \eqref{1D_general_sol} are the same up to set
\begin{equation*}\begin{aligned}
&A= C + \frac12\int_1^\infty e^{-s}\phi(s)\mathrm{d}s + \frac12\int_{-\infty}^1 e^{s}\phi(s)\mathrm{d}s,\\
&B= D + \frac12\int_1^\infty e^{-s}\phi(s)\mathrm{d}s - \frac12\int_{-\infty}^1 e^{s}\phi(s)\mathrm{d}s.
\end{aligned}\end{equation*}\bigskip

\noindent Since the problem is one dimensional, equation \eqref{eq_1D_phi} can be interpreted also as
\[
\nabla^2h-gh=\phi\geq0
\]
hence all the solutions given in \eqref{1D_general_sol_2} are automatically restrictions
 to $\hh^1$ of convex functions on $\mathcal F$.
 
When $\mu\in\mathcal R(\hh^1)$ is a positive measure, 
one expects to get the same conclusion for solutions of \eqref{eq_1D}, since, 
roughly speaking, $\nabla^2h-gh=\mu>0$ in the sense of distribution. 
To prove this, thanks to Lemma~\ref{lem:1d}, it is enough to show that 
\[
h_\mu(t+\alpha)+h_\mu(t-\alpha)\geq 2\cosh(\alpha)h_\mu(t)
\]
for all $t,\alpha\in\rr$. We let $t$ be fixed, and since this latter is an even condition, 
we can assume without loss of generality that $\alpha>0$.
Then, an explicit computation gives that
\begin{equation*}\begin{aligned}
2\cosh(\alpha)h_\mu(t)-h_\mu(t+\alpha)-h_\mu(t-\alpha)=\int_{-\infty}^\infty\left[-\cosh(\alpha)e^{-|t-s|}+\frac{e^{-|t+\alpha-s|}}2+\frac{e^{-|t-\alpha-s|}}2\right]\mathrm{d}\mu(s)\leq 0,
\end{aligned}\end{equation*}
since
\begin{equation*}\begin{aligned}
\left[-\cosh(\alpha)e^{-|t-s|}+\frac{e^{-|t+\alpha-s|}}2+\frac{e^{-|t-\alpha-s|}}2\right]=
\begin{cases} 0, &\textrm{if }|t-s|\geq\alpha,\\
              \sinh(|t-s|-\alpha)\leq0, &\textrm{if }|t-s|<\alpha. 
              \end{cases} 
\end{aligned}\end{equation*}

\begin{example}{\rm
To end this section, we remark that here we get also an explicit expression 
for the Elementary example of Example~\ref{ex: elemntary1} in the $d=1$ case. 
This is no more true in higher dimension, as we  discussed in Example~\ref{rem:elementary pol ex}. 
In fact, in one dimension we have to consider a measure concentrated in a point, that is for instance $\mu=\delta_0$, the Dirac mass at the origin. Hence a special solution $h_{\delta_0}$ given by \eqref{1D_general_sol} is 
\[
h_{\delta_0}(t)=\frac{e^{-|t|}}2,
\]
which is the restriction to $\H^1$ of the $1$-homogeneous piecewise linear function $H$ on $\mathbb R^2$ defined as
\[
H(x_1,x_2)=\begin{cases}x_2+x_1,&\textrm{if }x_1<0\\
x_2-x_1,&\textrm{if }x_1\geq 0.
\end{cases}
\]
}
\end{example}

\section{Quasi--Fuchsian solutions}\label{sec:qf sol}

\subsection{Uniqueness of solution}

We start this section with the simple proof of the fact that the solution to \eqref{eq_dist_Gamma}
is unique in the quasi-Fuchsian case. Actually, it can be shown that this result follows from
Theorem~\ref{th_sov_uniq}, see Remark~\ref{rem: bord infini domaine}.

\begin{proposition}\label{pr_unicity-fuchsian}
Given $\bar \mu\in \mathcal R(\hG)$, the equation \eqref{eq_dist_Gamma} has a unique solution $\bar h_{\bar\mu}$ in the sense of distributions, whose explicit expression is given in \eqref{sol_eq_dist_gamma}.
\end{proposition}

\begin{proof}
Let $\bar T_1,\bar T_2\in\mathcal D'(\hG)$ be two solution of $\eqref{eq_dist_Gamma}$. 
Choose $\bar \eta\in C^\infty(\hG)$ and let $\bar h_{\bar \eta}\in C^\infty(\hG)$ be a solution to 
$\frac{1}{d}\Delta \bar h_{\bar \eta}-\bar h_{\bar \eta}=\bar \eta$, which exists thanks to Theorem \ref{th_smooth}. Then
\[
(\bar T_1,\bar \eta)=(\bar T_1,\frac{1}{d}\Delta \bar h_{\bar \eta}-\bar h_{\bar \eta})=
(\frac{1}{d}\Delta \bar T_1 - \bar T_1,\bar h_{\bar \eta})=(\bar \mu,\bar h_{\bar \eta})=
(\frac{1}{d}\Delta \bar T_2 -\bar T_2,\bar h_{\bar \eta})=
(\bar T_2,\frac{1}{d}\Delta \bar h_{\bar \eta}-\bar h_{\bar \eta})=(\bar T_2,\bar \eta).
\]
Since $\bar \eta$ is arbitrary, this proves that $\bar T_1=\bar T_2$ in the sense of distributions.
\end{proof}

\begin{corollary}\label{cor_unicity-fuchsian}
Let $\tau$ be a cocycle and let $h$ and $h'$ be two $\tau$-equivariant maps such that
$S_1(h)=S_1(h')$. Then $h=h'$.
In particular there exists at most one $\tau$-F-convex set with a given first area measure.
\end{corollary}

\begin{proof}
 By linearity, $S_1(h-h')=0$, but $h-h'$ is $\Gamma$-invariant, so by Proposition~\ref{pr_unicity-fuchsian}, $h=h'$.
 The second part follows by considering support functions for $h$ and $h'$. 
\end{proof}


\subsection{The $\tau$-hedgehog of zero curvature}

\begin{lemma}\label{lem:tau herisson nul}
 For any $\tau\in Z^1(\Gamma,\R^{d+1})$, there exists a unique 
$C^{\infty}$ $\tau$-hedgehog $\lambda_{\tau}$ with $S_1(\lambda_{\tau})=0$.
It is the support function of a convex set if and only if $\tau$ is a coboundary.
\end{lemma}

In the Fuchsian case ($\tau=0$), $\lambda_{\tau}$ is the origin.

\begin{proof}
Let  $h$ be a $\tau$-equivariant map.
By Theorem~\ref{sol_fuchs_distr} there exits a 
$\Gamma$-invariant function $h_0$ such that $S_1(h_0)=S_1(h)$
in the sense of distribution ($h_0$ is a continuous function by the arguments of Subsubsection~\ref{sec_non-uniq}).
Let us define $\lambda_{\tau}=h-h_0$. 
It has the following properties:
\begin{itemize}[nolistsep]
\item \emph{$\lambda_{\tau}$ is unique}: by Corollary~\ref{cor_unicity-fuchsian}. In particular,
it is well-defined in the sense that is depends only on $\tau$.
\item\emph{$S_1(\lambda_{\tau})=0$}: by construction. 
\item \emph{$\lambda_{\tau}$ is $C^{\infty}$}: by the preceding item and elliptic regularity.
 \item \emph{If $\tau$ is a coboundary}, with the notations of  \ref{item cocycle diff} of Lemma~\ref{lem:equivariant},
$S_1(H)=S_1(H_0)$, $H-H_0=\langle \cdot,v\rangle_-$
 and this is the $1$-extension of
$\lambda_{\tau}$.
 \item \emph{If $H-H_0$ is convex,}  
  as $H$ and $H_0$ have the same area measure, by
Subsubsection~\ref{unique-convex}, $H$ and $H_0$ differ by the restriction to $\F$ of a linear form. So $\tau$ is a coboundary.
\item\emph{$\lambda_{\tau}$ is $\tau$-equivariant} by construction.
\end{itemize}
\end{proof}

\begin{remark}[\textbf{Formal eigenfunctions of the hyperbolic Laplacian}]{\rm 
Let us denote  by $E(d)$ the space of formal eigenfunctions of the  Laplacian of $\H^d$
 for the eigenvalue $d$. For any $\tau\in Z^1(\Gamma,\R^{d+1})$,
 $\lambda_{\tau}$ belongs to $E(d)$ (note that its $1$-extension is a formal
 eigenfunction of the wave operator). 
  Actually this correspondence
 is a linear injection.

\begin{lemma}
The map $\lambda:\tau\mapsto\lambda_{\tau}$
from
 $Z^1(\Gamma,\R^{d+1})$ to $E(d)$ is an injective linear map.
 
 The image of $B^1(\Gamma,\R^{d+1})$ is the set of the restrictions to
 $\H^d$ of linear forms of $\R^{d+1}$. 
\end{lemma}

\begin{proof}
We already know that the image of $Z^1$ belongs to $E(d)$.

\emph{$\lambda$ is injective:}
Let $\tau'\in Z^1$. 
If $\lambda(\tau)=\lambda(\tau')$, then there exists a 
$\tau$-equivariant function $h$, a
$\tau'$-equivariant function $h'$ and  $\Gamma$ invariant functions $h_0$ and $h_0'$ with $h'-h_0'=h-h_0$ i.e.~$h'+h_0=h+h_0$. The right hand side is 
a $\tau$-equivariant function and the left hand side is a $\tau'$-equivariant function. The result follows from 
Lemma~\ref{lem:equivariant}.

\emph{$\lambda$ is linear:} 
with the preceding notations and $\alpha$ 
a real number,  from Lemma~\ref{lem:equivariant}, 
$\alpha h+h'$ is 
$(\alpha \tau+\tau')$-equivariant. On one hand, $\alpha (h-h_0)+h'-h_0'$ is equal to
$\alpha \lambda_{\tau}+\lambda_{\tau'}$. On the other hand, $S_1(\alpha h+h')=S_1(\alpha h_0+h_0')$ hence
$\alpha h+h'-\alpha h_0+h_0'$ is equal to $\lambda_{\alpha \tau+\tau'}$.

\emph{$\lambda(B^1)$:} we already know that the image is made of restriction of linear 
forms. The result follows because 
$\lambda$ is linear and $B^1$ has dimension $d+1$.
\end{proof}
}\end{remark}

\begin{remark}[\textbf{Slicing by constant mean radius of curvature}]{\rm 
From Lemma~\ref{lem:tau herisson}, we get two positive constants $c_1$ and $c_2$ such that, for any positive $c$,
$\lambda_{\tau}-c_1-c$ is a slicing of an unbounded part of the $\tau$-F-regular domain $\Omega^+_{\tau}$
by smooth convex Cauchy surfaces with
constant mean radius of curvature. In the same way, $\lambda_{\tau}+c_2+c$ 
is a slicing of an unbounded part of the $\tau$-P-convex domain $\Omega^-_{\tau}$  by smooth convex Cauchy surfaces with
constant mean radius of curvature. Taking negative $c$, the slicing can be extended, going outside of $\Omega^+_{\tau}\cup\Omega^-_{\tau}$, and the slices are
$\tau$-hedgehogs. }\end{remark}

\begin{remark}[\textbf{Quasi-Fuchsian Christoffel problem}]{\rm 
 The uniqueness part of the problem is solved by Corollary~\ref{cor_unicity-fuchsian}.
 Given a $\Gamma$-invariant measure $\mu$, Theorem~\ref{th_smooth} gives the 
(unique) $\Gamma$-invariant solution $h_0$ of $S_1(h_0)=\mu$. So
$h:=h_0+\lambda_{\tau}$ is the unique $\tau$-equivariant solution of $S_1(h)=\mu$, in the sense of distribution.
To know when $h$ is the support function of a $\tau$-F-convex set, one has to
use Proposition~\ref{prop: conv cond}.
 }\end{remark}

\begin{remark}[\textbf{Relations with Codazzi tensors}]{\rm 
A Codazzi tensor (here on a hyperbolic surface $S$) is a self-adjoint $(0,2)$-tensor which satisfies the Codazzi equation. 
For example, for any smooth function $u$ on $S$, $\mathrm{Hess}u-u\mathrm{Id}$ is a Codazzi tensor.
If $S$ is compact, a group isomorphism $\Phi$ between the space of traceless Codazzi tensor and $H^1(\Gamma,\R^3)$ ($\Gamma=\pi_1(S)$) is constructed 
in \cite{BA}: Let $ \tilde{b}$ be the lifting of $b$ to $\H^2$. From a result of
\cite{OS83}, there exists a smooth map $h:\H^2\rightarrow R$ with $\tilde{b}=\mathrm{Hess}h-h\mathrm{Id}$.
It can be checked that $\tilde{b}$ is $\tau$-equivariant, for a $\tau\in Z^1(\Gamma,\R^3)$. Then define $\Phi(b)=\tau$. 
Lemma~\ref{lem:tau herisson nul} says that $\Phi$ is surjective: for any $ \tau$,  
$\mathrm{Hess}\lambda_{\tau}-\lambda_{\tau}\mathrm{Id}$ is a traceless Codazzi operator on $\H^d/\Gamma$.

}\end{remark}

\subsection{Mean width of flat GHCM spacetimes}\label{sub mean width}

Let $h$ be a $\tau$-equivariant map. The map $h-\lambda_{\tau}$ is $\Gamma$-invariant, and
$S_1(h)=S_1(h-\lambda_{\tau})$. With the notations of Subsection~\ref{sub:fuchssol} together with the definition of the action
given in \eqref{eq:distribution action}, $\forall f\in C^{\infty}(\H^d/\Gamma)$, the action of the first area measure on 
$\H^d/\Gamma$ writes as

\begin{equation*}(\overline{S}_1(\overline{h-\lambda_{\tau}}),f)=\int_{\H^d/\Gamma}(\overline{h-\lambda_{\tau}})\left(\frac{1}{d}\Delta-1\right)f.\end{equation*}

Let $h$ be the support function of a  $\tau$-F-convex set $K$. The Radon measure $S_1(K,\cdot)$ is $\Gamma$ invariant, so for any
fundamental domain $\omega$ for the action of $\Gamma$, we can define the \emph{total first area measure} of $K$
by $\overline{S}_1(K):=S_1(K,\omega)$. Actually,  $S_1(K,\cdot)$ gives a Radon measure $\overline{S}_1(K,\cdot)$ on 
$\H^d/\Gamma$, and $\overline{S}_1(K)=\overline{S}_1(K,\H^d/\Gamma)$.
By setting $f=1$ in the above formula, we obtain  (compare with Remark~\ref{rem:meanwidth fuch})
\begin{equation}\label{masse quotient}\overline{S}_1(K)=-\int_{\H^d/\Gamma}\overline{h-\lambda_{\tau}}.\end{equation}

Let us consider  a $\tau$-F-regular domain  $\Omega^+_{\tau}$ with simplicial singularity (see Subsection~\ref{subsub back poly}).
In this case, 
the total mass of the measured geodesic stratification on $\H^d/\Gamma$ is equal to 
$\overline{S}_1(\Omega^+_{\tau})$ (see Remark~\ref{subsub: strat}). 

From the given cocycle $\tau$, one also get a 
 $\tau$-P-regular domain $\Omega_{\tau}^-$ (it is given for example by the symmetry with respect to the origin of the F-convex domain $\Omega_{-\tau}$).
 The meaning of $\overline{S}_1(\Omega_{\tau}^-)$ is clear.
  Let us denote by $h^+_{\tau}$  the support function of $\Omega^+_{\tau}$ and by
$h^-_{\tau}$ the support function of $\Omega^-_{-\tau}$, which is  a $(-\tau)$-equivariant map.
Moreover $-\lambda_{\tau}=\lambda_{-\tau}$, so using \eqref{masse quotient} and the equation above,
$$\overline{S}_1(\Omega_{\tau}^-)+\overline{S}_1(\Omega_{\tau}^+)=-\int_{\H^d/\Gamma} \overline{h_{\tau}^++h_{\tau}^-}.$$
This last formula has the following geometric meaning.
Let $\eta\in\F$, and $-1:\R^d\rightarrow\R^d, x\mapsto -x$. Then 
$h^-_{\tau}\circ -1$ is the support function (defined on $-\F$) of $\Omega_{\tau}^-$. 
So $(h_{\tau}^++h_{\tau}^-)(\eta)$ is the ``distance'' between the support planes of $\Omega_{\tau}^+$ and $\Omega_{\tau}^+$
orthogonal to $\eta$ ($(h_{\tau}^++h_{\tau}^-)(\eta)<0$ says that the respective half-spaces are disjoint). Hence $-\int_{\H^d/\Gamma} \overline{h_{\tau}^++h_{\tau}^-}$ divided by the volume of $\H^d/\Gamma$
can be called the  \emph{mean width} of the  flat spacetime
$(\Omega_{\tau}^+\cup\Omega_{\tau}^-)/\Gamma_{\tau} $. We get that this mean width is determined by the total mass of 
the measured geodesic stratifications defining the spacetime. In the Fuchsian case $\tau=0$, the mean width is null.

\begin{spacing}{0.9}
\begin{footnotesize}
\newcommand{\etalchar}[1]{$^{#1}$}
\def\dbar{\leavevmode\hbox to 0pt{\hskip.2ex \accent"16\hss}d}

\end{footnotesize}
\end{spacing}

\end{document}